\documentclass[reqno]{article}
\usepackage{authblk}
\pdfoutput=1
\usepackage{amsmath}
\usepackage{amsfonts,amsthm,amssymb}
\usepackage{mathabx}
\usepackage{bm}
\usepackage{amscd}
\usepackage[latin2]{inputenc}
\usepackage{t1enc}
\usepackage[mathscr]{eucal}
\usepackage{indentfirst}
\usepackage{graphicx}
\usepackage{graphics}
\usepackage{pict2e}
\usepackage{epic}
\numberwithin{equation}{section}
\usepackage[margin=2.9cm]{geometry}
\usepackage{epstopdf}
\usepackage[colorlinks,linkcolor=blue]{hyperref}
\usepackage[capitalise,noabbrev]{cleveref}
\usepackage{comment}

\usepackage{todonotes}
\usepackage{verbatim}
\crefformat{equation}{(#2#1#3)}
\crefmultiformat{equation}{(#2#1#3)}{ and~(#2#1#3)}{, (#2#1#3)}{ and~(#2#1#3)}
\crefrangeformat{equation}{(#3#1#4) to~(#5#2#6)}
\usepackage[normalem]{ulem}

\allowdisplaybreaks

\theoremstyle{plain}
\newtheorem{Thm}{Theorem}[section]
\newtheorem*{Thm*}{Theorem}
\newtheorem{Lem}[Thm]{Lemma}

\newtheorem{Prop}[Thm]{Proposition}

\theoremstyle{definition}

\newtheorem{Rem}[Thm]{Remark}
\newtheorem{?}[Thm]{Problem}

\newcommand{\Do}{\mathbf{D}_{0}}
\newcommand{\Dn}{\mathbf{D}_{\neq}}

\newcommand{\p}{\partial}
\newcommand{\R}{\mathbb{R}}
\newcommand{\e}{\varepsilon}

\newcommand{\rhoa}{\acute{\rho}}

\newcommand{\phim}{\mathring{\phi}}
\newcommand{\psim}{\mathring{\psi}}

\newcommand{\Torus}{\mathbb{T}}

\newcommand{\dv}{\text{div}}

\newcommand{\rhob}{\bar{\rho}}

\newcommand{\ub}{\bar{u}}

\newcommand{\mut}{\tilde{\mu}}

\newcommand{\T}{\mathcal{T}}

\newcommand{\abs}[1]{\left\lvert#1\right\rvert}

\newcommand{\norm}[1]{\left\lVert#1\right\rVert}

\begin{document}
	
	\begin{titlepage}
		\title{Nonlinear stability of planar shock wave to 3-D compressible Navier-Stokes equations in half space with Navier Boundary conditions
		}
		
		\author[]{Lin Chang$^{1}$}
		\author[]{{Lingjun Liu$^2$}
		}
		\author[]{Lingda Xu$^{3}$}
			
\affil{\begin{center}
				\footnotesize
				 \footnotesize $ ^1 $ School of Mathematics and Physics, Handan University, Handan 056006, P.R.China,\\
			\vspace{0.07cm} 
			 Email: changlin23@buaa.edu.cn (L. Chang).\\
		 \vspace{0.07cm}
				$ ^2 $ College of Mathematics, Faculty of Science, Beijing University of Technology, Beijing 100124, P.R.China,\\
				\vspace{0.07cm} 
				 Email: lingjunliu@bjut.edu.cn (L. Liu).\\
		 \vspace{0.07cm}
		  $ ^3 $ Department of Applied Mathematics, The Hong Kong Polytechnic University, Hong Kong, P.R.China,\\
		\vspace{0.07cm} 
		 Email: lingda.xu@polyu.edu.hk (L. Xu).
		\end{center}}
		\date{}\end{titlepage}
\maketitle
\begin{abstract}
In this paper, we consider the large time behavior of planar shock wave for 3-D compressible isentropic Navier-Stokes equations (CNS) in half space. Providing the strength of the shock wave and initial perturbations are small, we proved the planar shock wave for 3-D CNS is nonlinearly stable in half space with Navier boundary condition.
 The main difficulty comes from the compressibility of shock wave, which leads to lower order terms with bad sign, see the third line in \cref{C17}. We apply a decomposition of the solution into zero and non-zero modes: we take the anti-derivative for the zero mode and obtain the space-time estimates for the energy of perturbation itself. Then combining the fact that the Poincar\'e inequality is available for the non-zero mode, we have successfully controlled the lower order terms with bad sign in \cref{C17}. To overcome the difficulty that comes from the boundary, we introduce the two crucial estimates on boundary \cref{CLem0} and fully utilize the property of Navier boundary conditions, which means that the normal velocity is zero on the boundary and the fluid tangential velocity is proportional to the tangential component of the viscous stress tensor on the boundary. Finally, the nonlinear stability is proved by the weighted energy method.

\end{abstract}



\section{Introduction and Main Result}
In this paper, we study the three-dimensional (3-D) compressible isentropic Navier-Stokes equations,
\begin{align}\label{NS}
\left\{\begin{aligned}
	&\rho_{t}+\dv (\rho{\bf u})=0,\\
	&(\rho{\bf u})_{t}+\dv(\rho{\bf u}\otimes{\bf u}
	)+\nabla p(\rho)=\mu\Delta {\bf u}+(\mu+\lambda)\nabla\dv {\bf u},
\end{aligned}\right.
\end{align}
where $t \geq 0$ is the time variable and $ x
=\left( {x_1}
, x'\right)=\left(x_{1}, x_{2}, x_{3}\right) \in\Omega\subset \mathbb{R}^{3}$ is the spatial variable and the half space domain $\Omega:=\mathbb R^+\times\mathbb T^2=\{x:x_1>0,x'\in\mathbb T^2\}$ with $\mathbb R^+$ being a real half line and $\mathbb T^2:=(\mathbb R/\mathbb Z)^2$ being a two-dimensional unit flat torus. The functions $\rho,  {\bf u}=\left(u_{1}, u_{2}, u_{3}\right)^{t},$ and $p(\rho)=a\rho^\gamma(a>0, \gamma>1)$ 
represent the fluid density, velocity, and pressure, respectively. 
The viscous coefficients $\mu$ and $\lambda$ are constants and satisfy the physical constraints
$$\mu>0,\ \ \ 2\mu+3\lambda\ge0.$$
For the proof of this paper, we only need the weaker restriction $\mu+\lambda\geq 0$.
 We consider the initial-boundary value problem of the system \eqref{NS} with the following initial and boundary conditions
\begin{align}\label{eq1}
 &(\rho,{\bf u})\big|_{t=0}=(\rho_0,{\bf u}_0)(x)\rightarrow(\rho_+,{\bf u}_+)
 ,\ \ \ x_1\rightarrow+\infty,\\[2mm]
 \label{eq2}
 &
 (u_1,u_2,u_3)(0,x')=k(x')\left(0,\p_{1}u_2(0,x'),\p_{1} u_3(0,x')\right), \ \ \ x_1=0,
 \end{align}
where $\rho_+>0$ and ${\bf u}_+=(u_{1+},0,0)^t$ with $u_{1+}$ are constant states. And \cref{eq2} is called the Navier boundary conditions, 
where $\p_1:=\p_{x_1}$, and $k(x')$ is periodic smooth function 
and there exist some positive constants $\underline k, \bar k$ such that
$$\underline k\le k(x')\le\bar k,\ \ \ \ x'=(x_2,x_3)\in\mathbb T^2.$$

It is known that the CNS \cref{NS} has {a} close relationship with the corresponding 3-D compressible isentropic Euler equations, which read
 \begin{align}\label{Euler}
\left\{\begin{aligned}
	&\rho_{t}+\dv (\rho{\bf u})=0,\\
	&(\rho{\bf u})_{t}+\dv(\rho{\bf u}\otimes{\bf u}
	)+\nabla p(\rho)=0.
\end{aligned}\right.
\end{align}
The Euler system is a typical hyperbolic system, one of the main features of which is no matter how small or smooth the initial data is, the shock will formulate. 
\cref{Euler} admits rich wave phenomena such as shock waves, rarefaction waves, and so on. Thus,
we study the Riemann problem which is proposed by Riemann \cite{R} to learn the wave phenomena. That is, we study \eqref{Euler} with the planar Riemann initial data
\begin{align}\label{eq3}
(\rho,{\bf u})(t=0,x)=(\rho_0,{\bf u}_0)(x_1)=\left\{\begin{aligned}
&(\rho_-,{\bf u}_-),\ \ x_1<0,\\
&(\rho_+,{\bf u}_+), \ \ x_1>0.
\end{aligned}\right.
\end{align}
Specifically in this paper, $(\rho_+,{\bf u}_+)$ is given in \eqref{eq1}, ${\bf u}_-=0$ is known by Navier boundary condition \eqref{eq2}, and $\rho_-$ is uniquely determined by the corresponding wave curve, which will be introduced in \cref{Euler2,eq5}.

The nonlinear stability of Navier-Stokes equation waves has been studied extensively. For the 1-d case, we refer to \cite{C,G, MN1} for the stability of shock waves, \cite{MN2} for the stability of rarefaction waves, and \cite{HMX} for the decay rate of contact waves. In particular, due to the application of the anti-derivative technique, the results of \cite{G,HMX,MN1} all require the so-called zero-mass condition. \cite{LZ} used the Green function method to prove the stability of a single shock wave that does not require a zero-mass condition. Moreover, they obtained pointwise estimates of the shock waves and their method can be applied to systems with more general physical viscosity. And \cite{HM} proved the case of the composite wave consisting of two shock waves for CNS. \cite{HXY} also obtained the decay rate of contact waves that do not require zero mass conditions. And the optimal decay rate of contact wave was obtained by \cite{LWX} for  both zero mass and non-zero mass cases.  In addition, the cases of combinations of multiple wave patterns are also very interesting and difficult. \cite{HLM} obtains the stability of the superposition of the contact wave and the rarefaction wave by applying an estimate of the heat kernel. We refer to \cite{KVW1,KVW2} for results about the stability of combinations of multiple wave patterns.

For the m-d Navier-Stokes equations, we refer to \cite{HXY,LWWarma} for the stability of rarefaction wave. For the stability of the shock wave, we refer to \cite{WW,Yuan}. Recently, an interesting result on the stability of the vortex sheet is proved in \cite{HXXY}.
For the nonlinear stability of the wave patterns to the CNS in the half space, we refer to \cite{M} and the reference therein and thereafter. For the inflow problem, we refer to \cite{HMS,MN3,QW}, for the outflow problem, we refer to \cite{HQ,KNZ}, and for the impermeable wall case we refer to \cite{MM1}. For the stability of the rarefaction wave for the m-d CNS in half space. we refer to \cite{Wang,WW2021CPAA}.




Before stating our main result, we recall the planar shock wave to \eqref{Euler}-\eqref{eq3}. Let $(\rho^s, u_1^s)=(\rho^s, u_1^s)(t, x_1)$ be a weak entropy solution of the following 1-D compressible isentropic Euler equations
 \begin{align}\label{Euler1}
\left\{\begin{aligned}
	&\rho^s_{t}+\p_1 (\rho^s  u_1^s)=0,\\
	&(\rho^s  u_1^s)_{t}+\p_1(\rho^s( u_1^s)^2
	)+\p_1 p(\rho^s)=0,\ \ \ \ x_1\in\mathbb R,\ \ t>0,
\end{aligned}\right.
\end{align}
with the initial data 
\begin{align}\label{eq4}
(\rho^s,{ u_1^s})(t=0,x_1)=(\rho^s_0, u^s_{10})(x_1)
=\left\{\begin{aligned}
&(\rho_-, u_{1-}),\ \ x_1<0,\\
&(\rho_+,{ u}_{1+}), \ \ x_1>0,
\end{aligned}\right.
\end{align}
where $(\rho_\pm, u_{1\pm})$ 
satisfying the Rankine-Hugoniot conditions,
 \begin{align}\label{Euler2}
\left\{\begin{aligned}
	&-s(\rho_+-\rho_-)+m_{1+}-m_{1-}=0,\\
	&-s(m_{1+}-m_{1-})+\frac{(m_{1+})^2}{\rho_+}-\frac{(m_{1-})^2}{\rho_-}+p(\rho_+)-p(\rho_-)=0,
\end{aligned}\right.
\end{align}
where $m_{1\pm}=\rho_\pm u_{1\pm},$ and $s$ is the speed of shock.
The system \eqref{Euler1} has two eigenvalues $$\lambda_1(\rho,u_{1})=u_1-\sqrt{p'(\rho)}, \ \ \ \lambda_2(\rho,u_{1})=u_1+\sqrt{p'(\rho)}.$$
We consider the 2-shock $(\rho^s,u_1^s)$  satisfying the Lax's entropy condition,
\begin{align}\label{eq5}
\lambda_2(\rho_+,u_{1+})<s<\lambda_2(\rho_-,u_{1-}) \ \ \text{and}\ \ s<\lambda_1(\rho_-,u_{1-}).
\end{align}

In this paper, we consider the viscosity effect, the Lax shock wave is smoothed out to be a viscous version, $(\bar\rho,\bar u_1)=(\bar\rho,\bar u_1)(x_1-st+\alpha)$ is a traveling wave solution of the following 1-D compressible isentropic Navier-Stokes equations
 \begin{align}\label{NS1}
\left\{\begin{aligned}
	&-s(\bar\rho)'+(\bar\rho\bar u_1)'=0,\\
	&-s(\bar\rho\bar u_{1})'+\left(\bar\rho(\bar u_1)^2+p(\bar\rho)\right)'=\tilde\mu(\bar u_1)'',\\
	&(\bar\rho,\bar u_1)(x_1)\rightarrow(\rho_\pm,u_{1\pm}),\ \ \ \text{as} \ \ x_1\rightarrow\pm\infty,
\end{aligned}\right.
\end{align}
where $\alpha $ is the shift, $\tilde\mu:=2\mu+\lambda$ and $':=\frac{d}{dx_1}.$
The planar viscous shock
\begin{align}\label{eq6}
(\bar\rho,\bar{\bf u})(x_1-st+\alpha)=(\bar\rho,\bar u_1,0,0)(x_1-st+\alpha)
\end{align}
is a traveling wave solution of \eqref{NS} and propagating along the $x_1$-axis with the shock speed $s$.


Let the perturbation be defined as
\begin{align}\label{eq18}
\begin{aligned}
&\qquad\phi(t,x):=\rho(t,x)-\bar{\rho}(x_1-st+\alpha),\\
&{\psi}(t, x)=({\psi}_1,\psi_2,\psi_3)^t:={\bf m} (t, x)-\left({\bar{m}_1}(x_1-st+\alpha),0,0\right)^t,
\end{aligned}
\end{align}
where ${\bf m}=(m_1,m_2,m_3)^t=\rho{\bf u}, {\bf \bar m}=(\bar m_1,\bar m_2,\bar m_3)^t=\bar\rho{\bf \bar u}$. Inspired by \cite{Yuan}, we decompose the perturbation $(\phi,\psi)$ into principal and transversal parts. 
 Denote one-dimensional zero {modes} of perturbation {as}
\begin{align}\label{eq20}
 \phim=\int_{\mathbb T^2}\phi dx_2dx_3,\ \ \ \psim=\int_{\mathbb T^2}\psi dx_2dx_3,
 \end{align}
and multi-dimensional non-zero modes as
 \begin{align}
 \acute\phi=\phi-\int_{\mathbb T^2}\phi dx_2dx_3,\ \ \ \acute\psi=\psi-\int_{\mathbb T^2}\psi dx_2dx_3.
 \end{align}
 Set the anti-derivative variables as follows,
 \begin{align}\label{eq19}
(\Phi,\Psi)(t,x_1)=-\int_{x_1}^{+\infty}(\phim,\psim_1)(t,y_1)dy_1,\ \  x_1\in\mathbb R^+,t\ge0.
\end{align}

Now we 
state the main result as follows.
\begin{Thm}\label{thm1}
Let $(\bar\rho,\bar{\bf u})(x_1-st+\alpha)$ be the planar shock wave defined in \eqref{eq6}. Then there exist constants $\delta_0>0$ and $\e_0>0$ such that if $\delta:=|\rho_+-\rho_-|\le\delta_0$ and
\begin{align}\label{eq7}
\e:=\|\left(\rho_0-\bar\rho(x_1+\alpha),{\bf u}_0-\bar{\bf u}(x_1+\alpha)\right)\|_{H^2(\Omega)}+\|(\Phi_0,\Psi_0)\|_{L^2(\Omega)}\le\e_0,
\end{align}
then 3-D initial and boundary value problem \eqref{NS}-\eqref{eq2} admits a unique global solution $(\rho,{\bf u})$ satisfying
\begin{align}\label{eq8}
\begin{aligned}
&{\left(\rho-\bar\rho,{\bf u}-\bar{\bf u}\right)\in C\left(0,+\infty;L^2(\Omega)\right), \ \ \ \nabla(\rho,{\bf u})\in C\left(0,+\infty);H^1(\Omega)\right),}\\
&\qquad\quad{\nabla^2\rho\in L^2\left(0,+\infty;L^2(\Omega)\right),\ \ \ \nabla^2{\bf u}\in L^2\left(0,+\infty;H^1(\Omega)\right)},
\end{aligned}
\end{align}
and the time-asymptotic stability toward the planar shock wave $(\bar\rho,\bar{\bf u})(x_1-st+\alpha)$ holds
\begin{align}\label{eq9}
\lim_{t\rightarrow+\infty}\sup_{x\in\Omega}\left|(\rho,{\bf u})(t,x)-(\bar\rho,\bar{\bf u})(x_1-st+\alpha)\right|=0.
\end{align}
\end{Thm}
\begin{Rem}\label{rem1}
To the best of our knowledge, \cref{thm1} is the first result considering the stability of shock wave for m-d CNS in half space. The case of inflow and outflow {problems} will be studied in the forthcoming paper.
\end{Rem}

The rest of this paper is arranged as follows. We will present the planar shock wave's properties and some useful lemmas in section \ref{sc1}. In section \ref{sc2}, {the detailed perturbation systems are given by decomposition techniques and anti-derivative methods. Then a priori estimates are obtained by 
 energy method in section \ref{sc3}.} The proof of Theorem \ref{thm1} is completed by energy methods in section \ref{sc4}. 

Note that denote $\|\cdot\|:=\|\cdot\|_{L^2(\Omega)},\ \|\cdot\|_{L^\infty}:=\|\cdot\|_{L^\infty(\Omega)},\ \|\cdot\|_{L^6}:=\|\cdot\|_{L^6(\Omega)},\ \|\cdot\|_{L^3}:=\|\cdot\|_{L^3(\Omega)}, \ \|\cdot\|_{H^l}:=\|\cdot\|_{H^l(\Omega)}$ for $l\ge1.$ For simplicity, throughout the paper we write C as some generic positive constants that are independent of time $t$ or $\tau$ and the wave {strength $\delta$}.

\section{Preliminaries}\label{sc1}
In this section, we present some important and useful lemmas to prepare for the following sections. 
First, we list the following properties of the viscous shock profile $(\bar\rho,\bar u_1).$
\begin{Lem}\label{lem1}\cite{MN1}
Assume that \eqref{Euler2}-\eqref{eq5} hold.
The viscous shock $(\bar\rho,\bar u_1)(x_1-st+\alpha)$ is unique and smooth solution of \eqref{NS1}, and satisfies the following properties,

\
i) $(\bar\rho)'(x_1)<0$ and $(\bar u_1)'(x_1)<0$ for all $x_1\in\mathbb R$;

\
ii) $\delta^2e^{-c_1\delta|x_1|}\le|(\bar u_1)'(x_1)|\le\delta^2e^{-c_2\delta|x_1|}$ for all $x_1\in\mathbb R$, 

\ \ \ \ where $c_1>c_2$ are two positive constants, independent of $\delta$ and $x_1$;

\
iii) $|(\bar u_1)''(x_1)|\le \delta|(\bar u_1)'(x_1)|$ for all $x_1\in\mathbb R.$

\end{Lem}

\vspace{0.3cm}
Since ${\bf u}_-=0 $ in this paper, combining with Lemma \ref{lem1} and \eqref{NS1}, one gets that
\begin{equation} \label{C7}
|\bar u_1|\le   C \delta.  
\end{equation}

The following Gagliardo-Nirenberg type inequality plays an important role in our proof of the main results.

\begin{Lem}
\label{Lem-GN}\cite{WW2021CPAA}
Assume that $g(x)\in H^2(\Omega)$, then there exists some generic constant $C$ such that it holds
\begin{align}\label{eq13}
\|g\|_{L^\infty(\Omega)}\le \sqrt{2}\|g\|^{\frac{1}{2}}_{L^2(\Omega)}\|\nabla g\|^{\frac{1}{2}}_{L^2(\Omega)}+C\|\nabla g\|^{\frac{1}{2}}_{L^2(\Omega)}\|\nabla^2 g\|^{\frac{1}{2}}_{L^2(\Omega)}.
\end{align}

	\end{Lem}

Here we introduce the following decomposition ideal. Assume for any $f(x)\in L^\infty(\Omega)$ that is periodic in $x'\in\Torus^2$,  we set $\int_{\Torus^2}1dx'=1$ without loss of generality.
 We decompose the function $f$ into principal and transversal parts. 
 Then we can define the decompositions $\mathbf{D}_0$ and $\mathbf{D}_{\neq}$ as follows,
\begin{align}\label{def-decom}
	\Do f:=\mathring{f}:=\int_{{\Torus^2}}f dx',\qquad \Dn f:=\acute{f}:=f-\mathring{f},
\end{align}
where $f$ is integrable on $\Torus^2$. By simple analysis, the following proposition of $\Do$ and $\Dn$ holds for any function $f$ on $\Torus^2$. 
\begin{Prop}\label{prop-decom}\cite{HLWX}
	For the projections $\Do$ and $\Dn$ defined in \cref{def-decom}, 
	it holds,
	
	i) $\Do\Dn f=\Dn\Do f=0$.
	
	ii) For any non-linear function $F$, one has
	\begin{align}
		\Do F(U)-F(\Do U)=O(1)F''(\Do U)\Do\big((\Dn U)^2\big).
	\end{align}
	
	iii) $\|f\|^2_{L^2(\Omega)}=\|\Do f\|^2_{L^2(\R^+)}+\|\Dn f\|^2_{L^2(\Omega)}.$
\end{Prop}
The proof of \cref{prop-decom} is basic and we omit it.

 \section{Reformulated Problem}\label{sc2}
In order to prove Theorem \ref{thm1}, we study the perturbation of the global solution $(\rho,{\bf m})$ to the problem \eqref{NS}-\eqref{eq2} around the viscous shock wave $(\bar\rho,\bar {\bf m})$ by 
 \eqref{eq18}.
 {Subtracting} \eqref{NS1}$_1$ from \eqref{NS}$_1$ and 
integrating the resulting equation 
 with respect to $x'$ in $\mathbb T^2$, we have
\begin{align}\label{eq15}
\phim_t+\psim_{x_1}=0. 
\end{align}
Then considering Navier boundary conditions \eqref{eq2}, and integrating \eqref{eq15} with respect to $x_1$ in $(0,+\infty)$, one has  
\begin{align}\label{eq16}
\int_0^{+\infty}\phim_tdx_1=\int_{\mathbb T^2}{\bf m}(t,0,x_2,x_3)-\bar{\bf m}(-st+\alpha)dx_2dx_3=-\bar {\bf m}(-st+\alpha).
\end{align}
Integrate \eqref{eq16} with respect to time variable $t$ in $[0,t]$ and let 
 $$I(\alpha):=\int_0^{+\infty}\phim(0,x_1)dx_1-\int_0^{+\infty}\bar{\bf m}(-s\tau+\alpha)d\tau.$$
 By choosing some suitable constant $\alpha$ such that $I(\alpha)=0$,
 we have
 \begin{align}\label{eq17}
 \lim_{t\rightarrow+\infty}\int_0^{+\infty}\phim(t,x_1)dx_1=0.
 \end{align}
 Set
 \begin{align}\label{eqj118}
 \begin{aligned}
 &\Phi(0,0)=-\int_0^{+\infty}\phim(x_1,0)dx_1=-\int_0^{+\infty}\bar{\bf m}(-s\tau+\alpha)d\tau,\\
& A(t)=-\int_t^{+\infty}\bar{\bf m}(-s\tau+\alpha)d\tau,\ \ \ A'(t)=\bar{\bf m}(-s\tau+\alpha),
 \end{aligned}
 \end{align}
 then it holds that
 \begin{align}\label{C50}
 \begin{aligned}
 \Phi(t,0)&=\Phi(0,0)+\int_0^t\Phi_\tau(\tau,0)d\tau=\Phi(0,0)-\int_0^t\Psi_{x_1}(\tau,0)d\tau=\Phi(0,0)+\int_0^{t}\bar{\bf m}(-s\tau+\alpha)d\tau\\
 &=\Phi(0,0)+\int_0^tA'(\tau)d\tau=\Phi(0,0)+A(t)-A(0),
 \end{aligned}
\end{align}
that is $\Phi(t,0)=A(t).$

Subtracting \eqref{NS1} from \eqref{NS}, one has the following perturbation system for $(\phi,\psi)$ given by \eqref{eq18},
\begin{align}\label{eq21}
\left\{\begin{aligned}
&\phi_t+\dv\psi=0,\\
&\psi_t+\dv\left(\frac{{\bf m}\otimes{\bf m}}{\rho}-\frac{\bar{\bf m}\otimes\bar{\bf m}}{\bar\rho}\right)+\nabla\left(p(\rho)-p(\bar\rho)\right)\\
&\qquad=\mu\Delta\left(\frac{{\bf m}}{\rho}-\frac{\bar{\bf m}}{\bar\rho}\right)+(\mu+\lambda)\nabla\dv\left(\frac{{\bf m}}{\rho}-\frac{\bar{\bf m}}{\bar\rho}\right),
\end{aligned}\right.
\end{align}
with the initial data
\begin{align}\label{eq22}
(\phi,\psi)(0,x)=(\phi_0,\psi_0)(x)\in H^3(\Omega),
\end{align}
and the boundary conditions
\begin{align}\label{eq23}
\left\{\begin{aligned}
&\psi_1(t,0,x')=-A(t),\\
&\psi_2(t,0,x')=\rho(t,0,x')\times k(x')\p_1\psi_2(t,0,x'),\\
&\psi_3(t,0,x')=\rho(t,0,x')\times k(x')\p_1\psi_3(t,0,x'),
\end{aligned}\right.
\end{align}
because of $\psi_i=\rho u_i, i=2,3.$

Integrating \eqref{eq21} with respect to $x'$ in $\mathbb T^2$ and taking {the} anti-derivative variable into the resulting equation, one has
\begin{equation}\label{equ-anti}
	\begin{cases}
		\p_t \Phi + \p_1 \Psi
		 =0,\\ 
		\p_t \Psi
		+ 2\bar u_1 \p_1 \Psi
		+ 
		\bar w\p_1 \Phi - \mut \p_1 \big[ \frac{1}{\bar\rho} \big( \p_1 \Psi
		 - \bar u_1 \p_1 \Phi \big) \big] = 
		 \Do\mathcal{N}_{11}+\p_1\left(\Do\mathcal{N}_{21}\right),
	\end{cases}
\end{equation}
where 
\begin{equation}\label{alpha}
	{{\bar{w}}  = p'(\bar\rho) - \abs{\bar u_1}^2,}
\end{equation}
\begin{equation}\label{q}
	\begin{aligned}
		&\mathcal{N}_{11}:=
		- \Big(\frac{m_1^2}{\rho} - \frac{\bar m_1^2}{\bar\rho} - 2\bar u_1 \psi_1 + \bar u_1^2 \phi\Big) - \big(p(\rho)-p(\bar\rho) - p'(\bar\rho) \phi \big)=O(1)\left(|\phi|^2+|\psi_1|^2\right)
		 ,\\
	&\mathcal{N}_{21} := \mut
	\Big(\frac{m_1}{\rho} - \frac{\bar m_1}{\bar \rho} - \frac{\psi_1}{\bar \rho} + \frac{\bar u_1}{\bar \rho} \phi\Big)=O(1)\left(|\phi|^2+|\psi_1|^2\right),
	\end{aligned}
\end{equation}
with the initial data
\begin{align}\label{eq24}
(\Phi,\Psi)(0,x_1)=-\int_{x_1}^{+\infty}(\phim_0,\psim_{10})(y_1)dy_1,
\end{align}
and the boundary conditions
\begin{align}\label{eq25}
\Phi(0,0)=A(0),\ \ \ \p_1\Psi\Big|_{x_1=0}=A'(t),\ \ \ t\ge0.
\end{align}

Define the perturbation of velocity as $\zeta(t,x)=(\zeta_1,\zeta_2,\zeta_3)^t:={\bf u}(t,x)-\bar{\bf u}(x_1-st+\alpha)$. The perturbation system of $(\phi,\zeta)$ can be written as
\begin{equation}\label{eq26}
	\begin{cases}
		\phi_t+\rho\dv\zeta+{\bf u}\cdot\nabla\phi+\dv{\bar{\bf u}}\phi+\nabla\bar\rho\cdot\zeta=0,\\
		\rho\zeta_t+\rho{\bf u}\cdot\nabla\zeta+\nabla\left(p(\rho)-p(\bar\rho)\right)+\rho\zeta\cdot\nabla\bar{\bf u}+\phi(\bar{\bf u}_t+\bar{\bf u}\cdot\nabla\bar{\bf u})-\mu\Delta\zeta-(\mu+\lambda)\nabla\dv\zeta=0,
	\end{cases}
\end{equation}
with the initial data
\begin{align}\label{eq27}
(\phi,\zeta)(0,x)=(\phi_0,\zeta_0)(x)\in H^3(\Omega),\ \ \ \zeta_0(x)
=\frac{1}{\rho_0(x)}\left(\psi_0(x)-\bar{\bf u}_0(x_1+\alpha)\phi_0(x)\right),
\end{align}
and the boundary conditions
\begin{align}\label{eq28}
\left\{\begin{aligned}
&\zeta_1(t,0,x')=\frac{-A(t)}{\rho(t,0,x')},\\
&\zeta_2(t,0,x')= k(x')\p_1\zeta_2(t,0,x'),\\
&\zeta_3(t,0,x')= k(x')\p_1\zeta_3(t,0,x').
\end{aligned}\right.
\end{align}

For any $0 \leq T \leq+\infty$, the solution $(\phi, {\zeta})(t, x)$ of system \eqref{eq26} is sought in the set of functional space $X(0,+\infty)$ defined by
\begin{equation}\begin{aligned}
X(0,T):=\Big\{(\phi,\zeta)&:(\phi,\zeta){\mathrm{~is~periodic~in~}}x'=(x_{2},x_{3})\in\mathbb{T}^{2},(\phi,\zeta)\in C(0,T;H^{2}(\Omega)), \\
&\zeta_{1}\in L^{2}\big(0,T;L^{2}(\Omega)\big),(\phi,\nabla\zeta)\in L^{2}\big(0,T;H^{2}(\Omega)\big)\Big\}.
\end{aligned}\end{equation}

\section{ A priori estimates}\label{sc3}


Note that if constant $\chi$ is suitably small, the condition
\begin{equation}\label{C4}
E(t):=\sup_{t\in(0,T)}\big(\:\|(\Phi,\Psi)\|_{L^2(\mathbb{R}^+)}+\|(\phi,\zeta)\|_{H^2(\Omega)}\:\big)\leq \chi,
\end{equation}
and Sobolev embedding theorem implies that
$$|(\phi, \zeta)| \leq \frac{1}{2} \rho_{-}\ \  \text{ and }\ \  |{\bf u}|=\left|\left(u_1, u_2, u_3\right)\right| \leq C,$$
where $C$ is a positive constant which only depends on $\rho_{+}, {\bf u}_{ +}$. Therefore, we have the lower and upper bounds of the density function
\begin{equation}\label{C39}
 0<\frac{1}{2} \rho_{-} \leq \rho(t, x) \leq \frac{1}{2} \rho_{-}+\rho_{+}.
\end{equation}

Since the proof for the local-in-time existence and uniqueness of the classical solution to \eqref{NS}-\eqref{eq2} is standard, the details will be omitted. To prove Theorem \ref{thm1}, it suffices to show the following a priori estimates.

\begin{Prop}[A priori estimates]\label{Cproposition}
    Under   the assumptions of Theorem \ref{thm1}, for any  fixed $T>0,$ assume that $(\phi,\zeta)\in X(0,T)$ solves the problem (\ref{eq26}), and the anti-derivative  variables,

\begin{equation}
(\Phi,\Psi)(x_1,t)=-\int^{+\infty}_{x_1}\int_{\mathbb{T}^2}(\phi,\psi_1)(y_1,x_2,x_3,t)dx_2 dx_3 dy_1,\quad x_1\in\mathbb{R}^+,\ t>0,
\end{equation}
exists  and  belongs  to  the  $C(0,T;H^3(\mathbb{R}^+))$    space, where $\psi_1=m_1\boldsymbol{-}\bar{m}_1=(\bar{\rho}+\phi)(\bar{u}_1+\zeta_1)\boldsymbol{-}\bar{\rho}\bar{u}_1.$ Then there exist  $\delta_0>0 $  and $ {\chi>0}$, such that, if $\delta\leqslant\delta_{0}$, and

\begin{equation}
\sup_{t\in(0,T)}\left(\:\left\|(\Phi,\Psi)\right\|_{L^{2}(\mathbb{R}^+)}+\left\|(\phi,\zeta)\right\|_{H^{2}(\Omega)}\right)\leqslant{ }\chi,
\end{equation}

 then it holds that
 \begin{equation}\label{eqj90}
 \begin{aligned}
\sup_{t\in(0,T)} &\left(\parallel(\Phi,\Psi)\parallel_{L^2(\mathbb{R}^+)}^2+\parallel(\phi,\zeta)\parallel_{H^2(\Omega)}^2+\left\|\zeta^{\prime}\big|_{x_1=0}(\cdot,t)\right\|_{L^2(\mathbb{T}^2)}^2+\left\|\partial_{x^{\prime}}\zeta^{\prime}\Big|_{x_1=0}(\cdot,t)\right\|_{L^2(\mathbb{T}^2)}^2\right)  \\
&+\int_{0}^{T}\Big(\left\|\left|(\bar{u}_{1}{})'\right|^{\frac{1}{2}}\Psi\right\|_{L^{2}(\mathbb{R}^+)}^{2}+\left\|\zeta_{1}\right\|_{L^{2}(\Omega)}^{2}+\left\|(\phi,\nabla\zeta)\right\|_{H^{2}(\Omega)}^{2} +\|\partial_{t}(\phi,\zeta)\|_{H^{1}(\Omega)}^{2} \\
& + 
{ \left(|\bar{u}_{1}| \Psi^{2}+|\p_1\bar u_1|\Phi_{x_1}^{2} \right)}\Big|_{x_1=0} \Big)dt+\int_{0}^{T}\left(\left\|\left.\partial_{t}\zeta^{\prime}\right|_{x_{1}=0}\right\|_{L^{2}(\mathbb{T}^{2})}^{2}+\left\|\left.\partial_{x^{\prime}}^{\iota}\zeta^{\prime}\right|_{x_{1}=0}\right\|_{L^{2}(\mathbb{T}^{2})}^{2}\right)dt \\
\leq& C \left\|(\Phi_{0},\Psi_{0})\right\|_{L^{2}(\mathbb{R}^+)}^{2}+ C\left\|(\phi_{0},\zeta_{0})\right\|_{H^{2}(\Omega)}^{2}+ C\delta,
\end{aligned}
\end{equation}
where the sequel $\zeta^{\prime}=\left(\zeta_2, \zeta_3\right), \partial_{x^{\prime}}=\partial_{x_2}$ or $\partial_{x_3}$ and  $|\iota|=0,1,2.$ 
\end{Prop}

From now on, we always assume that $\chi+\delta \ll1$. For convenience, we define $M(t) \geq 0$ by
$$
M^2(t)=(\chi+\delta)\left[E^2(t)+\int_0^t\left(\left\| \left(\phi, \zeta_1\right)\right\|^2+\left\|\partial_\tau(\phi, \zeta)\right\|_{H^1}^2+\|\nabla \phi\|_{H^1}^2+\|\nabla \zeta\|_{H^2}^2\right) d \tau\right].
$$

 Motivated by \cite{MM1}, we give the boundary estimates first.
\begin{Lem}\label{CLem0}
For $0\leq t \leq T$, the following inequalities hold:
\begin{equation}\label{C60}
\begin{aligned}
\left|\int_{0}^{t}(\Phi\Psi)|_{x_{1}=0}d\tau\right| \leqq C \delta,\quad\left|\int_{0}^{t}(\Psi\Psi_{x_{1}})|_{x_{1}=0}d\tau\right|\leqq C \delta.
\end{aligned}\end{equation}
\end{Lem}
\begin{proof}
 Since the proofs of the two inequalities in \eqref{C60} are similar, for simplicity, we only prove the first {one,}
\begin{equation}\begin{aligned}
&\left|\int_{0}^{t}(\Phi\Psi)|_{x_{1}=0}d\tau \right|
\leq  \int_{0}^{t} \|  \Psi \|_{L^{\infty}{(\mathbb R^+)}}      |A(\tau)|d\tau =    \int_{0}^{t}    \left\|     \|  \Psi \|_{L^{\infty}{(\mathbb R^+)}}      \right\|_{L^2({\mathbb T^2})}   |A(\tau)|d\tau \\
&\quad\leq  C (\frac{1}{2} \rho_{-}+\rho_{+})\int_{0}^{t}    \left\|     \|  \Psi \|_{L^{2}{(\mathbb R^+)}}^{\frac{1}{2}}     \|  \Psi_{x_1} \|_{L^{2}{(\mathbb R^+)}}^{\frac{1}{2}}     \right\|_{L^2({\mathbb T^2})}   \left|\frac{A(\tau)}{\rho(\tau,0,x')}\right|d\tau \\
&\quad\leq  C \int_{0}^{t}        \|  \Psi \|_{L^{2}{(\Omega)}}^{\frac{1}{2}}     \|  \psim   \|_{L^{2}{(\Omega)}}^{\frac{1}{2}}        |\bar{\mathbf{u}}(-s\tau+\alpha)|d\tau \quad\leq   C  \sup_{t}\left(     \|  \Psi \|_{L^{2}{(\Omega)}}^{\frac{1}{2}}     \|  \psim   \|_{L^{2}{(\Omega)}}^{\frac{1}{2}}  \right)    \int_{0}^{t}     |\bar{\mathbf{u}}(-s\tau+\alpha)|d\tau \\
&\quad\leq   C  \delta   \sup_{t} \left(      \|  \Psi \|_{L^{2}{(\Omega)}}^{\frac{1}{2}}     \|  \psi    \|_{L^{2}{(\Omega)}}^{\frac{1}{2}}  +     \|  \Psi \|_{L^{2}{(\Omega)}}^{\frac{1}{2}}     \|  \acute\psi   \|_{L^{2}{(\Omega)}}^{\frac{1}{2}}  \right) \\
&\quad\leq   C  \delta   \sup_{t} \left(      \|  \Psi \|_{L^{2}{(\Omega)}}^{\frac{1}{2}}     \|  \psi    \|_{L^{2}{(\Omega)}}^{\frac{1}{2}}  +     \|  \Psi \|_{L^{2}{(\Omega)}}^{\frac{1}{2}}     \|  \nabla\psi   \|_{L^{2}{(\Omega)}}^{\frac{1}{2}}  \right)\leq   C  \delta .
 \end{aligned}\end{equation}
 \end{proof}

Then, we have the following $L^2$ estimate.
\begin{Lem}\label{CLem1}
For $T>0$ and $(\phi, \psi) \in X(0, T)$, under the assumptions of Proposition \ref{Cproposition}  with suitably small $\chi+\delta$, we have for $t \in[0, T]$,
\begin{align}\label{C11}
\begin{aligned}
&\sup_{t\in(0,T)}\left\|(\Phi,\Psi)\right\|_{L^{2}(\mathbb{R}^+)}^{2}+\int_{0}^{T}\left(\left\||(\bar{u}_{1}{})'|^{\frac{1}{2}}\Psi\right\|_{L^{2}(\mathbb{R}^+)}^{2}+\left\|\partial_{1}\Psi\right\|_{L^{2}(\mathbb{R}^+)}^{2}  +   |\bar{u}_{1}| \Psi^{2} \Big|_{x_1=0}  \right)   \mathrm{d}t\\
\leq &C\left\|(\Phi_0,\Psi_0)\right\|_{L^2(\mathbb{R}^+)}^2+ C(\delta+\chi)\int_0^T\left\|\partial_{1}\Phi\right\|_{L^2(\mathbb{R}^+)}^2 \mathrm{d}t+   C   \delta\int_{0}^{T}    \|\phi_{ x_1}\| ^2   \mathrm{d}t\\
&+C\chi\int_0^T\left\|\nabla(\phi,\zeta_1)\right\|_{L^2(\Omega)}^2 \mathrm{d}t +C     \int_{0}^{T}    |\bar{u}_{1}| \partial_{1}\Phi^2 \Big|_{x_1=0}  \mathrm{d}t + C\delta.
\end{aligned}
\end{align}

\end{Lem}
 \begin{proof}

 Multiplying $\Phi$ on $( {\ref{equ-anti}})_1\ \mathrm{and}\ \frac\Psi{\bar{w}}$ on $( {\ref{equ-anti}})_2$, respectively, and summing the resulting equations up, one can get that
\begin{align}\label{C1}
\begin{aligned}
&\left(\frac{\Phi^2}{2}+\frac{\Psi^2}{2 {\bar{w}}}\right)_{t}+ \beta \Psi^2 +\frac{\tilde{\mu}}{{\bar{w}} \bar{\rho}}\Psi^{2}_{x_1}
=(\cdots)_{x_1}  +\tilde{\mu}\frac{{\bar{w}}_{x_1}}{{\bar{w}}^2 \bar{\rho}}\Psi\Psi_{x_1}+  \tilde{\mu}\frac{{\bar{w}}_{x_1}}{{\bar{w}}  \bar{\rho}}\bar{u}_1 \Psi_{x_1}\Phi_{x_1}       +    \frac{   \tilde{\mu} \bar{  u }_{1}     }{{\bar{w}}  \bar{\rho}} \Psi_{x_1} \Phi_{x_1} \\
 &\qquad-\frac{\tilde{\mu} \Psi}{{\bar{w}}}\left[  \left(\frac{1}{\bar{\rho}}\right)_{x_{1}} \bar{  u }_{1}  \Phi_{x_1} + \left(\frac{1}{\bar{\rho}}\right) \bar{  u }_{1}'  \Phi_{x_1} + \left(\frac{1}{\bar{\rho}}\right) \bar{  u }_{1}  \Phi_{x_1 x_1}  \right] +  \frac{\Psi}{{\bar{w}}}\left [\Do\mathcal{N}_{11}+ \left(\Do\mathcal{N}_{21}\right) _{x_1}   \right ],
 \end{aligned}
 \end{align}
where
$$\beta=-\left(\frac{1}{2{\bar{w}}}\right)_t-\left(\frac{\bar{u}}{{\bar{w}}}\right)_{x_1},$$
and
$$\left(\cdots\right)=-\Phi\Psi-\frac{\bar{u}_{1}{}}{{\bar{w}}}\Psi^{2}+\frac{\tilde{\mu}}{{\bar{w}}\bar{\rho}{}}\Psi(\Psi_{x_1}-\bar{u}_{1}\Phi_{x_1}) .$$
Moreover, if $\delta> 0$ is small, by (\ref{C39}),  one gets that
\begin{equation}\label{C3}
{\bar{w}}=p'(\bar{\rho}^s)-|\bar{u}_1^s|^2\geqslant\frac12p'(\bar{\rho}^s)>0,\ \quad\forall\  x\in\Omega,\ t\geqslant0.
\end{equation}
By$(\ref{NS1})_1$, (\ref{C3}) and  Lemma \ref{lem1}, one  can  verify  that
\begin{equation}
\begin{aligned}
\beta =\frac1{2{\bar{w}}^2}\big(\partial_t{\bar{w}}+2\bar{u}_1{}\partial_{1}{\bar{w}}-2(\bar{u}_1{})'{\bar{w}}\big) =-\frac{1}{2{\bar{w}}^{2}}\Big(\frac{s\bar{\rho}{}}{s-\bar{u}_{1}{}}p^{\prime\prime}(\bar{\rho}{})+2p^{\prime}(\bar{\rho}{})\Big)(\bar{u}_{1}{})^{\prime}+O(1)\:|\bar{u}_{1}{}|\:|(\bar{u}_{1}{})^{\prime}|\:.
\end{aligned}
\end{equation}
Thus, if $\delta> 0$ is small, by  (\ref{C7})  and the fact that $( \bar{u} _1{}) ^\prime< 0, $ one has that
\begin{equation}\label{C9}
\beta\geq  c_0 |(\bar{u}_1{})^{\prime}|, \quad \quad       \forall\   x\in\Omega,\ t\geqslant0.
\end{equation}
Integrating (\ref{C1}) with respect to $t$ and $x_1$ over $ [0,t]\times \mathbb{R}^{+}$, we have
\begin{equation}\label{C2}
\begin{aligned}
&\int_{\mathbb{R}^{+}} \left(\frac{\Phi^2}{2}+\frac{\Psi^2}{2 {\bar{w}}}\right)  \mathrm{d}x_1 +\int_{0}^{t} \int_{\mathbb{R}^{+}} \beta \Psi^2 \mathrm{d}x_1\mathrm{d}\tau+\int_{0}^{t}\int_{\mathbb{R}^{+}} \frac{\tilde{\mu}}{{\bar{w}} \bar{\rho}}\Psi^{2}_{x_1}\mathrm{d}x_1\mathrm{d}\tau+\int_{0}^{t}  \frac{-\bar{u}_{1}{}}{{\bar{w}}}\Psi^{2} \Big|_{x_1=0}  \mathrm{d}\tau \\
&=\int_{\mathbb{R}^{+}} \left(\frac{\Phi_{0}^2}{2}+\frac{\Psi_{0}^2}{2 {\bar{w}}}\right)  \mathrm{d}x_1 +\int_{0}^{t} \left[\Phi\Psi-\frac{\tilde{\mu}}{{\bar{w}}\bar{\rho}{}}\Psi(\Psi_{x_1}-\bar{u}_{1}{}\Phi_{x_1})\right]\Big|_{x_1=0}  \mathrm{d}\tau +\int_{0}^{t} \int_{\mathbb{R}^{+}} \tilde{\mu}\frac{{\bar{w}}_{x_1}}{{\bar{w}}^2 \bar{\rho}}\Psi \Psi_{x_1}  \mathrm{d}x_1\mathrm{d}\tau\\
&+ \int_{0}^{t} \int_{\mathbb{R}^{+}} \frac{   \tilde{\mu} \bar{  u }_{1}     }{{\bar{w}}  \bar{\rho}} \Psi_{x_1} \Phi_{x_1}  \mathrm{d}x_1\mathrm{d}\tau-    \int_{0}^{t} \int_{\mathbb{R}^{+}} \frac{\tilde{\mu} \Psi}{{\bar{w}}}\left[  \left(\frac{1}{\bar{\rho}}\right)_{x_{1}} \bar{  u }_{1}  \Phi_{x_1} + \frac{1}{\bar{\rho}} \bar{  u }_{1}'  \Phi_{x_1} + \frac{1}{\bar{\rho}}\bar{  u }_{1}  \Phi_{x_1 x_1}  \right]\mathrm{d}x_1 \mathrm{d}\tau \\
 & + \int_{0}^{t} \int_{\mathbb{R}^{+}} \left(\frac{\Psi}{{\bar{w}}} \Do\mathcal{N}_{11}   +   \frac{\Psi}{{\bar{w}}}  \left(\Do\mathcal{N}_{21}\right) _{x_1}  \right)  \mathrm{d}x_1 \mathrm{d}\tau :=\int_{\mathbb{R}^{+}} \left(\frac{\Phi_{0}^2}{2}+\frac{\Psi_{0}^2}{2 {\bar{w}}}\right)  \mathrm{d}x_1 +\sum_{i=1}^5 I _1^i.
 \end{aligned}
 \end{equation}
We now estimate each $I_1^i$ from $i=1$ to 6. With the aid of Lemma \ref{CLem0} and Cauchy inequality, one gets that
\begin{equation}\label{C25}
\begin{aligned}
  I _{1}^{1} \leq  & C  \delta+   \int_{0}^{t}  \frac{-\bar{u}_{1}{}}{2 {\bar{w}}}\Psi^{2} \Big|_{x_1=0}  \mathrm{d}\tau +  \frac{\tilde{\mu}^2}{2 {\bar{w}} \bar{\rho}}  \int_{0}^{t} |\bar{u}_{1}| \partial_{1}\Phi^2 \Big|_{x_1=0}  \mathrm{d}\tau\\
  \leq  & C  \delta+   \int_{0}^{t}  \frac{-\bar{u}_{1}{}}{2 {\bar{w}}}\Psi^{2} \Big|_{x_1=0}  \mathrm{d}\tau +C    \int_{0}^{t}   |\bar{u}_{1}| \partial_{1}\Phi^2 \Big|_{x_1=0}  \mathrm{d}\tau,
\end{aligned}
\end{equation}
where we have used (\ref{C7}), (\ref{C39}) and (\ref{C3}) in the last inequality. It follows from Lemma $\ref{lem1}$ and $ {(\ref{C7})}$ that
\begin{equation}
\begin{aligned}
  I _{1}^{2} \leq C  \int_{0}^{t}  \int_{\mathbb{R}^{+}} |  \bar{u}_{x_1}  \Psi \Psi_{x_1}|         d x_1  d\tau \leq &    \frac{c_0}{2}  \int_{0}^{t} \left\|    |\bar{u}_{x_1}|^{\frac{1}{2}}  \Psi \right\|_{L^2(\mathbb{R}^{+})}^2 d\tau+   C \delta\int_{0}^{t}  \left\|  \Psi_{x_1}  \right\|_{L^2(\mathbb{R}^{+})}^2d\tau,
\end{aligned}
\end{equation}
and
\begin{equation}
I _{1}^{3} \leq C \delta
 \int_{0}^{t} \left\|(\partial_{1}\Phi,\partial_{1}\Psi)\right\|_{L^{2}(\mathbb{R}^{+})}^{2}d\tau.
\end{equation}
By Cauchy inequality, one gets that
\begin{equation}
\begin{aligned}
I _{1}^{4} &\leq  \int_{0}^{t} \int_{\mathbb{R}^{+}}  \frac{\beta}{2}  \Psi^2 \mathrm{d}x_1 \mathrm{d}\tau+ C   \int_{0}^{t} \int_{\mathbb{R}^{+}}  \bar{ \rho }_{x_1} ^2  \Phi_{x_1  }^2  \mathrm{d}x_1 \mathrm{d}\tau+ C   \int_{0}^{t} \int_{\mathbb{R}^{+}}  \bar{  u }_{x_1}^2   \Phi_{  x_1} ^2  \mathrm{d}x_1 \mathrm{d}\tau\\
&\qquad+ C   \int_{0}^{t} \int_{\mathbb{R}^{+}}  \bar{  u }_{1}^2   \Phi_{x_1 x_1} ^2  \mathrm{d}x_1 \mathrm{d}\tau \\
\leq & \int_{0}^{t} \int_{\mathbb{R}^{+}}  \frac{\beta}{2}  \Psi^2 \mathrm{d}x_1 \mathrm{d}\tau+ C  \delta \int_{0}^{t} \int_{\mathbb{R}^{+}}    \Phi_{x_1  }^2  \mathrm{d}x_1 \mathrm{d}\tau +   C   \delta\int_{0}^{t}    \|\phi_{ x_1}\| ^2   \mathrm{d}\tau.
\end{aligned}
\end{equation}
Using $(\ref{q}), (\ref{C4}) $ and Lemma \ref{Lem-GN}, the term $I_{1}^{5}$ satisfies that
\begin{equation}\label{C13}
\begin{aligned}
& I_{1}^{5} \leq C \int_{0}^{t}\left\|\Psi\right\|_{L^{\infty}(\mathbb{R}^{+})} \left\|(\phi,\zeta_{1})\right\|_{L^{2}(\Omega)}^{2} d\tau\leq C  \int_{0}^{t} \chi\left\|(\phi,\zeta_1)\right\|_{L^2(\Omega)}^2 d\tau.
\end{aligned}
\end{equation}
By (\ref{eq15}), the non-zero mode of $(\phi,\zeta_1)$ satisfies that
$$
\int_{\mathbb{T}^2}(\acute\phi ,\acute\zeta_1)dx_2dx_3=0,\quad\forall \ x_1\in\mathbb{R}^+,\ \ t\geqslant0.
$$
It follows that
\begin{equation}\label{C12}
\|\phi \|_{L^2(\Omega)} \leq  C ( \|\partial_{1} \Phi \|_{L^2(\Omega)}+ \|\nabla  \phi \|_{L^2(\Omega)} ).
\end{equation}
{It is similar to proving that}
\begin{equation}\label{C10}
\left\|\zeta_{1}\right\|_{L^{2}(\Omega)}^{2}\leqslant\left\|\partial_{1}\Psi\right\|_{L^{2}(\mathbb{R}^+)}^{2}+\left\|\nabla\zeta_{1}\right\|_{L^{2}(\Omega)}^{2}.
\end{equation}
Collecting (\ref{C3}) and (\ref{C9}) to (\ref{C10}), one can  obtain (\ref{C11}) directly.
\end{proof}

  \begin{Lem}\label{CLem2}
For $T>0$ and $(\phi, \psi) \in X(0, T)$, under   the  assumptions of Proposition \ref{Cproposition}  with suitably small $\chi+\delta$, for $t \in[0, T]$, one has
\begin{equation}\label{eqC27}
\begin{gathered}
\operatorname*{sup}_{t\in(0,T)}\left\|\partial_{1}\Phi\right\|_{L^{2}(\mathbb{R}^+)}^{2}+\int_{0}^{T}\left\|\partial_{1}\Phi\right\|_{L^{2}(\mathbb{R}^+)}^{2} +    |\partial_{1}\bar{u}_{1}| \Phi_{x_1}^{2} \Big|_{x_1=0} dt \\
\leq C\left\|(\Phi_0,\Psi_0)\right\|_{L^2(\mathbb{R}^+)}^2+C\left\|\phi_0\right\|_{L^2(\Omega)}^2
+C\chi\int_0^T\left\|\nabla(\phi,\psi_1)\right\|_{L^2(\Omega)}^2
dt+C\delta.
\end{gathered}\end{equation}
\end{Lem}
\begin{proof}
It follows  from $(\ref{equ-anti})_1$ that
\begin{equation}
   \partial_{t}\Psi\partial_{1}\Phi=\partial_{t}\big(\Psi\partial_{1}\Phi\big)+\partial_{1}\big(\Psi\partial_{1}\Psi\big)-\big|\partial_{1}\Psi\big|^{2},
\end{equation}
 and
\begin{equation}
\begin{aligned}
&-\tilde{\mu}\partial_{1}\Big[\frac{1}{\bar{\rho}{}}\big(\partial_{1}\Psi-\bar{u}_{1}{}\partial_{1}\Phi\big)\Big]\partial_{1}\Phi\\
&
=\partial_{t}\bigg(\frac{\tilde{\mu}}{2\bar{\rho}{}}|\partial_{1}\Phi|^{2}\bigg) +\frac{\tilde{\mu}\partial_{1}\bar{\rho}{}}{|\bar{\rho}{}|^{2}}\partial_{1}\Psi\partial_{1}\Phi -\frac{\tilde{\mu}}{2}\Big[\partial_{t}\Big(\frac{1}{\bar{\rho}{}}\Big)-\partial_{1}\Big(\frac{\bar{u}_{1}{}}{\bar{\rho}{}}\Big)\Big]\left|\partial_{1}\Phi\right|^{2}+\partial_{1}\Big(\frac{\tilde{\mu}\bar{u}_{1}{}}{2\bar{\rho}{}}|\partial_{1}\Phi|^{2}\Big). \\
\end{aligned}
\end{equation}
 Then multiplying $\partial_1\Phi$ on $(\ref{equ-anti})_2$, and using the above two equities, one can get that
\begin{equation}\label{C21}
\begin{aligned}
&\partial_{t}\Big(\frac{\tilde{\mu}}{2\bar{\rho}{}}\left|\partial_{1}\Phi\right|^{2}+\Psi\partial_{1}\Phi\Big)+{\bar{w}}\left|\partial_{1}\Phi\right|^{2}+\left(\frac {\tilde{\mu} \bar{u} _{1}{}}{2\bar{\rho} {}}\left | \partial_{1}\Phi\right | ^{2}\right)_{x_1}\\
=&-\partial_{1}( \Psi\partial_{1}\Psi) {}-\left(2\bar{u}_{1}{}+\frac{\tilde{\mu}\partial_{1}\bar{\rho}{}}{\left|\bar{\rho}{}\right|^{2}}\right)\partial_{1}\Phi\partial_{1}\Psi+\frac{\tilde{\mu}}{2}\Big[\partial_{t}\Big(\frac{1}{\bar{\rho}{}}\Big)-\partial_{1}\Big(\frac{\bar{u}_{1}{}}{\bar{\rho}{}}\Big)\Big]|\partial_{1}\Phi|^{2}+|\partial_{1}\Psi|^{2}\\
&+\Do\mathcal{N}_{11}\partial_1\Phi+\p_1\left(\Do\mathcal{N}_{21}\right) \partial_1\Phi.
 \end{aligned}
\end{equation}
Integrating (\ref{C21}) with respect to $t$ and $x$ over $ [0,t]\times\Omega 
$, we have
\begin{equation}\label{C22}
\begin{aligned}
& \int_{\Omega} \Big(\frac{\tilde{\mu}}{2\bar{\rho}{}}\left|\partial_{1}\Phi\right|^{2}+\Psi\partial_{1}\Phi\Big)(t)\mathrm{d}x
+{\bar{w}}\int_{0}^{t}  \left\|\partial_{1}\Phi\right\|^{2} \mathrm{d}\tau
+\int_{0}^{t} \int_{\mathbb T^2}\frac {\tilde{\mu} \bar{u} _{1}{}}{2\bar{\rho} {}}\left | \partial_{1}\Phi\right | ^{2}\Big|_{x_1=0}\mathrm{d}x_2\mathrm{d}x_3 \mathrm{d}\tau\\
=&\int_{\Omega} \Big(\frac{\tilde{\mu}}{2\bar{\rho}{}}\left|\partial_{1}\Phi_0\right|^{2}+\Psi_0\partial_{1}\Phi_0\Big)\mathrm{d}x -\int_{0}^{t} \int_{\mathbb T^2}  \Psi\partial_{1}\Psi   \Big|_{x_1=0}\mathrm{d}x_2\mathrm{d}x_3d\tau+\int_{0}^{t} \int_{\Omega}-\left(2\bar{u}_{1}{}+\frac{\tilde{\mu}\partial_{1}\bar{\rho}{}}{\left|\bar{\rho}{}\right|^{2}}\right)\partial_{1}\Phi\partial_{1}\Psi\\
&+\frac{\tilde{\mu}}{2}\Big[\partial_{\tau}\Big(\frac{1}{\bar{\rho}{}}\Big)-\partial_{1}\Big(\frac{\bar{u}_{1}{}}{\bar{\rho}{}}\Big)\Big]|\partial_{1}\Phi|^{2}\mathrm{d}x \mathrm{d}\tau+\int_{0}^{t}  \|\partial_{1}\Psi\|^{2} \mathrm{d}\tau +\int_{0}^{t} \int_{\Omega}\Do\mathcal{N}_{11}\partial_1\Phi+\p_1\left(\Do\mathcal{N}_{21}\right) \partial_1\Phi\mathrm{d}x \mathrm{d}\tau\\
:=&\int_{\Omega} \Big(\frac{\tilde{\mu}}{2\bar{\rho}{}}\left|\partial_{1}\Phi_0\right|^{2}+\Psi_0\partial_{1}\Phi_0\Big)\mathrm{d}x+\sum_{i=1}^4 I_2^i.
 \end{aligned}
\end{equation}
Using Lemma \ref{CLem0}, similar to (\ref{C25}), one gets that
\begin{equation}
 I_{2}^1 \leq C\delta .
\end{equation}
By Lemma \ref{lem1} and (\ref{C7}), one has that
\begin{equation}
 I_{2}^2 \leq C\delta \int_{0}^{t} \left\|\partial_{1}\Phi\right\|_{L^{2}(\mathbb{R}^+)}^{2} d\tau+ C \int_{0}^{t}\left\|\partial_{1}\Psi\right\|_{L^{2}(\mathbb{R}^+)}^{2}d\tau.
\end{equation}
Using the Cauchy inequality, one gets that
\begin{equation}\label{C26}
 \begin{aligned}
 I_2^{4} & \leq C\int_{0}^{t}  \frac{{\bar{w}}}{2}\|\partial_{1}\Phi\|^{2} \mathrm{d}\tau +\int_{0}^{t} \int_{\Omega}\left(\Do\mathcal{N}_{11} \right)^2+\left[\p_1\left(\Do\mathcal{N}_{21}\right)\right] ^2 \mathrm{d}x \mathrm{d}\tau\\
&  \leq C\int_{0}^{t}  \frac{{\bar{w}}}{2}\|\partial_{1}\Phi\|^{2} \mathrm{d}\tau + C  \int_{0}^{t}  \|(\phi,\zeta_1) \|_{L^\infty}^2 \|(\phi,\zeta_1, \p_1\phi,\p_1\zeta_1)\|
  d\tau\\
&  \leq C\int_{0}^{t}  \frac{{\bar{w}}}{2}\|\partial_{1}\Phi\|^{2} \mathrm{d}\tau + C \chi \int_{0}^{t}\left( \left\|\partial_{1}(\Phi,\Psi)\right\|_{L^{2}(\mathbb{R}^+)}^{2}+\left\|\nabla(\phi,\zeta_{1})\right\|\right)
d\tau,
\end{aligned}
\end{equation}
 where we have used (\ref{C4}),  (\ref{C12}),  (\ref{C10}) and Lemma \ref{Lem-GN} in the last inequality.
Note that $\Phi_0^{\prime}={\phim_0} $, then   $\|\Phi_0^{\prime}\|_{L^2(\mathbb{R}^+)}\leq C\|\phi_0\|_{L^2(\Omega)}.$ Collecting (\ref{C22})-(\ref{C26}), we    obtain (\ref{eqC27}).
\end{proof}

\begin{Lem}\label{CLem3}
For $T>0$ and $(\phi, \psi) \in X(0, T)$, under   the  assumptions of Proposition \ref{Cproposition}  with suitably small $\chi+\delta$, we have for $t \in[0, T]$,

\begin{equation}\label{C14}
\|\nabla\psi_{1}\|_{L^{2}(\Omega)}\leq C\left[\|\partial_{1}(\Phi,\Psi)\|_{L^{2}(\mathbb{R}^+)}+\|\nabla\zeta_{1}\|_{L^{2}(\Omega)}+(\delta +\chi)\:\|\nabla\phi\|_{L^{2}(\Omega)}\right],
\end{equation}
 and
\begin{equation}
\begin{aligned}
\|\zeta_1\|_{L^2(\Omega)}&\le C\left[\|\partial_1(\Phi,\Psi)\|_{L^2(\mathbb{R}^+)}+\|\nabla\zeta_1\|_{L^2(\Omega)}+(\delta +\chi)\:\|\nabla\phi\|_{L^2(\Omega)}\right].
\end{aligned}
\end{equation}
\end{Lem}
\begin{proof}
Note that $\psi_1=\rho\zeta_1+\bar{u}_1\phi.$ It follows from  (\ref{C12}), (\ref{C10}), Lemma \ref{lem1} and Lemma \ref{Lem-GN} that
\begin{align*}
\begin{aligned}
\|\nabla\psi_{1}\|_{L^{2}(\Omega)}& \leq C\left[\left\|\nabla\zeta_{1}\right\|_{L^{2}(\Omega)}+\left(\delta+\chi\right)\left\|\nabla\phi\right\|_{L^{2}(\Omega)}+\delta\left\|(\phi,\zeta_{1})\right\|_{L^{2}(\Omega)}\right] \\
&\leq C\left[\left\|\nabla\zeta_{1}\right\|_{L^{2}(\Omega)}+\left(\delta+\chi\right)\left\|\nabla\phi\right\|_{L^{2}(\Omega)}+\delta\left\|\nabla\zeta_{1}\right\|_{L^{2}(\Omega)}+\delta\left\|\partial_{1}(\Phi,\Psi)\right\|_{L^{2}(\mathbb{R}^+)}\right].
\end{aligned}
\end{align*}
Thus, (\ref{C14}) holds true. Then with (\ref{C14}), one can use  (\ref{C12}), (\ref{C10}) and Lemma \ref{lem1} again to obtain
\begin{align*}
\begin{aligned}
\|\zeta_{1}\|_{L^{2}(\Omega)}& \leq C\left(\left\|\psi_{1}\right\|_{L^{2}(\Omega)}+\delta\left\|\phi\right\|_{L^{2}(\Omega)}\right)  \\
&\leq C\left[\|\partial_{1}\Psi\|_{L^{2}(\mathbb{R}^+)}+\|\nabla\psi_{1}\|_{L^{2}(\Omega)}+\delta \left(\|\partial_{1}\Phi\|_{L^{2}(\mathbb{R}^+)}+\|\nabla\phi\|_{L^{2}(\Omega)}\right)\right] \\
&\leq C\left[\|\partial_{1}(\Phi,\Psi)\|_{L^{2}(\mathbb{R}^+)}+\|\nabla\zeta_{1}\|_{L^{2}(\Omega)}+(\delta+\chi)\|\nabla\phi\|_{L^{2}(\Omega)}\right],
\end{aligned}
\end{align*}
which completes the proof.
\end{proof}

Using the above three lemmas, we directly get the estimate about ($\Phi,\Psi$).

\begin{Lem}\label{CLem5}
For $T>0$ and $(\phi, \psi) \in X(0, T)$, under the  assumptions of Proposition \ref{Cproposition}  with suitably small $\chi+\delta$, we have for $t \in[0, T]$,
\begin{equation}\label{C30}
\begin{aligned}
 &\sup_{t\in(0,T)} (\left\| \Phi\right\|_{H^{1}(\mathbb{R}^+)}^{2}+\left\| \Psi\right\|_{L^{2}(\mathbb{R}^+)}^{2})+\int_{0}^{T} \left\||(\bar{u}_{1}{})'|^{\frac{1}{2}}\Psi\right\|_{L^{2}(\mathbb{R}^+)}^{2}dt\\
 &+\int_{0}^{T}  \left(\left\|\partial_{1}(\Phi,\Psi)\right\|_{L^{2}(\mathbb{R}^+)}^{2}  +   |\bar{u}_{1}| \Psi^{2} \Big|_{x_1=0}       +    |\partial_{1}\bar{u}_{1}| \Phi_{x_1}^{2} \Big|_{x_1=0} \right)dt \\
\leq &C\left\|(\Phi_0,\Psi_0)\right\|_{L^2(\mathbb{R}^+)}^2+\left\|\phi_0\right\|_{L^2(\Omega)}^2  +C(\delta+\chi)\int_0^T\left\|(\nabla\phi,\nabla\zeta_1)\right\|_{L^2(\Omega)}^2 \mathrm{d}t + C\delta. \\
 \end{aligned}
\end{equation}

\end{Lem}
Now we return to the original system $( {\ref{eq27}})$ to estimate the 3-d perturbation $(\phi,\zeta)$.

\begin{Lem}\label{Lem1}
For $T>0$ and  $(\phi, \psi) \in X(0, T)$, under   the  assumptions of Proposition \ref{Cproposition}  with suitably small $\chi+\delta$, we have for $t \in[0, T]$,
\begin{equation}\label{C27}
\begin{aligned}
\sup_{t\in(0,T)}&\|(\phi,\zeta)\|_{L^2(\Omega)}^2+\int_0^T\left(\|\nabla\zeta\|_{L^2(\Omega)}^2   +\|{\zeta^{\prime}}\big|_{x_{1}=0}\|_{L^{2}\left(\mathbb{T}^{2}\right)}^{2}\right)dt\\
&
\leqslant  C \|(\phi_0,\zeta_0)\|_{L^2(\Omega)}^2+ C\delta \int_0^T\|\nabla\phi\|_{L^2(\Omega)}^2dt+C \delta.
\end{aligned}
\end{equation}
\end{Lem}
 \begin{proof}
Define
\begin{equation}
\Xi(\rho,\bar{\rho}):=\int_{\bar{\rho}}^{\rho}\frac{p(s)-p(\bar{\rho})}{s^{2}}ds=\frac{\gamma}{(\gamma-1)\rho}\left[p(\rho)-p(\bar{\rho})-p'(\bar{\rho})(\rho-\bar{\rho})\right]{=O(1)|\phi|^2.}
\end{equation}
It follows from $(\ref{NS})_1$ that
\begin{equation}\label{C15}
\begin{aligned}
\partial_{t}(\rho\Xi)+\mathrm{div}(\rho\Xi\mathbf{u})=\rho\big(\partial_{t}\Xi+\mathbf{u}\cdot\nabla\Xi\big)
=&-\left[p(\rho)-p(\bar{\rho})-p^{\prime}(\bar{\rho})\phi\right]\mathrm{div}\bar{\mathbf{u}}-\mathrm{div}\left[\left(p(\rho)-p(\bar{\rho})\right)\zeta\right]\\
&+\zeta\cdot\nabla(p(\rho)-p(\bar\rho))-\frac{p'(\bar\rho)}{\bar\rho}\zeta\cdot\nabla\bar\rho   \phi.
\end{aligned}
\end{equation}
Multiplying $\cdot\zeta$ on $(\ref{eq27})_2$ yields that
\begin{equation}\label{C16}
 \begin{aligned}&\partial_t\Big(\frac{1}{2}\rho\:|\zeta|^2\Big)+\mu\:|\nabla\zeta|^2+(\mu+\lambda)\:|\mathrm{div}\zeta|^2\\
 &=\mathrm{div}(\cdots)-\zeta\cdot\nabla(p(\rho)-p(\bar{\rho}))-\rho(\zeta\cdot\nabla\bar{\mathbf{u}})\cdot\zeta-\phi(\partial_t\bar{\mathbf{u}}+\bar{\mathbf{u}}\cdot\nabla\bar{\mathbf{u}})\cdot\zeta,\end{aligned}
\end{equation}
where $( \cdots ) = \frac \mu2\nabla( |\zeta|^2) + ( \mu+ \lambda) \zeta$div$\zeta- \frac 12\rho|\zeta|^2\mathbf{u} .$ Then summing up (\ref{C15}) and (\ref{C16}) yields that
\begin{equation}\label{C17}
\begin{aligned}
&\partial_{t}\left(\rho\Xi+\frac{1}{2}\rho\left|\zeta\right|^{2}\right)+\mu\left|\nabla\zeta\right|^{2}+\left(\mu+\lambda\right)\left|\operatorname{div}\zeta\right|^{2}-\operatorname{div}\left(\frac \mu2\nabla( |\zeta|^2)  - \frac 12\rho|\zeta|^2\mathbf{u}\right)\\
=&(\mu+ \lambda)\operatorname{div}\left(   \zeta   \operatorname{div} \zeta \right)-\mathrm{div}\left[\left(p(\rho)-p(\bar{\rho})\right)\zeta+\rho\Xi\mathbf{u}\right]\\
&+\left(-\left[p(\rho)-p(\bar{\rho})-p'(\bar{\rho})\phi\right]\mathrm{div}\bar{\mathbf{u}}-\frac{p'(\bar{\rho})}{\bar{\rho}}\nabla\bar{\rho}\cdot\zeta\phi
-\rho(\zeta\cdot\nabla\bar{\mathbf{u}})\cdot\zeta-\phi(\partial_{t}\bar{\mathbf{u}}+\bar{\mathbf{u}}\cdot\nabla\bar{\mathbf{u}})\cdot\zeta\right)\\
:=& (\mu+ \lambda)\operatorname{div}\left(   \zeta   \operatorname{div} \zeta \right)-\mathrm{div}\left[\left(p(\rho)-p(\bar{\rho})\right)\zeta+\rho\Xi\mathbf{u}\right]+\sum_{i=1}^3{\tilde I_4^i}.
\end{aligned}
\end{equation}
Integrating (\ref{C17}) over $\Omega\times(0,t)$, one gets that
\begin{equation}\label{C18}
\begin{aligned}
& \int_{\Omega} \left(\rho\Xi+\frac{1}{2}\rho\left|\zeta\right|^{2}\right)  \mathrm{d}x  +\int_{0}^{t} \int_{\Omega}\left(\mu\left|\nabla\zeta\right|^{2} + \left(\mu+\lambda\right)\left|\operatorname{div}\zeta\right|^{2}\right)\mathrm{d}x \mathrm{d}\tau\\
&\qquad+\int_{0}^{t} \int_{\mathbb{T}^2}  \frac{\mu}{k(x')} (\zeta_{2}^{2}+\zeta_{3}^{2})\Big|_{x_1=0}  \mathrm{d}x_2\mathrm{d}x_3\mathrm{d}\tau \\
&=\int_{\Omega} \left( \rho\Xi+\frac{1}{2}\rho\left|\zeta\right|^{2}\right)  \mathrm{d}x \Big|_{t=0}  -\int_{0}^{t} \int_{\mathbb{T}^2}\left(  \mu \zeta_{1} (\zeta_{1}  )_{x_1}+(\mu+\lambda) \zeta_{1} \operatorname{div}\zeta \Big|_{x_1=0}\right) \mathrm{d}x_2\mathrm{d}x_3 \mathrm{d}\tau\\
&\qquad+\int_{0}^{t} \int_{\mathbb{T}^2}  \left(p(\rho)-p(\bar{\rho})\right)\zeta  \Big|_{x_1=0} \mathrm{d}x_2\mathrm{d}x_3 \mathrm{d}\tau+\int_{\Omega} \sum_{i=1}^3\tilde I_4^i  \mathrm{d}x
:=\int_{\Omega} \left( \rho\Xi+\frac{1}{2}\rho\left|\zeta\right|^{2}\right)  \mathrm{d}x \Big|_{t=0}  +\sum_{i=1}^3 I _4^i,
 \end{aligned}
 \end{equation}
where we have used the boundary condition (\ref{eq2}) and (\ref{eq28}). Next we estimate the right-hand side of the above equality. With the help of Lemma \ref{CLem0}, one gets that
\begin{equation}\label{C99}
\begin{aligned}
I_{4}^1 &\leq C \int_{0}^{t}   \left\|\zeta_{1}\Big|_{x_1=0} \right\|_{L_2(\mathbb{T}^2)}   \left\|\left\| \operatorname{div}\zeta   \right\|_{L^{\infty}(\mathbb R^+)}      \right\|_{L_2(\mathbb{T}^2)} \mathrm{d}\tau\\
&\leq C    \sup_{t>0} \left\|\left\| \operatorname{div}\zeta   \right\|_{L^{\infty}(\mathbb R^+)}      \right\|_{L^2(\mathbb{T}^2)}  \int_{0}^{t}   \left\|\bar{\bf u}(-s\tau+\alpha) \right\|_{L^2(\mathbb{T}^2)}  \mathrm{d}\tau\\
&\leq C    \sup_{t>0} \left\|\left\| \operatorname{div}\zeta   \right\|_{L^{2}(\mathbb R^+)} ^{\frac{1}{2}}   \left\| \operatorname{div}\zeta_{x_1}   \right\|_{L^{2}(\mathbb R^+)} ^{\frac{1}{2}}  \right\|_{L^2(\mathbb{T}^2)} \int_{0}^{t}  \bar{\bf u}(-s\tau+\alpha)  \mathrm{d}\tau\\
&\leq C   \delta \sup_{t>0}  \left\| \operatorname{div}\zeta   \right\|_{L^2(\Omega)} ^{\frac{1}{2}}   \left\| \operatorname{div}\zeta_{x_1}   \right\|_{L^2(\Omega)} ^{\frac{1}{2}} \leq C   \delta.
\end{aligned}
\end{equation}
Similarly, we can get that $I_{4}^2\leq C   \delta $. With the help of the Cauchy's inequality and Lemma {\ref{lem1}}, one gets that
\begin{equation}\label{C29}
\begin{aligned}
I_{4}^3=\int_{0}^{t} \|\sum_{i=1}^3\tilde I_{4}^i\|_{L^{1}(\Omega)}\mathrm{d}\tau& \leq C \delta  \int_{0}^{t}  \left\|(\phi,\zeta_{1})\right\|_{L^{2}(\Omega)}^{2}  \mathrm{d}\tau.
\end{aligned}
\end{equation}
Thus, making use of (\ref{C18}) and (\ref{C29}),  one can obtain  (\ref{C27}).
\end{proof}
Compared to the whole space problem in \cite{LWWarma}, the derivative estimates are obtained by the tangential direction and normal direction respectively due to the boundary effect in this paper. We are ready to estimate the first-order tangential derivative of $(\phi, \zeta)$, based on the symmetric hyperbolic-parabolic structure of the perturbation system.  We rewrite $( {\ref{eq26}})$ for simplicity {as follows}.
\begin{align}\label{eq12}
\left\{\begin{aligned}
&\partial_t \phi+{\bf u} \cdot \nabla \phi+\rho \operatorname{div} \zeta=f, \\
&\rho\left(\partial_t \zeta+{\bf u} \cdot \nabla \zeta\right)+p^{\prime}(\rho) \nabla \phi-\mu \Delta \zeta-(\mu+\lambda) \nabla \operatorname{div} \zeta=g,
\end{aligned}\right.
\end{align}
where
\begin{align}\label{eqj101}
\left\{\begin{array}{l}
f=-\zeta_1 \partial_{1} \bar{\rho}-\phi \partial_{1} \bar{u}_1, \\
g=\left(g_1, 0,0\right)^t, g_1=-\left(p^{\prime}(\rho)-\frac{\rho}{\bar{\rho}} p^{\prime}(\bar{\rho})\right) \partial_{1} \bar{\rho}-\rho \zeta_1 \partial_{1} \bar{u}_1.
\end{array}\right.
\end{align}

\begin{Lem}\label{Lem2}
     Let $T>0$ be a constant and $(\phi, \zeta) \in X(0, T)$ be the solution of \cref{C4} satisfying a priori assumption \cref{C4} with suitably small $\chi+\delta$, it holds that for $t \in[0, T]$,
\begin{align}\label{lem2-1}
\begin{aligned}
 \left\|\partial_{x^{\prime}}(\phi, \zeta)(t)\right\|^2+\int_0^t\left\|\nabla \partial_{x^{\prime}} \zeta\right\|^2 d \tau+\int_0^t\left\|\left.\partial_{x^{\prime}} \zeta^{\prime}\right|_{x_1=0}\right\|_{L^2\left(\mathbb{T}^2\right)}^2 d \tau
\leq  C\left(E^2(0)+M^2(t)+\delta\right),
\end{aligned}
\end{align}
and
\begin{align}\label{lem2-2}
    \begin{aligned}
 \left\|\partial_t(\phi, \zeta)(t)\right\|^2+\int_0^t\left\|\nabla \partial_\tau \zeta\right\|^2 d \tau+\int_0^t\left\|\left.\partial_\tau \zeta^{\prime}\right|_{x_1=0}\right\|_{L^2\left(\mathbb{T}^2\right)}^2 d \tau
\leq  C\left(E^2(0)+M^2(t)+\delta\right) .
\end{aligned}
\end{align}
\end{Lem}
\begin{proof}
First, we derive the first-order special tangential derivative. Applying $\partial_{x^{\prime}}$ to \cref{eq12} yields
\begin{align}\label{eq42}
\left\{\begin{aligned}
&\partial_t \partial_{x^{\prime}} \phi+{\bf u} \cdot \nabla \partial_{x^{\prime}} \phi+\rho \operatorname{div} \partial_{x^{\prime}} \zeta=-\partial_{x^{\prime}} {\bf u} \cdot \nabla \phi-\partial_{x^{\prime}} \rho \operatorname{div} \zeta+\partial_{x^{\prime}} f, \\
&\rho\left(\partial_t \partial_{x^{\prime}} \zeta+{\bf u} \cdot \nabla \partial_{x^{\prime}} \zeta\right)+\nabla\left(p^{\prime}(\rho) \partial_{x^{\prime}} \phi\right)-\mu \Delta \partial_{x^{\prime}} \zeta-(\mu+\lambda) \nabla \operatorname{div} \partial_{x^{\prime}} \zeta \\
&\qquad\qquad\qquad\quad=p^{\prime \prime}(\rho) \nabla \bar{\rho} \partial_{x^{\prime}} \phi-\partial_{x^{\prime}} \rho \partial_t \zeta-\partial_{x^{\prime}}(\rho {\bf u}) \cdot \nabla \zeta+\partial_{x^{\prime}} g.
\end{aligned}\right.
\end{align}

Multiplying \cref{eq42}$_1$ by $\frac{p^{\prime}(\rho)}{\rho} \partial_{x^{\prime}} \phi$, \cref{eq42}$_2$ by $\partial_{x^{\prime}} \zeta$, adding them up and then integrating it with respect to $t, x$ over $(0, t) \times \Omega$, we have
\begin{align*}
\int & \left.\left(\frac{p^{\prime}(\rho)}{\rho} \frac{\left(\partial_{x^{\prime}} \phi\right)^2}{2}+\rho \frac{\left|\partial_{x^{\prime}} \zeta\right|^2}{2}\right) d x\right|_{\tau=0} ^{\tau=t}+\int_0^t \int\left(\mu\left|\nabla \partial_{x^{\prime}} \zeta\right|^2+(\mu+\lambda)\left(\operatorname{div} \partial_{x^{\prime}} \zeta\right)^2\right) d x d \tau \\
& +\left.\mu \int_0^t \int_{\mathbb{T}^2} \frac{1}{k\left(x^{\prime}\right)}\left(\left(\partial_{x^{\prime}} \zeta_2\right)^2+\left(\partial_{x^{\prime}} \zeta_3\right)^2\right)\right|_{x_1=0} d x_2 d x_3 d \tau \\
=- & \left.\mu \sum_{i=2}^3 \int_0^t \int_{\mathbb{T}^2} \partial_{x^{\prime}}\left(\frac{1}{k\left(x^{\prime}\right)}\right) \zeta_i \partial_{x^{\prime}} \zeta_i\right|_{x_1=0} d x_2 d x_3 d \tau \\
& +\int_0^t \int\left((3-\gamma) \frac{p^{\prime}(\rho)}{\rho} \operatorname{div} {\bf u} \frac{\left(\partial_{x^{\prime}} \phi\right)^2}{2}+p^{\prime \prime}(\rho) \partial_{1} \bar{\rho} \partial_{x^{\prime}} \zeta_1 \partial_{x^{\prime}} \phi\right) d x d \tau \\
& -\int_0^t \int\left(\left(\partial_{x^{\prime}} {\bf u} \cdot \nabla \phi+\partial_{x^{\prime}} \rho \operatorname{div} \zeta\right) \frac{p^{\prime}(\rho)}{\rho} \partial_{x^{\prime}} \phi+\left(\partial_{x^{\prime}} \rho \partial_\tau \zeta+\partial_{x^{\prime}}(\rho {\bf u}) \cdot \nabla \zeta\right) \cdot \partial_{x^{\prime}} \zeta\right) d x d \tau \\
& +\int_0^t \int\left(\frac{p^{\prime}(\rho)}{\rho} \partial_{x^{\prime}} f \partial_{x^{\prime}} \phi+\partial_{x^{\prime}} g \cdot \partial_{x^{\prime}} \zeta\right) d x d \tau \\
 &+{\int_0^t \int_{\mathbb{T}^2}  (\mu+\lambda) \operatorname{div} \left(\partial_{x^{\prime}}\zeta\right)\partial_{x^{\prime}}\zeta_1 -\mu \partial_{1}\partial_{x^{\prime}}\zeta_1\partial_{x^{\prime}}\zeta_1 - \partial_{x^{\prime}}\zeta_1 {p^{\prime}(\rho)} \partial_{x^{\prime}} \phi\Big|_{x_1=0} d x_2 d x_3d \tau:=\sum_{i=5}^{9} I_i} .
\end{align*}
For $I_5$, we have
$$
\begin{aligned}
I_5 \leq \frac{\mu}{2} & \left.\int_0^t \int_{\mathbb{T}^2} \frac{1}{k\left(x^{\prime}\right)}\left(\left(\partial_{x^{\prime}} \zeta_2\right)^2+\left(\partial_{x^{\prime}} \zeta_3\right)^2\right)\right|_{x_1=0} d x_2 d x_3 d \tau  +\left.C \int_0^t \int_{\mathbb{T}^2}\left|\zeta^{\prime}\right|^2\right|_{x_1=0} d x_2 d x_3 d \tau .
\end{aligned}
$$
Using Cauchy's inequality, Sobolev's inequality, and the assumption \cref{C4}, we have
\begin{align}\label{eqj71}
\begin{aligned}
I_6& \leq C \int_0^t\left(\|\operatorname{div} \zeta\|_{L^{\infty}}\left\|\partial_{x^{\prime}} \phi\right\|^2+\left\|\p_{1} \bar{u}_1\right\|_{L^{\infty}}\left\|\partial_{x^{\prime}} \phi\right\|^2+\left\|\p_{1} \bar{\rho}\right\|_{L^{\infty}}\left\|\partial_{x^{\prime}} \zeta_1\right\|\left\|\partial_{x^{\prime}} \phi\right\|\right) d \tau\\
& \leq C \int_0^t\left(\|\operatorname{div} \zeta\|_{H^2}\left\|\partial_{x^{\prime}} \phi\right\|^2+\left\|\p_{1} \bar{u}_1\right\|_{L^{\infty}}\left(\left\|\partial_{x^{\prime}} \phi\right\|^2+\left\|\partial_{x^{\prime}} \zeta_1\right\|^2\right)\right) d \tau \\
& \leq C \int_0^t\left(\|\operatorname{div} \zeta\|_{H^1}+\left\|\nabla^2 \operatorname{div} \zeta\right\|\right)\left\|\partial_{x^{\prime}} \phi\right\|^2 d \tau+C \delta\int_0^t\left(\left\|\partial_{x^{\prime}} \phi\right\|^2+\left\|\partial_{x^{\prime}} \zeta_1\right\|^2\right) d \tau \\
& \leq C(\chi+\delta) \int_0^t\left(\left\|\partial_{x^{\prime}} \phi\right\|^2+\left\|\partial_{x^{\prime}} \zeta_1\right\|^2+\left\|\nabla^2 \operatorname{div} \zeta\right\|^2\right) d \tau.
\end{aligned}
\end{align}
From {Cauchy's} inequality, Sobolev's inequality and the assumption \cref{C4}, one has
$$
\begin{aligned}
I_7 \leq & C \int_0^t\left(\left\|\partial_{x^{\prime}} \zeta\right\|_{L^3}\|\nabla \phi\|_{L^6}+\left\|\partial_{x^{\prime}} \phi\right\|_{L^3}\|\nabla \zeta\|_{L^6}\right)\left\|\partial_{x^{\prime}} \phi\right\| d \tau \\
& \left.+C \int_0^t\left\|\partial_{x^{\prime}} \phi\right\|_{L^3}\left\|\partial_\tau \zeta\right\|_{L^6}+\left\|\partial_{x^{\prime}}(\phi, \zeta)\right\|_{L^3}\|\nabla \zeta\|_{L^6}\right)\left\|\partial_{x^{\prime}} \zeta\right\| d \tau \\
\leq & C \int_0^t\left(\left\|\partial_{x^{\prime}} \zeta\right\|^{\frac{1}{2}}\left\|\partial_{x^{\prime}} \zeta\right\|_{H^1}^{\frac{1}{2}}\|\nabla \phi\|_{H^1}+\left\|\partial_{x^{\prime}} \phi\right\|^{\frac{1}{2}}\left\|\partial_{x^{\prime}} \phi\right\|_{H^1}^{\frac{1}{2}}\|\nabla \zeta\|_{H^1}\right)\left\|\partial_{x^{\prime}} \phi\right\| d \tau \\
& +C \int_0^t\left(\left\|\partial_{x^{\prime}} \phi\right\|^{\frac{1}{2}}\left\|\partial_{x^{\prime}} \phi\right\|_{H^1}^{\frac{1}{2}}\left\|\partial_\tau \zeta\right\|_{H^1}+\left\|\partial_{x^{\prime}}(\phi, \zeta)\right\|^{\frac{1}{2}}\left\|\partial_{x^{\prime}}(\phi, \zeta)\right\|_{H^1}^{\frac{1}{2}}\|\nabla \zeta\|_{H^1}\right)\left\|\partial_{x^{\prime}} \zeta\right\| d \tau \\
\leq & C \chi \int_0^t\left(\|\nabla \phi\|_{H^1}^2+\|\nabla \zeta\|_{H^1}^2+\left\|\partial_\tau \zeta\right\|_{H^1}^2\right)d\tau .
\end{aligned}
$$
By Cauchy's inequality and \cref{eq12}, it holds that
$$
I_8 \leq C \delta \int_0^t\left(\left\|\partial_{x^{\prime}} \phi\right\|^2+\left\|\partial_{x^{\prime}} \zeta\right\|^2\right) d \tau .
$$
Next, we estimate the boundary term $I_9$. With the aid of  \eqref{eq28}, one gets that
\begin{equation}
\partial_{x^{\prime}}\zeta_1  \Big|_{x_1=0}= {-A(t)}[({\rho(t,0,x')})^{-1}]_{x'}=O(1)[A(t)+\phi_{x'}].
\end{equation}
 Then, similar to \eqref{C99}, it follows that
 \begin{equation}\label{C100}
 \begin{aligned}
I_9& \leqslant C\int_{0}^{t}\left\|
A(\tau) 
\left\|\operatorname{\nabla}(\zeta,
\phi)\right\|_{L^{\infty}(\mathbb{R}^{+})}\right\|_{L_{2}(\mathbb{T}^{2})}\mathrm{d}\tau  \\
 &\leqslant C \delta \sup_{t>0}
 \left\|\left\|\operatorname{\nabla}(\zeta,
 \phi)\right\|_{L^2(\mathbb{R}^+)}^{\frac12}\|\operatorname{\nabla}(\zeta_{x_1},
 \phi_{x_1})\|_{L^2(\mathbb{R}^+)}^{\frac12}\right\|_{L^2(\mathbb{T}^2)} \leqslant C\delta.
\end{aligned}\end{equation}
Combining the estimates in $I_k$, $k=5,6,7,8,9$, one can get \cref{lem2-1}. The proof of \cref{lem2-2} is similar, noting that
$$
\left\|\partial_t(\phi, \zeta)(0)\right\|^2 \leq C\left\|\nabla^2 \zeta_0\right\|^2+C\left\|\nabla\left(\phi_0, \zeta_0\right)\right\|^2+C( \chi+\delta)\left\|\left(\phi_0, \zeta_0\right)\right\|^2+C \delta.
$$

Then the proof of \cref{Lem2} is completed.
\end{proof}
Next we derive the $H^1$-parabolic estimates for $\zeta$.
\begin{Lem}\label{Lem3}
 Let $T>0$ be a constant and $(\phi, \zeta) \in X(0, T)$ be the solution of \cref{eq12} satisfying a priori assumption \cref{C4} with suitably small $\chi+\delta$, it holds that for $t \in[0, T]$,
\begin{align}\label{eqj70}
\begin{aligned}
& \|\nabla \zeta(t)\|^2+\left\|\left.\zeta^{\prime}(t)\right|_{x_1=0}\right\|_{L^2\left(\mathbb{T}^2\right)}^2+\int_0^t\left\|\partial_\tau \zeta\right\|^2 d \tau \\
\leq & C\|\phi(t)\|^2+ C\eta \int_0^t\|\nabla \phi\|^2 d \tau+C_\eta \int_0^t\|\nabla \zeta\|^2 d \tau+C\left(E^2(0)+M^2(t)\right)+C_\eta \delta,
\end{aligned}
\end{align}
where $\eta$ is a small positive constant to be determined.
\end{Lem}
\begin{proof}
Multiplying \cref{eq12}$_2$ by $\partial_t \zeta$ and integrating the resulting equality with respect to $t, x$ over $[0, t] \times \Omega$, one has 
 \begin{align}
 \begin{aligned}
& \left.\int\left(\frac{\mu}{2}|\nabla \zeta|^2+\frac{\mu+\lambda}{2}(\operatorname{div} \zeta)^2\right) d x\right|_{\tau=0} ^{\tau=t}+\left.\left.\frac{\mu}{2} \int_{\mathbb{T}^2} \frac{\left|\zeta^{\prime}\right|^2}{k\left(x^{\prime}\right)}\right|_{x_1=0} d x_2 d x_3\right|_{\tau=0} ^{\tau=t}+\int_0^t \int \rho\left|\partial_\tau \zeta\right|^2 d x d \tau\nonumber \\
= & -\int_0^t \int p^{\prime}(\rho) \nabla \phi \cdot \partial_\tau \zeta d x d \tau-\int_0^t \int \rho u \cdot \nabla \zeta \cdot \partial_\tau \zeta d x d \tau+\int_0^t \int g_1 \cdot \partial_\tau \zeta_1 d x d \tau\\
&-  \left. \int_0^t \int_{\mathbb{T}^2} \left((\mu+\lambda)\partial_\tau \zeta_1 \operatorname{div} \zeta+\mu \ \zeta_{1}\partial_{{1}}\zeta_{1}\bigg|_{x_{1}=0}\right)dx_{2}dx_{3}\right|_{\tau=0}^{\tau=t}+\mu\int_{0}^{t}\int_{\mathbb{T}^{2}}\zeta_{1}\cdot\partial_{1}\partial_{\tau}\zeta_{1}dx_{2}dx_{3}d\tau\\
=&:I_{10}+I_{11}+I_{12}+I_{13},
\end{aligned}
\end{align}
where we have used
$$
\begin{aligned}
&\quad-\mu \int_0^t \int \operatorname{div}\left(\nabla \zeta \cdot \partial_\tau \zeta\right) d x d \tau =\left.\mu \int_0^t \int_{\mathbb{T}^2} \p_{1} \zeta \cdot \partial_\tau \zeta\right|_{x_1=0} d x_2 d x_3 d \tau \\
& =\left.\mu \int_0^t \int_{\mathbb{T}^2} \frac{1}{k\left(x^{\prime}\right)}\left(\zeta_2 \partial_\tau \zeta_2+\zeta_3 \partial_\tau \zeta_3\right)\right|_{x_1=0} d x_2 d x_3 d \tau\\
\quad&+\left.\left.\mu  \int_{\mathbb{T}^2}   \zeta_1\partial_{1} \zeta_1 \right|_{x_1=0} d x_2 d x_3 \right|_{\tau=0} ^{\tau=t} -\mu \int_0^t \int_{\mathbb{T}^2}  \zeta \cdot \p_{1}\partial_\tau \zeta  \Big|_{x_1=0} d x_2 d x_3 d \tau\\
&=\left.\left.\frac{\mu}{2} \int_{\mathbb{T}^2} \frac{\left|\zeta^{\prime}\right|^2}{k\left(x^{\prime}\right)}\right|_{x_1=0} d x_2 d x_3\right|_{\tau=0} ^{\tau=t}+\left.\left.\mu  \int_{\mathbb{T}^2}   \zeta_1\partial_{1} \zeta_1 \right|_{x_1=0} d x_2 d x_3 \right|_{\tau=0} ^{\tau=t} -\mu \int_0^t \int_{\mathbb{T}^2}  \zeta_1 \cdot \p_{1}\partial_\tau \zeta_1 \Big|_{x_1=0}  d x_2 d x_3 d \tau,
\end{aligned}
$$
and
$$
\begin{aligned}
& -(\mu+\lambda) \int_0^t \int \operatorname{div}\left(\partial_\tau \zeta \operatorname{div} \zeta\right) d x d \tau
=  \left.(\mu+\lambda) \int_0^t \int_{\mathbb{T}^2} \partial_\tau \zeta_1 \operatorname{div} \zeta\right|_{x_1=0} d x_2 d x_3 d \tau  .
\end{aligned}
$$
For $I_{10}$, the integration by parts implies that
$$
\begin{aligned}
& -\int_0^t \int p^{\prime}(\rho) \nabla \phi \cdot \partial_\tau \zeta d x d \tau
=  -\left.\int p^{\prime}(\rho) \nabla \phi \cdot \zeta d x\right|_{\tau=0} ^{\tau=t}+\int_0^t \int \partial_\tau\left(p^{\prime}(\rho) \nabla \phi\right) \cdot \zeta d x d \tau \\
= & \left.\int p^{\prime}(\rho) \phi \operatorname{div} \zeta d x\right|_{\tau=0} ^{\tau=t}+\left.\int \phi \zeta \cdot \nabla p^{\prime}(\rho) d x\right|_{\tau=0} ^{\tau=t}+\int_0^t \int \partial_\tau\left(p^{\prime}(\rho) \nabla \phi\right) \cdot \zeta d x d \tau .
\end{aligned}
$$
Now we handle the last term above. By direct calculation, one has
$$
\begin{aligned}
& \partial_t\left(p^{\prime}(\rho) \nabla \phi\right)+{\bf u} \cdot \nabla\left(p^{\prime}(\rho) \nabla \phi\right)+p^{\prime}(\rho) \rho \nabla \operatorname{div} \zeta \\
= & p^{\prime \prime}(\rho)\left(\partial_t \rho+{\bf u} \cdot \nabla \rho\right) \nabla \phi+p^{\prime}(\rho)\left(\partial_t \nabla \phi+{\bf u} \cdot \nabla(\nabla \phi)+\rho \nabla \operatorname{div} \zeta\right) \\
= & -p^{\prime \prime}(\rho) \rho \operatorname{div}{\bf u} \nabla \phi+p^{\prime}(\rho)(-\nabla {\bf u} \cdot \nabla \phi-\nabla \rho \operatorname{div} \zeta+\nabla f) .
\end{aligned}
$$
Consequently, it follows that,
\begin{align}\label{I99}
& \int_0^t \int \partial_\tau\left(p^{\prime}(\rho) \nabla \phi\right) \cdot \zeta d x d \tau \nonumber\\
= & -\int_0^t {\bf u} \cdot \nabla\left(p^{\prime}(\rho) \nabla \phi\right) \cdot \zeta d x d \tau-\int_0^t \int p^{\prime}(\rho) \rho \zeta \cdot \nabla \operatorname{div} \zeta d x d \tau \nonumber\\
& -\int_0^t \int p^{\prime \prime}(\rho) \rho \operatorname{div} {\bf u} \zeta \cdot \nabla \phi d x d \tau+\int_0^t \int p^{\prime}(\rho)(-\nabla {\bf u} \cdot \nabla \phi-\nabla \rho \operatorname{div} \zeta+\nabla f) \cdot \zeta d x d \tau \\
= & \int_0^t \int p^{\prime}(\rho)\left({\bf u} \cdot \nabla \zeta \cdot \nabla \phi+\rho(\operatorname{div} \zeta)^2+\zeta \cdot \nabla \phi \operatorname{div} \zeta-\zeta \cdot \nabla \zeta \cdot \nabla \phi\right) d x d \tau\nonumber \\
& +\int_0^t \int p^{\prime}(\rho) \zeta \cdot \nabla f d x d \tau+\int_0^t \int p^{\prime}(\rho) \p_{1} \bar{u}_1\left(\zeta \cdot \nabla \phi-\zeta_1 \p_{1} \phi\right) d x d \tau \nonumber\\
& +\int_0^t \int p^{\prime \prime}(\rho) \rho\left(\zeta_1 \p_{1} \bar{\rho} \operatorname{div} \zeta-\p_{1} \bar{u}_1 \zeta \cdot \nabla \phi\right) d x d \tau:=\sum_{i=1}^{4} I_{10}^i .\nonumber
\end{align}
For $I_{10}^1$, we have
$$
I_{10}^1 \leq \eta\int_0^t\|\nabla \phi\|^2 d \tau+C_\eta \int_0^t\|\nabla \zeta\|^2 d \tau+C \chi \int_0^t\|\nabla(\phi, \zeta)\|^2 d \tau .
$$
Noting that $\abs{\p_{1}(\rhob,\ub)}\leq C\delta$, we have
\begin{align*}
I_{10}^2 \leq &\int_0^t\int \abs{p'(\rho)\dv\zeta}\abs{f}+\abs{p''(\rho)\nabla\rho\cdot\zeta} \abs{f}dxd\tau \\
\leq&C \int_0^t\big(\abs{\nabla\zeta}+\abs{\p_{1}\rhob\zeta_1}+\abs{\nabla\phi}\abs{\zeta}\big)\big(\abs{\zeta_1 \partial_{1} \bar{\rho}}+\abs{\phi \partial_{1} \bar{u}_1}\big)dxd\tau\\
\leq & \delta \int_{0}^{t}\|(\nabla \phi,\nabla\zeta)\|^{2} d \tau+C \delta \int_{0}^{t}\left\|\left(\phi, \zeta_{1}\right)\right\|^{2} d \tau
\leq C\delta\int_0^t\|\partial_{1}(\Phi,\Psi)\|^2_{L^{2}(\mathbb{R}^+)}+\|(\nabla\zeta_{1},\nabla\phi)\|^2_{L^{2}(\Omega)}d\tau.
\end{align*}
Similarly, one has
$$
I_{10}^3+I_{10}^4 \leq C\delta\int_0^t\left(\|\partial_{1}(\Phi,\Psi)\|^2_{L^{2}(\mathbb{R}^+)}+\|(\nabla\zeta_{1},\nabla\phi)\|^2_{L^{2}(\Omega)}\right)d\tau.
$$
Combining the estimates of $I_{10}^1-I_{10}^4$, we can get
\begin{align}\label{I9}
    \begin{aligned}
& \int_{0}^{t} \int \partial_{t}\left(p^{\prime}(\rho) \nabla \phi\right) \cdot \zeta d x d \tau
\leq  \int_{0}^{t}\|\nabla \phi\|^{2} d \tau+C_{\eta} \int_{0}^{t}\|\nabla \zeta\|^{2} d \tau+C (\delta+\chi)M(t).
\end{aligned}
\end{align}
Then by \cref{I99,I9} and Cauchy's inequality, we arrive at
$$
\begin{aligned}
& -\int_{0}^{t} \int p^{\prime}(\rho) \nabla \phi \cdot \partial_{\tau} \zeta d x d \tau \\
\leq & \frac{\mu+\lambda}{4}\|\operatorname{div} \zeta(t)\|^{2}+C\|\phi(t)\|^{2}+C\left(E^{2}(0)+M^{2}(t)\right) \\
& +C(\chi+\delta)\left(\|\phi(t)\|_{H^{1}}^{2}+\|\zeta(t)\|^{2}\right)+\eta \int_{0}^{t}\|\nabla \phi\|^{2} d \tau+C_{\eta} \int_{0}^{t}\|\nabla \zeta\|^{2} d \tau+C_{\eta} \delta.
\end{aligned}
$$
For $I_{11}$ and $I_{12}$, by Cauchy's inequality, it holds
$$
-\int_{0}^{t} \int \rho u \cdot \nabla \zeta \cdot \partial_{\tau} \zeta d x d \tau \leq \frac{1}{4} \int_{0}^{t}\left\|\sqrt{\rho} \partial_{\tau} \zeta\right\|^{2}d\tau+C \int_{0}^{t}\|\nabla \zeta\|^{2} d \tau,
$$
and
$$
\int_{0}^{t} \int g_{1} \cdot \partial_{\tau} \zeta_{1} d x d \tau \leq \frac{1}{4} \int_{0}^{t}\left\|\sqrt{\rho} \partial_{\tau} \zeta\right\|^{2}d\tau+C\delta\int_0^t\|\partial_{1}(\Phi,\Psi)\|^2_{L^{2}(\mathbb{R}^+)}+\|(\nabla\zeta_{1},\nabla\phi)\|^2_{L^{2}(\Omega)}d\tau+C
{\delta} .
$$
Finally, one has
$$
\begin{aligned}
\left.\left.\frac{\mu}{2} \int_{\mathbb{T}^{2}} \frac{\left|\zeta^{\prime}\right|^{2}}{k\left(x^{\prime}\right)}\right|_{x_{1}=0} d x_{2} d x_{3}\right|_{\tau=0}
\leq & C \int_{\mathbb{T}^{2}}\left\|\zeta_{0}^{\prime}\right\|_{L^{\infty}(\mathbb R^+)}^{2} d x_{2} d x_{3} \leq C \int_{\mathbb{T}^{2}}\left\|\zeta_{0}^{\prime}\right\|_{L^{2}(\mathbb R^+)}\left\|\p_{1} \zeta_{0}^{\prime}\right\|_{L^{2}(\mathbb R^+)} d x_{2} d x_{3} \\
\leq & C\left(\left\|\zeta_{0}^{\prime}\right\|^{2}+\left\|\p_{1} \zeta_{0}^{\prime}\right\|^{2}\right).
\end{aligned}
$$
Similar like (\ref{C99}), we have
  $I_{13}\leq C \delta$. \color{black}By estimates of $ I_{10},I_{11},I_{12},I_{13}$, we can obtain \eqref{eqj70}. The proof of \cref{Lem3} is completed.
\end{proof}

Then, we derive the dissipative estimates for {the} normal derivative of $\phi$, i.e., $\left\|\partial_{{1}} \phi\right\|^{2}$, which follows from a hyperbolic-parabolic structure of the perturbation system \eqref{eq26}.

\begin{Lem}\label{Lem4}
 Let $T>0$ be a constant and $(\phi, \zeta) \in X(0, T)$ be the solution of \eqref{eq26} satisfying a priori assumption \eqref{C4} with suitably small $\chi+\delta$, it holds that for $t \in[0, T]$,
\begin{align}\label{eqj0}
\begin{aligned}
&\left\|\partial_{{1}} \phi(t)\right\|^{2}+\int_{0}^{t}\left(\left\|\partial_{{1}} \phi\right\|^{2}+\left\|\partial_{{1}}^{2} \zeta_{1}\right\|^{2}\right) d \tau\\
\leq & C \int_{0}^{t}\left(\left\|\partial_{t} \zeta_{1}\right\|^{2}+\left\|\nabla \zeta_{1}\right\|^{2}+\left\|\nabla \partial_{x^{\prime}} \zeta\right\|^{2}\right) d \tau
+C\left(E^{2}(0)+M^{2}(t)+\delta\right).
\end{aligned}
\end{align}
\end{Lem}
\begin{proof}
Applying $\partial_{1}$ to \eqref{eq26}$_{1}$ and rewriting \eqref{eq26}$_2$, one has
\begin{align}\label{eqj1}
\left\{\begin{array}{l}
\partial_{t} \partial_{{1}} \phi+{\bf u} \cdot \nabla \partial_{{1}} \phi+\rho \partial_{{1}}^{2} \zeta_{1}+\rho \partial_{{1}} \nabla^{\prime} \cdot \zeta^{\prime}+\partial_{{1}}{\bf u} \cdot \nabla \phi+\partial_{{1}} \rho \operatorname{div} \zeta=\partial_{{1}} f, \\
\rho\left(\partial_{t} \zeta_{1}+{\bf u} \cdot \nabla \zeta_{1}\right)+p^{\prime}(\rho) \partial_{{1}} \phi-(2 \mu+\lambda) \partial_{{1}}^{2} \zeta_{1}-\mu \Delta^{\prime} \zeta_{1}-(\mu+\lambda) \partial_{{1}} \nabla^{\prime} \cdot \zeta^{\prime}=g_{1},
\end{array}\right.
\end{align}
where $\Delta^{\prime}=\partial_{x_{2}}^{2}+\partial_{x_{3}}^{2}, \nabla^{\prime}=\left(\partial_{x_{2}}, \partial_{x_{3}}\right)$ and $\nabla^{\prime} \cdot \zeta^{\prime}=\partial_{x_{2}} \zeta_{2}+\partial_{x_{3}} \zeta_{3}$. Multiplying \eqref{eqj1}$_1$ by $\frac{1}{\rho} \partial_{{1}} \phi$ and \eqref{eqj1}$_{2}$ by $\frac{1}{2 \mu+\lambda} \partial_{{1}} \phi$ and adding the resulted equations together, then integrating the resulted equality with respect to $t, x$ over $[0, t] \times \Omega$, it yields
\begin{align}\label{eqj3}
\begin{aligned}
& \left.\int \frac{1}{\rho} \frac{\left(\partial_{{1}} \phi\right)^{2}}{2} d x\right|_{\tau=0} ^{\tau=t}+\int_{0}^{t} \int \frac{p^{\prime}(\rho)}{2 \mu+\lambda}\left(\partial_{{1}} \phi\right)^{2} d x d \tau \\
= & \int_{0}^{t} \int\left[\mu\left(\Delta^{\prime} \zeta_{1}-\partial_{{1}} \nabla^{\prime} \cdot \zeta^{\prime}\right)
-\rho\left(\partial_{\tau} \zeta_{1}+{\bf u} \cdot \nabla \zeta_{1}\right)\right] \frac{\partial_{{1}} \phi}{2 \mu+\lambda} d x d \tau \\
& +\int_{0}^{t} \int\left(\frac{\operatorname{div} {\bf u}}{\rho}\left(\partial_{{1}} \phi\right)^{2}-\frac{1}{\rho} \partial_{{1}} {\bf u} \cdot \nabla \phi \partial_{{1}} \phi-\frac{1}{\rho} \partial_{{1}} \rho \operatorname{div} \zeta \partial_{{1}} \phi\right) d x d \tau \\
& +\int_{0}^{t} \int\left(\frac{1}{\rho} \partial_{{1}} f \partial_{{1}} \phi+\frac{1}{2 \mu+\lambda} g_{1} \partial_{{1}} \phi\right) d x d \tau:=\sum_{i=14}^{16} I_{i}.
\end{aligned}
\end{align}
Here the boundary term vanishes under the condition \eqref{eq28}. By Cauchy's inequality, one has
\begin{align}\label{eqj7}
I_{14} \leq \frac{1}{16(2 \mu+\lambda)} \int_{0}^{t}\left\|\sqrt{p^{\prime}(\rho)} \partial_{{1}} \phi\right\|^{2} d \tau+C \int_{0}^{t}\left(
\left\|\nabla \partial_{x_2} \zeta\right\|^{2}+\left\|\nabla \partial_{x_3} \zeta\right\|^{2}+\left\|\partial_{\tau} \zeta_{1}\right\|^{2}+\left\|\nabla \zeta_{1}\right\|^{2}\right) d \tau.
\end{align}
From Cauchy's inequality and the assumption \eqref{C4}, one obtains that
\begin{align}\label{eqj4}
\begin{aligned}
I_{15}  &=\int_{0}^{t} \int \frac{1}{\rho}\left(-\partial_{{1}} \zeta \cdot \nabla \phi \partial_{{1}} \phi-\partial_{{1}} \bar{\rho} \operatorname{div} \zeta \partial_{{1}} \phi\right) d x d \tau \\
& \leq C \int_{0}^{t}\left(\left\|\partial_{{1}} \zeta\right\|_{L^{6}}\|\nabla \phi\|_{L^{3}}\left\|\partial_{{1}} \phi\right\|+\delta\|\nabla \zeta\|\left\|\partial_{{1}} \phi\right\|\right) d \tau  \leq C(\chi+\delta) \int_{0}^{t}\left(\|\nabla \zeta\|_{H^{1}}^{2}+\left\|\partial_{{1}} \phi\right\|^{2}\right)d\tau.
\end{aligned}
\end{align}
By Cauchy's inequality and Lemma \ref{lem1}, one has
\begin{align}\label{eqj5}
\begin{aligned}
I_{16} \leq & \frac{1}{16(2 \mu+\lambda)} \int_{0}^{t}\left\|\sqrt{p^{\prime}(\rho)} \partial_{{1}} \phi\right\|^{2} d \tau+\int_{0}^{t}\int |\p_1f|^2+|g_1|^2dxd\tau\\
\le& \frac{1}{16(2 \mu+\lambda)} \int_{0}^{t}\left\|\sqrt{p^{\prime}(\rho)} \partial_{{1}} \phi\right\|^{2} d \tau+C\delta\int_0^t\left(\|(\phi,\zeta_1)\|^2+\|\p_1(\phi,\zeta_1)\|^2\right)d\tau.
\end{aligned}
\end{align}
By Lemma 
\ref{CLem3} and \ref{Lem1}, substituting \eqref{eqj7}-\eqref{eqj5} into \eqref{eqj3} yields
\begin{align}\label{eqj8}
\begin{aligned}
 \left\|\partial_{{1}} \phi(t)\right\|^{2}+\int_{0}^{t}\left\|\partial_{{1}} \phi\right\|^{2} d \tau 
\leq&  C\left\|\partial_{{1}} \phi_{0}\right\|^{2}+C \int_{0}^{t}\left(\left\|\partial_{\tau} \zeta_{1}\right\|^{2}+\left\|\nabla \zeta_{1}\right\|^{2}+\left\|\nabla \partial_{x^{\prime}} \zeta\right\|^{2}\right) d \tau \\
&+C(\chi+\delta) \int_{0}^{t}\|\nabla \zeta\|_{H^{1}}^{2} d \tau+C \delta\int_{0}^{t}\left\|
\left(\phi, \zeta_{1}\right)\right\|^{2} d \tau\\
\le &C\left\|\partial_{{1}} \phi_{0}\right\|^{2}+C \int_{0}^{t}\left(\left\|\partial_{\tau} \zeta_{1}\right\|^{2}+\left\|\nabla \zeta_{1}\right\|^{2}+\left\|\nabla \partial_{x^{\prime}} \zeta\right\|^{2}\right) d \tau \\
& +C(\chi+\delta) \int_{0}^{t}\left(\|\nabla \zeta\|_{H^{1}}^{2} +\|\nabla\phi\|^2\right)d \tau+C \delta\int_{0}^{t}\left\|
\left(\phi,\p_1\Phi,\p_{x_2}\Psi\right)\right\|^{2} d \tau.
\end{aligned}
\end{align}
It follows from $\eqref{eqj1}_{2}$ that
\begin{align}\label{eqj9}
\begin{aligned}
\int_{0}^{t}\left\|\partial_{{1}}^{2} \zeta_{1}\right\|^{2} d \tau \leq & C \int_{0}^{t}\left(\left\|\partial_{{1}} \phi\right\|^{2}+\left\|\partial_{\tau} \zeta_{1}\right\|^{2}+\left\|\nabla \zeta_{1}\right\|^{2}+\left\|\nabla \partial_{x^{\prime}} \zeta\right\|^{2}\right) d \tau+C\delta.
\end{aligned}
\end{align}
Hence, multiplying \eqref{eqj8} by a large constant $C$ and combining \eqref{eqj9}, one can get \eqref{eqj0}. The proof of Lemma \ref{Lem7} is completed.
\end{proof}

Next we are ready to estimate the tangential derivatives of $\phi$ by the momentum equations.

\begin{Lem}\label{Lem5}
 Let $T>0$ be a constant and $(\phi, \zeta) \in X(0, T)$ be the solution of \eqref{eq26} satisfying a priori assumption \eqref{C4} with suitably small $\chi+\delta$, it holds that for $t \in[0, T]$,
 \begin{align}\label{eqj30}
\begin{aligned}
\int_{0}^{t}\left(\left\|\partial_{x^{\prime}} \phi\right\|^{2}+\left\|\partial_{{1}}^{2} \zeta^{\prime}\right\|^{2}\right) d \tau
&\leq  C \eta \int_{0}^{t}\left\|\partial_{{1}} \phi\right\|^{2} d \tau+C_{\eta} \int_{0}^{t}\left\|\nabla \partial_{x^{\prime}} \zeta\right\|^{2} d \tau  \\
&\qquad+C \int_{0}^{t}\left(\left\|\partial_{\tau} \zeta^{\prime}\right\|^{2}+\left\|\nabla \zeta^{\prime}\right\|^{2}+\left\|\partial_{{1}} \partial_{x^{\prime}} \phi\right\|^{2}+\left\|\left.\zeta^{\prime}\right|_{x_{1}=0}\right\|_{L^{2}\left(\mathbb{T}^{2}\right)}^{2}\right) d \tau.
\end{aligned}
\end{align}
\end{Lem}
\begin{proof}
Rewrite the second perturbed momentum equations in \eqref{eq26} as
\begin{align}\label{eqj37}
\rho\left(\partial_{t} \zeta_{2}+{\bf u} \cdot \nabla \zeta_{2}\right)+p^{\prime}(\rho) \partial_{x_{2}} \phi-\mu \partial_{{1}}^{2} \zeta_{2}-\mu \Delta^{\prime} \zeta_{2}-(\mu+\lambda) \partial_{x_{2}} \operatorname{div} \zeta=0.
\end{align}
Multiplying the above equation by $\partial_{x_{2}} \phi$ and integrating the resulting equation with respect to $t, x$ over $(0, t) \times \Omega$, one has
\begin{align}\label{eqj31}
\begin{aligned}
& \int_{0}^{t} \int p^{\prime}(\rho)\left(\partial_{x_{2}} \phi\right)^{2} d x d \tau=\hat I_1+\hat I_2\\
:= & \int_{0}^{t} \int \mu \partial_{{1}}^{2} \zeta_{2} \partial_{x_{2}} \phi d x d \tau+\int_{0}^{t} \int\left(\mu \Delta^{\prime} \zeta_{2}+(\mu+\lambda) \partial_{x_{2}}\dv \zeta-\rho\left(\partial_{\tau} \zeta_{2}+{\bf u} \cdot \nabla \zeta_{2}\right)\right) \partial_{x_{2}} \phi d x d \tau.
\end{aligned}
\end{align}
Integration by parts under the boundary conditions \eqref{eq28} leads to
\begin{align}\label{eqj32}
\begin{aligned}
\hat I_1= & -\mu \int_{0}^{t} \int \partial_{{1}}^{2} \partial_{x_{2}} \zeta_{2} \phi d x d \tau
=  \left.\mu \int_{0}^{t} \int_{\mathbb{T}^{2}} \partial_{{1}} \partial_{x_{2}} \zeta_{2} \phi\right|_{x_{1}=0} d x_{2} d x_{3} d \tau+\mu \int_{0}^{t} \int \partial_{{1}} \partial_{x_{2}} \zeta_{2} \partial_{{1}} \phi d x d \tau \\
=&-\left.\mu \int_{0}^{t} \int_{\mathbb{T}^{2}} \frac{1}{k\left(x^{\prime}\right)} \zeta_{2} \partial_{x_{2}} \phi\right|_{x_{1}=0} d x_{2} d x_{3} d \tau+\mu \int_{0}^{t} \int \partial_{{1}} \partial_{x_{2}} \zeta_{2} \partial_{{1}} \phi d x d \tau.
\end{aligned}
\end{align}
By direct calculation, it yields 
\begin{align}\label{eqj33}
\begin{aligned}
 -\left.\mu \int_{0}^{t} \int_{\mathbb{T}^{2}} \frac{1}{k\left(x^{\prime}\right)} \zeta_{2} \partial_{x_{2}} \phi\right|_{x_{1}=0}& d x_{2} d x_{3} d \tau
\leq  C \int_{0}^{t} \int_{\mathbb{T}^{2}}\left|\zeta_{2}\right|_{x_{1}=0} \mid\left\|\partial_{x_{2}} \phi\right\|_{L^{\infty}(\mathbb R^+)} d x_{2} d x_{3} d \tau \\
\leq & C \int_{0}^{t} \int_{\mathbb{T}^{2}}\left|\zeta_{2}\right|_{x_{1}=0} \mid\left\|\partial_{x_{2}} \phi\right\|_{L^{2}(\mathbb R^+)}^{\frac{1}{2}}\left\|\partial_{{1}} \partial_{x_{2}} \phi\right\|_{L^{2}(\mathbb R^+)}^{\frac{1}{2}} d x_{2} d x_{3} d \tau \\
\leq & {C \int_{0}^{t}\left\|\left.\zeta_{2}\right|_{x_{1}=0}\right\|_{L^{2}\left(\mathbb{T}^{2}\right)}\left\|\partial_{x_{2}} \phi\right\|_{L^{2}(\Omega)}^{\frac{1}{2}}\left\|\partial_{{1}} \partial_{x_{2}} \phi\right\|_{L^{2}(\Omega)}^{\frac{1}{2}} d \tau} \\
\leq & \frac{1}{16} \int_{0}^{t}\left\|\sqrt{p^{\prime}(\rho)} \partial_{x_{2}} \phi\right\|^{2} d \tau+C \int_{0}^{t}\left(\left\|\left.\zeta_{2}\right|_{x_{1}=0}\right\|_{L^{2}\left(\mathbb{T}^{2}\right)}^{2}+\left\|\partial_{{1}} \partial_{x_{2}} \phi\right\|^{2}\right) d \tau.
\end{aligned}
\end{align}
Substituting the above inequality into \eqref{eqj32} and using Cauchy's inequality, one can obtain
\begin{align}\label{eqj35}
\begin{aligned}
\hat I_1
\leq  \frac{1}{16} \int_{0}^{t}&\left\|\sqrt{p^{\prime}(\rho)} \partial_{x_{2}} \phi\right\|^{2} d \tau+\eta \int_{0}^{t}\left\|\partial_{{1}} \phi\right\|^{2} d \tau+C_{\eta} \int_{0}^{t}\left\|\partial_{{1}} \partial_{x_{2}} \zeta_{2}\right\|^{2} d \tau \\
& +C \int_{0}^{t}\left(\left\|\left.\zeta_{2}\right|_{x_{1}=0}\right\|_{L^{2}\left(\mathbb{T}^{2}\right)}^{2}+\left\|\partial_{{1}} \partial_{x_{2}} \phi\right\|^{2}\right) d \tau,
\end{aligned}
\end{align}
where $\eta$ is some constant to be determined.
By Cauchy's inequality, one has
\begin{align}\label{eqj34}
\begin{aligned}
& \int_{0}^{t} \int\left(\mu \Delta^{\prime} \zeta_{2}+(\mu+\lambda) \partial_{x_{2}} \operatorname{div} \zeta-\rho\left(\partial_{\tau} \zeta_{2}+u \cdot \nabla \zeta_{2}\right)\right) \partial_{x_{2}} \phi d x d \tau \\
& \leq \frac{1}{16} \int_{0}^{t}\left\|\sqrt{p^{\prime}(\rho)} \partial_{x_{2}} \phi\right\|^{2} d \tau+C \int_{0}^{t}\left(\left\|\nabla \partial_{x^{\prime}} \zeta\right\|^{2}+\left\|\partial_{\tau} \zeta_{2}\right\|^{2}+\left\|\nabla \zeta_{2}\right\|^{2}\right) d \tau.
\end{aligned}
\end{align}
Substitution \eqref{eqj35} and \eqref{eqj34} into \eqref{eqj31} yields
\begin{align}\label{eqj36}
\begin{aligned}
\int_{0}^{t}\left\|\partial_{x_{2}} \phi\right\|^{2} d \tau \leq & \eta \int_{0}^{t}\left\|\partial_{{1}} \phi\right\|^{2} d \tau+C_{\eta} \int_{0}^{t}\left\|\nabla \partial_{x^{\prime}} \zeta\right\|^{2} d \tau \\
& +C \int_{0}^{t}\left(\left\|\partial_{\tau} \zeta_{2}\right\|^{2}+\left\|\nabla \zeta_{2}\right\|^{2}+\left\|\partial_{{1}} \partial_{x_{2}} \phi\right\|^{2}+\left\|\left.\zeta_{2}\right|_{x_{1}=0}\right\|_{L^{2}\left(\mathbb{T}^{2}\right)}^{2}\right) d \tau.
\end{aligned}
\end{align}
It follows from \eqref{eqj37} that
\begin{align}\label{eqj38}
\int_{0}^{t}\left\|\partial_{{1}}^{2} \zeta_{2}\right\|^{2} d \tau \leq C \int_{0}^{t}\left(\left\|\partial_{x_{2}} \phi\right\|^{2}+\left\|\partial_{\tau} \zeta_{2}\right\|^{2}+\left\|\nabla \zeta_{2}\right\|^{2}+\left\|\nabla \partial_{x^{\prime}} \zeta\right\|^{2}\right) d \tau+C\delta.
\end{align}
Then, multiplying \eqref{eqj36} by a large constant $C$ and combining with \eqref{eqj38}, one has
\begin{align}\label{eqj39}
\begin{aligned}
\int_{0}^{t}\left(\left\|\partial_{x_{2}} \phi\right\|^{2}+\left\|\partial_{{1}}^{2} \zeta_{2}\right\|^{2}\right) &d \tau \leq  \eta \int_{0}^{t}\left\|\partial_{{1}} \phi\right\|^{2} d \tau+C_{\eta} \int_{0}^{t}\left\|\nabla \partial_{x^{\prime}} \zeta\right\|^{2} d \tau \\
&+C \int_{0}^{t}\left(\left\|\partial_{\tau} \zeta_{2}\right\|^{2}+\left\|\nabla \zeta_{2}\right\|^{2}+\left\|\partial_{{1}} \partial_{x_{2}} \phi\right\|^{2}+\left\|\left.\zeta_{2}\right|_{x_{1}=0}\right\|_{L^{2}\left(\mathbb{T}^{2}\right)}^{2}\right) d \tau+C\delta.
\end{aligned}
\end{align}

Similar to \eqref{eqj39}, one can obtain the estimate of $\int_{0}^{t}\left(\left\|\partial_{x_{3}} \phi\right\|^{2}+\left\|\p_{1}^{2} \zeta_{3}\right\|^{2}\right) d \tau$. Then we have \eqref{eqj0}, and the proof of \cref{Lem5} is completed.
\end{proof}

The higher order derivative estimates can be given by the same methods as the lower order derivative estimates. We obtain the tangential derivative estimates as follows.
\begin{Lem}\label{Lem6}
 Let $T>0$ be a constant and $(\phi, \zeta) \in X(0, T)$ be the solution of \cref{eq26} satisfying a priori assumption \cref{C4} with suitably small $\chi+\delta$, it holds that for $t \in[0, T]$,
\begin{align}\label{eqj74}
\begin{aligned}
& \left\|\partial_{x^{\prime}}^{2}(\phi, \zeta)(t)\right\|^{2}+\int_{0}^{t}\left\|\nabla \partial_{x^{\prime}}^{2} \zeta\right\|^{2} d \tau+\int_{0}^{t}\left\|\left.\partial_{x^{\prime}}^{2} \zeta^{\prime}\right|_{x_{1}=0}\right\|_{L^{2}\left(\mathbb{T}^{2}\right)}^{2} d \tau
\leq  C\left(E^{2}(0)+M^{2}(t)+\delta\right).
\end{aligned}
\end{align}
\end{Lem}
\begin{proof}
 Applying $\partial_{x^{\prime}}^{2}$ to \cref{eq12}, one has
\begin{align}\label{eqlem7}
\left\{\begin{aligned}
&\partial_{t} \partial_{x^{\prime}}^{2} \phi+{\bf u} \cdot \nabla \partial_{x^{\prime}}^{2} \phi+\rho \operatorname{div} \partial_{x^{\prime}}^{2} \zeta=-\left[\partial_{x^{\prime}}^{2}, {\bf u} \cdot \nabla\right] \phi-\left[\partial_{x^{\prime}}^{2}, \rho\right] \operatorname{div} \zeta+\partial_{x^{\prime}}^{2} f, \\
&\rho\left(\partial_{t} \partial_{x^{\prime}}^{2} \zeta+{\bf u} \cdot \nabla \partial_{x^{\prime}}^{2} \zeta\right)+\nabla\left(p^{\prime}(\rho) \partial_{x^{\prime}}^{2} \phi\right)-\mu \Delta \partial_{x^{\prime}}^{2} \zeta-(\mu+\lambda) \nabla \operatorname{div} \partial_{x^{\prime}}^{2} \zeta \\
&\qquad=p^{\prime \prime}(\rho) \nabla \rho \partial_{x^{\prime}}^{2} \phi-\left[\partial_{x^{\prime}}^{2}, \rho\right] \partial_{t} \zeta-\left[\partial_{x^{\prime}}^{2}, \rho {\bf u} \cdot \nabla\right] \zeta-\left[\partial_{x^{\prime}}^{2}, p^{\prime}(\rho)\right] \nabla \phi+\partial_{x^{\prime}}^{2} g.
\end{aligned}\right.
\end{align}
 Denote $[A, B]:=A B-B A$ as the commutator of $A$ and $B$. Multiplying the equation $\cref{eqlem7}_{1}$ by $\frac{p^{\prime}(\rho)}{\rho} \partial_{x^{\prime}}^{2} \phi,\cref{eqlem7}_{2}$ by $\partial_{x^{\prime}}^{2} \zeta$ and adding them up, and then integrating the resulting equation with respect to $t, x$ over $[0, t] \times \Omega$, one has
\begin{equation}\label{eqj107}
\begin{aligned}
& \left.\int\left(\frac{p^{\prime}(\rho)}{\rho} \frac{\left(\partial_{x^{\prime}}^{2} \phi\right)^{2}}{2}+\rho \frac{\left|\partial_{x^{\prime}}^{2} \zeta\right|^{2}}{2}\right) d x\right|_{\tau=0} ^{\tau=t}+\int_{0}^{t} \int\left(\mu\left|\nabla \partial_{x^{\prime}}^{2} \zeta\right|^{2}+(\mu+\lambda)\left(\operatorname{div} \partial_{x^{\prime}}^{2} \zeta\right)^{2}\right) d x d \tau \\
& +\left.\mu \int_{0}^{t} \int_{\mathbb{T}^{2}} \frac{1}{k\left(x^{\prime}\right)}\left(\left(\partial_{x^{\prime}}^{2} \zeta_{2}\right)^{2}+\left(\partial_{x^{\prime}}^{2} \zeta_{3}\right)^{2}\right)\right|_{x_{1}=0} d x_{2} d x_{3} d \tau \\
= & -\left.\mu \sum_{i=2}^{3} \int_{0}^{t} \int_{\mathbb{T}^{2}}\left[\partial_{x^{\prime}}^{2}\left(\frac{1}{k\left(x^{\prime}\right)}\right) \zeta_{i} \partial_{x^{\prime}}^{2} \zeta_{i}+2 \partial_{x^{\prime}}\left(\frac{1}{k\left(x^{\prime}\right)}\right) \partial_{x^{\prime}} \zeta_{i} \partial_{x^{\prime}}^{2} \zeta_{i}\right]\right|_{x_{1}=0} d x_{2} d x_{3} d \tau \\
& +\int_{0}^{t} \int(3-\gamma) \frac{p^{\prime}(\rho)}{\rho} \operatorname{div} {\bf u} \frac{\left(\partial_{x^{\prime}}^{2} \phi\right)^{2}}{2} d x d \tau+\int_{0}^{t} \int p^{\prime \prime}(\rho) \partial_{x^{\prime}}^{2} \zeta \cdot \nabla \rho \partial_{x^{\prime}}^{2} \phi d x d \tau \\
& -\int_{0}^{t} \int\left(\left[\partial_{x^{\prime}}^{2}, {\bf u} \cdot \nabla\right] \phi+\left[\partial_{x^{\prime}}^{2}, \rho\right] \operatorname{div} \zeta\right) \frac{p^{\prime}(\rho)}{\rho} \partial_{x^{\prime}}^{2} \phi d x d \tau \\
 &- \int_{0}^{t} \int\left(\left[\partial_{x^{\prime}}^{2}, \rho\right] \partial_{\tau} \zeta+\left[\partial_{x^{\prime}}^{2}, \rho {\bf u} \cdot \nabla\right] \zeta+\left[\partial_{x^{\prime}}^{2}, p^{\prime}(\rho)\right] \nabla \phi\right) \cdot \partial_{x^{\prime}}^{2} \zeta d x d \tau \\
 &+ \int_{0}^{t} \int\left(\frac{p^{\prime}(\rho)}{\rho} \partial_{x^{\prime}}^{2} f \partial_{x^{\prime}}^{2} \phi+\partial_{x^{\prime}}^{2} g \cdot \partial_{x^{\prime}}^{2} \zeta\right) d x d \tau\\
&+\int_0^t \int_{\mathbb{T}^2}  (\mu+\lambda) \operatorname{div} \left(\partial^2_{x^{\prime}}\zeta\right)\partial^2_{x^{\prime}}\zeta_1 -\mu \partial_{x_1}\partial^2_{x^{\prime}}\zeta_1\partial^2_{x^{\prime}}\zeta_1 - \partial^2_{x^{\prime}}\zeta_1 {p^{\prime}(\rho)} \partial^2_{x^{\prime}} \phi\Big|_{x_1=0} d x_2 d x_3d \tau=\sum_{i=1}^{7} J_{i} .\\
\end{aligned}
\end{equation}
For $J_{1}$, we have
\begin{align}\label{eqj102}
\begin{aligned}
J_{1} \leq & \left.\frac{\mu}{2} \int_{0}^{t} \int_{\mathbb{T}^{2}} \frac{1}{k\left(x^{\prime}\right)}\left(\left(\partial_{x^{\prime}}^{2} \zeta_{2}\right)^{2}+\left(\partial_{x^{\prime}}^{2} \zeta_{3}\right)^{2}\right)\right|_{x_{1}=0} d x_{2} d x_{3} d \tau \\
& +\left.C \int_{0}^{t} \int_{\mathbb{T}^{2}} \frac{1}{k\left(x^{\prime}\right)}\left(\left|\zeta^{\prime}\right|^{2}+\left|\partial_{x^{\prime}} \zeta^{\prime}\right|^{2}\right)\right|_{x_{1}=0} d x_{2} d x_{3} d \tau .
\end{aligned}
\end{align}
Similar to \eqref{eqj71}, it yields 
\begin{align}\label{eqj103}
J_{2} \leq C(\chi+\delta) \int_{0}^{t}\left(\left\|\partial_{x^{\prime}}^{2} \phi\right\|^{2}+\left\|\nabla^{2} \operatorname{div} \zeta\right\|^{2}\right) d \tau.
\end{align}
It follows from Cauchy's inequality, Sobolev's inequality and assumption (3.5) that
\begin{align}\label{eqj104}
\begin{aligned}
J_{3} & \leq C \int_{0}^{t}\left(\left\|\partial_{x^{\prime}}^{2} \zeta\right\|_{L^{6}}\|\nabla \phi\|_{L^{3}}\left\|\partial_{x^{\prime}}^{2} \phi\right\|+\left\|\p_{1} \bar{\rho}\right\|_{L^{\infty}}\left\|\partial_{x^{\prime}}^{2} \zeta_{1}\right\|\left\|\partial_{x^{\prime}}^{2} \phi\right\|\right) d \tau \\
& \leq C(\chi+\delta) \int_{0}^{t}\left(\left\|\partial_{x^{\prime}}^{2} \zeta\right\|_{H^{1}}^{2}+\left\|\partial_{x^{\prime}}^{2} \phi\right\|^{2}\right) d \tau.
\end{aligned}
\end{align}
Note that
$$
\begin{aligned}
& {\left[\partial_{x^{\prime}}^{2}, {\bf u} \cdot \nabla\right] \phi=\partial_{x^{\prime}}^{2}({\bf u} \cdot \nabla \phi)-{\bf u} \cdot \nabla \partial_{x^{\prime}}^{2} \phi=\partial_{x^{\prime}}^{2} {\bf u} \cdot \nabla \phi+2\partial_{x^{\prime}} {\bf u} \cdot \nabla \partial_{x^{\prime}} \phi} \\
& {\left[\partial_{x^{\prime}}^{2}, \rho\right] \operatorname{div} \zeta=\partial_{x^{\prime}}^{2}(\rho \operatorname{div} \zeta)-\rho \operatorname{div} \partial_{x^{\prime}}^{2} \zeta=\partial_{x^{\prime}}^{2} \rho \operatorname{div} \zeta+2\partial_{x^{\prime}} \rho \operatorname{div} \partial_{x^{\prime}} \zeta},
\end{aligned}
$$
by using Cauchy's inequality and Sobolev's inequality, one has 
\begin{align}\label{eqj100}
\begin{aligned}
J_{4} & \leq C \int_{0}^{t}\left(\left\|\nabla \partial_{x^{\prime}} \zeta\right\|_{L^{6}}\|\nabla \phi\|_{L^{3}}+\|\nabla \zeta\|_{L^{\infty}}\left\|\nabla \partial_{x^{\prime}} \phi\right\|\right)\left\|\partial_{x^{\prime}}^{2} \phi\right\| d x d \tau \\
& \leq C \chi \int_{0}^{t}\left(\left\|\nabla^{2} \zeta\right\|_{H^{1}}^{2}+\left\|\nabla \partial_{x^{\prime}} \phi\right\|^{2}\right) d \tau.
\end{aligned}
\end{align}
Similar to \eqref{eqj100}, it yields
\begin{align}\label{eqj105}
J_{5} \leq C \chi \int_{0}^{t}\left(\left\|\partial_{\tau} \zeta\right\|_{H^{1}}^{2}+\left\|\partial_{x^{\prime}}^{2} \zeta\right\|_{H^{1}}^{2}+\left\|\partial_{x^{\prime}} \phi\right\|_{H^{1}}^{2}\right) d \tau.
\end{align}
By Cauchy's inequality and \eqref{eqj101}, one gets 
\begin{align}\label{eqj106}
J_{6} \leq C\delta \int_{0}^{t}\left(\left\|\partial_{x^{\prime}} \phi\right\|_{H^{1}}^{2}+\left\|\partial_{x^{\prime}} \zeta\right\|_{H^{1}}^{2}\right) d \tau.
\end{align}
Similar like  \eqref{C100}, one gets that $J_7 \leq C \delta$. And substitution \eqref{eqj102}-\eqref{eqj106} into \eqref{eqj107} leads to \eqref{eqj74}, then the proof of \cref{Lem6} is completed.
\end{proof}
\begin{Lem}\label{Lem7}
 Let $T>0$ be a constant and $(\phi, \zeta) \in X(0, T)$ be the solution of \eqref{eq26} satisfying a priori assumption \eqref{C4} with suitably small $\chi+\delta$, we have for $t \in[0, T]$,
\begin{align}\label{eqj77}
\begin{aligned}
& \left\|\nabla^{2} \zeta(t)\right\|^{2}+\left\|\left.\partial_{x^{\prime}} \zeta^{\prime}\right|_{x_{1}=0}(t)\right\|_{L^{2}\left(\mathbb{T}^{2}\right)}^{2} \\
& \leq C\left(\left\|\partial_{t} \zeta(t)\right\|^{2}+\|\nabla(\phi, \zeta)(t)\|^{2}\right)+C\left\|\left.\zeta^{\prime}\right|_{x_{1}=0}(t)\right\|_{L^{2}\left(\mathbb{T}^{2}\right)}^{2}+C \chi\left\|\left(\phi, \zeta_{1}\right)(t)\right\|^{2}+C\delta.
\end{aligned}
\end{align}
\end{Lem}
\begin{proof}
    We first note the following two facts:
\begin{align*}
\begin{aligned}
& \mu \int|\Delta \zeta|^{2} d x 
= \mu \int\left(\left|\p_{1}^{2} \zeta\right|^{2}+\left|\partial_{x_{2}}^{2} \zeta\right|^{2}+\left|\partial_{x_{3}}^{2} \zeta\right|^{2}\right) d x  +2 \mu \int\left(\partial_{x_{2}}^{2} \zeta \cdot \partial_{x_{3}}^{2} \zeta+\p_{1}^{2} \zeta \cdot \partial_{x_{2}}^{2} \zeta+\p_{1}^{2} \zeta \cdot \partial_{x_{3}}^{2} \zeta\right) d x \\
= & \mu \int\left(\left|\p_{1}^{2} \zeta\right|^{2}+\left|\partial_{x_{2}}^{2} \zeta\right|^{2}+\left|\partial_{x_{3}}^{2} \zeta\right|^{2}+2\left|\partial_{x_{2}} \partial_{x_{3}} \zeta\right|^{2}\right) d x  -2 \mu \int\left(\p_{1}^{2} \partial_{x_{2}} \zeta \cdot \partial_{x_{2}} \zeta+\p_{1}^{2} \partial_{x_{3}} \zeta \cdot \partial_{x_{3}} \zeta\right) d x \\
= & \mu\left\|\nabla^{2} \zeta\right\|^{2}+\left.2 \mu \int_{\mathbb{T}^{2}}\left(\p_{1} \partial_{x_{2}} \zeta \cdot \partial_{x_{2}} \zeta+\p_{1} \partial_{x_{3}} \zeta \cdot \partial_{x_{3}} \zeta\right)\right|_{x_{1}=0} d x_{2} d x_{3} \\
= & \mu\left\|\nabla^{2} \zeta\right\|^{2}+\left.2 \mu \int_{\mathbb{T}^{2}} \frac{1}{k\left(x^{\prime}\right)}\left|\partial_{x^{\prime}} \zeta^{\prime}\right|^{2}\right|_{x_{1}=0} d x_{2} d x_{3}\\
&+\left.2 \mu \sum_{i=2}^{3} \int_{\mathbb{T}^{2}}\left[\partial_{x_{2}}\left(\frac{1}{k\left(x^{\prime}\right)}\right) \partial_{x_{2}} \zeta_{i} \zeta_{i}+\partial_{x_{3}}\left(\frac{1}{k\left(x^{\prime}\right)}\right) \partial_{x_{3}} \zeta_{i} \zeta_{i}\right]\right|_{x_{1}=0} d x_{2} d x_{3},
\end{aligned}
\end{align*}
\begin{align*}
\begin{aligned}
& (\mu+\lambda) \int \Delta \zeta \cdot \nabla \operatorname{div} \zeta d x 
=  (\mu+\lambda) \int \operatorname{div}(\Delta \zeta \operatorname{div} \zeta-\nabla \operatorname{div} \zeta \operatorname{div} \zeta) d x+(\mu+\lambda)\|\nabla \operatorname{div} \zeta\|^{2} \\
= & (\mu+\lambda)\|\nabla \operatorname{div} \zeta\|^{2}-\left.(\mu+\lambda) \int_{\mathbb{T}^{2}}\left(\Delta \zeta_{1}-\p_{1} \operatorname{div} \zeta\right) \operatorname{div} \zeta\right|_{x_{1}=0} d x_{2} d x_{3} \\
= & (\mu+\lambda)\|\nabla \operatorname{div} \zeta\|^{2}+\left.(\mu+\lambda) \int_{\mathbb{T}^{2}}\left(\partial_{x_{2}} \p_{1} \zeta_{2}+\partial_{x_{3}} \p_{1} \zeta_{3}\right) \operatorname{div} \zeta\right|_{x_{1}=0} d x_{2} d x_{3} \\
= & (\mu+\lambda)\|\nabla \operatorname{div} \zeta\|^{2}+\left.(\mu+\lambda) \int_{\mathbb{T}^{2}}\left(\partial_{x_{2}}\left(\frac{\zeta_{2}}{k\left(x^{\prime}\right)}\right)+\partial_{x_{3}}\left(\frac{\zeta_{3}}{k\left(x^{\prime}\right)}\right)\right) \operatorname{div} \zeta\right|_{x_{1}=0} d x_{2} d x_{3} \\
= & (\mu+\lambda)\|\nabla \operatorname{div} \zeta\|^{2}+\left.(\mu+\lambda) \int_{\mathbb{T}^{2}} \frac{1}{k\left(x^{\prime}\right)}\left|\nabla^{\prime} \cdot \zeta^{\prime}\right|^{2}\right|_{x_{1}=0} d x_{2} d x_{3} \\
& +\left.(\mu+\lambda) \int_{\mathbb{T}^{2}} \frac{1}{k\left(x^{\prime}\right)} \nabla^{\prime} \cdot \zeta^{\prime} \p_{1} \zeta_{1}\right|_{x_{1}=0} d x_{2} d x_{3}  +\left.(\mu+\lambda) \sum_{i=2}^{3} \int_{\mathbb{T}^{2}} \zeta_{i} \partial_{x_{i}}\left(\frac{1}{k\left(x^{\prime}\right)}\right) \operatorname{div} \zeta\right|_{x_{1}=0} d x_{2} d x_{3}.
\end{aligned}
\end{align*}
Then multiplying $\cref{eq12}_{2}$ by $-\Delta \zeta$ and then integrating the resulting equation with respect to $x$ over $\Omega$, it yields
\begin{align}\label{eqj112}
\begin{aligned}
 &\mu\left\|\nabla^{2} \zeta\right\|^{2}+(\mu+\lambda)\|\nabla \operatorname{div} \zeta\|^{2}+\left.2 \mu \int_{\mathbb{T}^{2}} \frac{1}{k\left(x^{\prime}\right)}\left|\partial_{x^{\prime}} \zeta^{\prime}\right|^{2}\right|_{x_{1}=0} d x_{2} d x_{3} \\
& +\left.(\mu+\lambda) \int_{\mathbb{T}^{2}} \frac{1}{k\left(x^{\prime}\right)}\left|\nabla^{\prime} \cdot \zeta^{\prime}\right|^{2}\right|_{x_{1}=0} d x_{2} d x_{3}
=  -\left.(\mu+\lambda) \int_{\mathbb{T}^{2}} \frac{1}{k\left(x^{\prime}\right)} \nabla^{\prime} \cdot \zeta^{\prime} \p_{1} \zeta_{1}\right|_{x_{1}=0} d x_{2} d x_{3} \\
&\qquad \qquad\qquad\qquad -\left.2 \mu \sum_{i=2}^{3} \int_{\mathbb{T}^{2}}\left[\partial_{x_{2}}\left(\frac{1}{k\left(x^{\prime}\right)}\right) \partial_{x_{2}} \zeta_{i} \zeta_{i}+\partial_{x_{3}}\left(\frac{1}{k\left(x^{\prime}\right)}\right) \partial_{x_{3}} \zeta_{i} \zeta_{i}\right]\right|_{x_{1}=0} d x_{2} d x_{3} \\
& \qquad \qquad\qquad\qquad-\left.(\mu+\lambda) \sum_{i=2}^{3} \int_{\mathbb{T}^{2}} \zeta_{i} \partial_{x_{i}}\left(\frac{1}{k\left(x^{\prime}\right)}\right) \operatorname{div} \zeta\right|_{x_{1}=0} d x_{2} d x_{3} \\
&\qquad \qquad\qquad\qquad+\int\left(p^{\prime}(\rho) \nabla \phi+\rho \partial_{t} \zeta+\rho {\bf u} \cdot \nabla \zeta-g\right) \cdot \Delta \zeta d x .
\end{aligned}
\end{align}
{According to} Cauchy's inequality and Sobolev's inequality, one has 
\begin{align}\label{eqj108}
\begin{aligned}
&-\left.(\mu+\lambda) \int_{\mathbb{T}^{2}} \frac{1}{k\left(x^{\prime}\right)} \nabla^{\prime} \cdot \zeta^{\prime} \p_{1} \zeta_{1}\right|_{x_{1}=0} d x_{2} d x_{3}
\leq  C \int_{\mathbb{T}^{2}}\left\|\nabla^{\prime} \cdot \zeta^{\prime}\right\|_{L^{\infty}(\mathbb R^+)}\left\|\p_{1} \zeta_{1}\right\|_{L^{\infty}(\mathbb R^+)} d x_{2} d x_{3} \\
&\qquad\qquad\qquad\qquad\leq  C \int_{\mathbb{T}^{2}}\left\|\nabla^{\prime} \cdot \zeta^{\prime}\right\|_{L^{2}(\mathbb R^+)}^{\frac{1}{2}}\left\|\p_{1} \nabla^{\prime} \cdot \zeta^{\prime}\right\|_{L^{2}(\mathbb R^+)}^{\frac{1}{2}}\left\|\p_{1} \zeta_{1}\right\|_{L^{2}(\mathbb R^+)}^{\frac{1}{2}}\left\|\p_{1}^{2} \zeta_{1}\right\|_{L^{2}(\mathbb R^+)}^{\frac{1}{2}} d x_{2} d x_{3}\\
&\qquad\qquad\qquad\qquad\leq C\left\|\nabla^{\prime} \cdot \zeta^{\prime}\right\|^{\frac{1}{2}}\left\|\p_{1} \nabla^{\prime} \cdot \zeta^{\prime}\right\|^{\frac{1}{2}}\left\|\p_{1} \zeta_{1}\right\|^{\frac{1}{2}}\left\|\p_{1}^{2} \zeta_{1}\right\|^{\frac{1}{2}} \leq \frac{\mu}{4}\left\|\nabla^{2} \zeta\right\|^{2}+C\|\nabla \zeta\|^{2}.
\end{aligned}
\end{align}
Consequently,
\begin{align}\label{eqj109}
\begin{aligned}
& -\left.2 \mu \sum_{i=2}^{3} \int_{\mathbb{T}^{2}}\left[\partial_{x_{2}}\left(\frac{1}{k\left(x^{\prime}\right)}\right) \partial_{x_{2}} \zeta_{i} \zeta_{i}+\partial_{x_{3}}\left(\frac{1}{k\left(x^{\prime}\right)}\right) \partial_{x_{3}} \zeta_{i} \zeta_{i}\right]\right|_{x_{1}=0} d x_{2} d x_{3} \\
\leq & \left.\frac{\mu}{2} \int_{\mathbb{T}^{2}} \frac{1}{k\left(x^{\prime}\right)}\left|\partial_{x^{\prime}} \zeta^{\prime}\right|^{2}\right|_{x_{1}=0} d x_{2} d x_{3}+\left.C \int_{\mathbb{T}^{2}}\left|\zeta^{\prime}\right|^{2}\right|_{x_{1}=0} d x_{2} d x_{3},
\end{aligned}
\end{align}
and
\begin{align}\label{eqj110}
\begin{aligned}
& -\left.(\mu+\lambda) \sum_{i=2}^{3} \int_{\mathbb{T}^{2}} \zeta_{i} \partial_{x_{i}}\left(\frac{1}{k\left(x^{\prime}\right)}\right) \operatorname{div} \zeta\right|_{x_{1}=0} d x_{2} d x_{3} \\
\leq & \left.\frac{\mu}{2} \int_{\mathbb{T}^{2}} \frac{1}{k\left(x^{\prime}\right)}\left|\partial_{x^{\prime}} \zeta^{\prime}\right|^{2}\right|_{x_{1}=0} d x_{2} d x_{3}+\left.C \int_{\mathbb{T}^{2}}\left|\zeta^{\prime}\right|^{2}\right|_{x_{1}=0} d x_{2} d x_{3} \\
& +C \int_{\mathbb{T}^{2}}\left|\zeta^{\prime}\right|_{x_{1}=0} \mid\left\|\p_{1} \zeta_{1}\right\|_{L^{\infty}(\mathbb R^+)} d x_{2} d x_{3} \\
\leq & \left.\frac{\mu}{2} \int_{\mathbb{T}^{2}} \frac{1}{k\left(x^{\prime}\right)}\left|\partial_{x^{\prime}} \zeta^{\prime}\right|^{2}\right|_{x_{1}=0} d x_{2} d x_{3}+\left.C \int_{\mathbb{T}^{2}}\left|\zeta^{\prime}\right|^{2}\right|_{x_{1}=0} d x_{2} d x_{3} \\
& +C \int_{\mathbb{T}^{2}}\left|\zeta^{\prime}\right|_{x_{1}=0} \mid\left\|\p_{1} \zeta_{1}\right\|_{L^{2}(\mathbb R^+)}^{\frac{1}{2}}\left\|\p_{1}^{2} \zeta_{1}\right\|_{L^{2}(\mathbb R^+)}^{\frac{1}{2}} d x_{2} d x_{3} \\
\leq & \left.\frac{\mu}{2} \int_{\mathbb{T}^{2}} \frac{1}{k\left(x^{\prime}\right)}\left|\partial_{x^{\prime}} \zeta^{\prime}\right|^{2}\right|_{x_{1}=0} d x_{2} d x_{3}+\frac{\mu}{4}\left\|\nabla^{2} \zeta\right\|^{2}+C\left(\left\|\left.\zeta^{\prime}\right|_{x_{1}=0}\right\|_{L^{2}\left(\mathbb{T}^{2}\right)}^{2}+\left\|\p_{1} \zeta_{1}\right\|^{2}\right) .
\end{aligned}
\end{align}
Then one gets
\begin{align}\label{eqj111}
\begin{aligned}
& \int\left(p^{\prime}(\rho) \nabla \phi+\rho \partial_{t} \zeta+\rho {\bf u} \cdot \nabla \zeta-g\right) \cdot \Delta \zeta d x \\
\leq & \frac{\mu}{4}\left\|\nabla^{2} \zeta\right\|^{2}+C\left(\|\nabla \phi\|^{2}+\|\nabla \zeta\|^{2}+\left\|\partial_{t} \zeta\right\|^{2}\right)+C \delta\left(\|\phi\|^{2}+\left\|\zeta_{1}\right\|^{2}\right)+
{C\delta}.
\end{aligned}
\end{align}
Substitution \eqref{eqj108}-\eqref{eqj111} into \eqref{eqj112} gives \eqref{eqj77}, and the proof of \cref{Lem7} is completed.
\end{proof}

Now, we derive the second order normal derivative estimates of $\phi$.

\begin{Lem}\label{Lem8}
Let $T>0$ be a constant and $(\phi, \zeta) \in X(0, T)$ be the solution of \eqref{eq26} satisfying a priori assumption \eqref{C4} with suitably small $\chi+\delta$, it holds that for $t \in[0, T]$,
\begin{align}\label{eqj11}
\begin{aligned}
& \left\|\partial_{{1}}^{2} \phi(t)\right\|^{2}+\int_{0}^{t}\left(\left\|\partial_{{1}}^{2} \phi\right\|^{2}+\left\|\partial_{{1}}^{3} \zeta_{1}\right\|^{2}\right) d \tau \\
\leq & C
\int_{0}^{t}\left(\left\|\partial_{\tau} \partial_{{1}} \zeta_{1}\right\|^{2}+\left\|\nabla \partial_{{1}} \zeta_{1}\right\|^{2}+\left\|\nabla \partial_{{1}} \partial_{x^{\prime}} \zeta\right\|^{2}\right) d \tau
{{+C\left(E^{2}(0)+M^{2}(t)+\delta\right)}}, \\
\end{aligned}
\end{align}
\begin{align}\label{eqj122}
\begin{aligned}
 \left\|\partial_{{1}} \partial_{x^{\prime}} \phi(t)\right\|^{2}+\int_{0}^{t}\left(\left\|\partial_{{1}} \partial_{x^{\prime}} \phi\right\|^{2}+\left\|\partial_{{1}}^{2} \partial_{x^{\prime}} \zeta_{1}\right\|^{2}\right) d \tau \leq C
 \int_0^t\|\nabla\zeta\|_{H^2}^2d\tau
 {{+C\left(E^{2}(0)+M^{2}(t)+\delta\right).}}
\end{aligned}
\end{align}
\end{Lem}
\begin{proof}
 Applying $\partial_{{1}}^{2}$ to \eqref{eq12}$_1$ and $\partial_{{1}}$ to \eqref{eq12}$_2$, one has
\begin{align}\label{eqj12}
\left\{\begin{aligned}
&\partial_{t} \partial_{{1}}^{2} \phi+{\bf u} \cdot \nabla \partial_{{1}}^{2} \phi+\rho \partial_{{1}}^{3} \zeta_{1}+\rho \partial_{{1}}^{2} \nabla^{\prime} \cdot \zeta^{\prime}=\partial_{{1}}^{2} f-\left[\partial_{{1}}^{2}, {\bf u} \cdot \nabla\right] \phi-\left[\partial_{{1}}^{2}, \rho\right] \operatorname{div} \zeta, \\
&\rho\left(\partial_{t} \partial_{{1}} \zeta_{1}+{\bf u} \cdot \nabla \partial_{{1}} \zeta_{1}\right)+p^{\prime}(\rho) \partial_{{1}}^{2} \phi-(2 \mu+\lambda) \partial_{{1}}^{3} \zeta_{1}-\mu \partial_{{1}} \Delta^{\prime} \zeta_{1}-(\mu+\lambda) \partial_{{1}}^{2} \nabla^{\prime} \cdot \zeta^{\prime}\\
&\qquad=\partial_{{1}} g_{1}-\partial_{{1}} \rho \partial_{t} \zeta_{1}-\partial_{{1}}(\rho {\bf u}) \cdot \nabla \zeta_{1} -p^{\prime \prime}(\rho) \partial_{{1}} \rho \partial_{{1}} \phi.
\end{aligned}\right.
\end{align}
Similar to lemma \ref{Lem4}, multiplying \eqref{eqj12}$_{1}$ by $\frac{1}{\rho} \partial_{{1}}^{2} \phi$ and \eqref{eqj12}$_2$ by $\frac{1}{2 \mu+\lambda} \partial_{{1}}^{2} \phi$, adding the resulted equations together and then integrating it with respect to $t, x$ over $(0, t) \times \Omega$, it yields
\begin{align}\label{eqj13}
\begin{aligned}
& \left.\int \frac{\left(\partial_{{1}}^{2} \phi\right)^{2}}{2\rho} d x\right|_{\tau=0} ^{\tau=t}+\int_{0}^{t} \int \frac{p^{\prime}(\rho)}{2 \mu+\lambda}\left(\partial_{{1}}^{2} \phi\right)^{2} d x d \tau \\
= & \int_{0}^{t} \int\left(\mu\left(\partial_{{1}} \Delta^{\prime} \zeta_{1}-\partial_{{1}}^{2} \nabla^{\prime} \cdot \zeta^{\prime}\right)-\rho\left(\partial_{\tau} \partial_{{1}} \zeta_{1}+{\bf u} \cdot \nabla \partial_{{1}} \zeta_{1}\right)\right) \frac{\partial_{{1}}^{2} \phi}{2 \mu+\lambda} d x d \tau \\
& +\int_{0}^{t} \int\left(\frac{\operatorname{div} {\bf u}}{\rho}\left(\partial_{{1}}^{2} \phi\right)^{2}-\frac{1}{\rho} \partial_{{1}}^{2} \phi\left[\partial_{{1}}^{2}, {\bf u} \cdot \nabla\right] \phi-\frac{1}{\rho} \partial_{{1}}^{2} \phi\left[\partial_{{1}}^{2}, \rho\right] \operatorname{div} \zeta\right) d x d \tau \\
& -\int_{0}^{t} \int\left[\partial_{{1}} \rho \partial_{\tau} \zeta_{1}+\partial_{{1}}(\rho {\bf u}) \cdot \nabla \zeta_{1}+p^{\prime \prime}(\rho) \partial_{{1}} \rho \partial_{{1}} \phi\right] \frac{\partial_{{1}}^{2} \phi}{2 \mu+\lambda} d x d \tau \\
& +\int_{0}^{t} \int\left(\frac{1}{\rho} \partial_{{1}}^{2} f \partial_{{1}}^{2} \phi+\frac{1}{2 \mu+\lambda} \partial_{{1}} g_{1} \partial_{{1}}^{2} \phi\right) d x d \tau:=\sum_{i=8}^{11} J_{i} .
\end{aligned}
\end{align}
By Cauchy's inequality, one has
\begin{align}\label{eqj14}
\begin{aligned}
J_{8} \leq & \frac{1}{16(2 \mu+\lambda)} \int_{0}^{t}\left\|\sqrt{p^{\prime}(\rho)} \partial_{{1}}^{2} \phi\right\|^{2} d \tau +C \int_{0}^{t}\left(\left\|\nabla \partial_{{1}} \partial_{x^{\prime}} \zeta\right\|^{2}+\left\|\partial_{\tau} \partial_{{1}} \zeta_{1}\right\|^{2}+\left\|\nabla \partial_{{1}} \zeta_{1}\right\|^{2}\right) d \tau.
\end{aligned}
\end{align}
Note that
\begin{align}\label{eqj15}
\begin{gathered}
{\left[\partial_{{1}}^{2}, {\bf u} \cdot \nabla\right] \phi=\partial_{{1}}^{2}({\bf u} \cdot \nabla \phi)-{\bf u} \cdot \nabla \partial_{{1}}^{2} \phi=\partial_{{1}}^{2} {\bf u} \cdot \nabla \phi+2\partial_{{1}} {\bf u} \cdot \nabla \partial_{{1}} \phi}, \\
{\left[\partial_{{1}}^{2}, \rho\right] \operatorname{div} \zeta=\partial_{{1}}^{2}(\rho \operatorname{div} \zeta)-\rho \operatorname{div} \partial_{{1}}^{2} \zeta=\partial_{{1}}^{2} \rho \operatorname{div} \zeta+2\partial_{{1}} \rho \operatorname{div} \partial_{{1}} \zeta}.
\end{gathered}
\end{align}
Then it holds
\begin{align}\label{eqj17}
\begin{aligned}
& \frac{\operatorname{div} {\bf u}}{\rho}\left(\partial_{{1}}^{2} \phi\right)^{2}-\frac{1}{\rho} \partial_{{1}}^{2} \phi\left[\partial_{{1}}^{2}, {\bf u} \cdot \nabla\right] \phi-\frac{1}{\rho} \partial_{{1}}^{2} \phi\left[\partial_{{1}}^{2}, \rho\right] \operatorname{div} \zeta \\
= & \frac{1}{\rho} \partial_{{1}}^{2} \phi\left(\operatorname{div} {\bf u} \partial_{{1}}^{2} \phi-\partial_{{1}}^{2} {\bf u} \cdot \nabla \phi-2\partial_{{1}} {\bf u} \cdot \nabla \partial_{{1}} \phi-\partial_{{1}}^{2} \rho \operatorname{div} \zeta-2\partial_{{1}} \rho \operatorname{div} \partial_{{1}} \zeta\right) \\
= & -\frac{1}{\rho} \partial_{{1}}^{2} \phi\left(2\partial_{{1}} \zeta \cdot \nabla \partial_{{1}} \phi+\partial_{{1}}^{2} \zeta \cdot \nabla \phi+2\partial_{{1}} \phi \operatorname{div} \partial_{{1}} \zeta+\p_1^2\phi\dv\zeta-\dv\zeta
\p_1^2\phi\right) \\
& -\frac{1}{\rho} \partial_{{1}}^{2} \phi\left(\partial_{{1}} \bar{\rho} \operatorname{div} \partial_{{1}} \zeta+\partial_{{1}}^{2} \bar{u}_{1} \partial_{{1}}
\phi+\partial_{{1}}^{2} \bar{\rho} \operatorname{div} \zeta+
\p_1\bar u_1\p_1^2
\phi
\right).
\end{aligned}
\end{align}
By Cauchy's inequality and assumption \eqref{C4}, similar to \eqref{eqj71}, it yields
\begin{align}\label{eqj18}
\begin{aligned}
J_{9} \leq & C\int_{0}^{t}\left[\left\|\partial_{{1}} \zeta\right\|_{L^{\infty}}\left\|\nabla \partial_{{1}} \phi\right\|^{2}+\left\|\partial_{{1}}^{2} \phi\right\|\left\|\nabla^{2} \zeta\right\|_{L^{6}}\|\nabla \phi\|_{L^{3}}+\delta\left\|\partial_{{1}}^{2} \phi\right\|\left(\left\|\nabla^{2} \zeta\right\|+\|\nabla(\phi, \zeta)\|\right)\right] d \tau \\
\leq & C(\chi+\delta) \int_{0}^{t}\left(\left\|\nabla \partial_{{1}} \phi\right\|^{2}+\left\|\nabla^{2} \zeta\right\|_{H^{1}}^{2}+\|\nabla(\phi, \zeta)\|^{2}\right) d \tau \\
\leq &C(\chi+\delta) \int_{0}^{t}\left(\left\|\nabla \phi\right\|_{H^1}^{2}+\left\|\nabla \zeta\right\|_{H^{2}}^{2}\right) d \tau.
\end{aligned}
\end{align}
Similar to \eqref{eqj18}, 
from Cauchy's inequality and assumption \eqref{C4}, we know that
\begin{align}\label{eqj19}
J_{10} \leq C(\chi+\delta) \int_{0}^{t}\left(\left\|\partial_{{1}} \phi\right\|_{H^{1}}^{2}+\left\|\partial_{\tau} \zeta\right\|_{H^{1}}^{2}+\|\nabla \zeta\|_{H^{1}}^{2}\right) d \tau. 
\end{align}
By Cauchy's inequality and Lemma \ref{lem1}, one has
\begin{align}\label{eqj20}
J_{11} \leq C \delta \int_{0}^{t}\left(\left\|\partial_{{1}}^{2}\left(\phi, \zeta_{1}\right)\right\|^{2}+\left\|\partial_{{1}}\left(\phi, \zeta_{1}\right)\right\|^{2}\right) d \tau+{ C\delta}.
\end{align}

Therefore, substituting \eqref{eqj14}, \eqref{eqj18}-\eqref{eqj20} into \eqref{eqj13} and choosing $\chi+\delta$ suitably small, one obtains
\begin{align}\label{eqj21}
\begin{aligned}
& \left\|\partial_{{1}}^{2} \phi(t)\right\|^{2}+\int_{0}^{t}\left\|\partial_{{1}}^{2} \phi\right\|^{2} d \tau \\
&\leq  C\|\p_1^2\phi_0\|^2+C \int_{0}^{t}\left(\left\|\partial_{\tau} \partial_{{1}} \zeta_{1}\right\|^{2}+\left\|\nabla \partial_{{1}} \zeta_{1}\right\|^{2}+\left\|\nabla \partial_{{1}} \partial_{x^{\prime}} \zeta\right\|^{2}\right) d \tau\\
&\qquad+C(\chi+\delta)\int_0^t\left(\|\nabla\zeta\|^2_{H^2}+\|\p_\tau\zeta\|_{H^1}^2\right)d\tau+C\delta\int_0^t\left(\|\p_1^2(\phi,\zeta_1)\|^2+\|\p_1(\phi,\zeta_1)\|^2\right)d\tau\\
&\le C\|\p_1\phi_0\|^2_{H^1}+C \int_{0}^{t}\left(\left\|\partial_{\tau} \partial_{{1}} \zeta_{1}\right\|^{2}+\left\|\nabla \partial_{{1}} \zeta_{1}\right\|^{2}+\left\|\nabla \partial_{{1}} \partial_{x^{\prime}} \zeta\right\|^{2}\right) d \tau\\
&\qquad+C(\chi+\delta)\int_0^t\left(\|\nabla\zeta\|^2_{H^2}+\|\nabla\phi\|^2+\|\p_\tau\zeta\|_{H^1}^2\right)d\tau+C\delta\int_0^t
\|(\phi,\p_1\Phi,\p_{x_2}\Psi)\|^2d\tau .
\end{aligned}
\end{align}
It follows from \eqref{eqj12}$_{2}$ that
\begin{align}\label{eqj22}
\begin{aligned}
 \int_{0}^{t}\left\|\partial_{{1}}^{3} \zeta_{1}\right\|^{2}
\leq  C \int_{0}^{t}\left(\left\|\partial_{{1}}^{2} \phi\right\|^{2}+\left\|\partial_{\tau} \partial_{{1}} \zeta\right\|^{2}+\left\|\nabla \partial_{{1}} \zeta\right\|^{2}+\left\|\nabla \partial_{{1}} \partial_{x^{\prime}} \zeta\right\|^{2}\right) d \tau+C\delta.
\end{aligned}
\end{align}
Then Multiplying \eqref{eqj21} by a large constant $C$ and combining with \eqref{eqj22}, we have \eqref{eqj11}.

Next, applying $\partial_{x^{\prime}}$ to \eqref{eqj1}, it yields
\begin{align}\label{eqj23}
\left\{\begin{aligned}
&\partial_{t} \partial_{{1}} \partial_{x^{\prime}} \phi+{\bf u} \cdot \nabla \partial_{{1}} \partial_{x^{\prime}} \phi+\rho \partial_{{1}}^{2} \partial_{x^{\prime}} \zeta_{1}+\rho \partial_{{1}} \partial_{x^{\prime}} \nabla^{\prime} \cdot \zeta^{\prime} =\partial_{{1}} \partial_{x^{\prime}} f-\left[\partial_{{1}} \partial_{x^{\prime}}, {\bf u} \cdot \nabla\right] \phi-\left[\partial_{{1}} \partial_{x^{\prime}}, \rho\right] \operatorname{div} \zeta, \\
&\rho\left(\partial_{t} \partial_{x^{\prime}} \zeta_{1}+{\bf u} \cdot \nabla \partial_{x^{\prime}} \zeta_{1}\right)+p^{\prime}(\rho) \partial_{{1}} \partial_{x^{\prime}} \phi-(2 \mu+\lambda) \partial_{{1}}^{2} \partial_{x^{\prime}} \zeta_{1}-\mu \partial_{x^{\prime}} \Delta^{\prime} \zeta_{1}-(\mu+\lambda) \partial_{{1}} \partial_{x^{\prime}} \nabla^{\prime} \cdot \zeta^{\prime}  \\
&\quad=\partial_{x^{\prime}} g_{1}-\partial_{x^{\prime}} \rho \partial_{t} \zeta_{1}-\partial_{x^{\prime}}(\rho {\bf u}) \cdot \nabla \zeta_{1}-p^{\prime \prime}(\rho) \partial_{x^{\prime}} \rho \partial_{{1}} \phi.
\end{aligned}\right.
\end{align}
Multiplying \eqref{eqj23}$_{1}$ by $\frac{1}{\rho} \partial_{{1}} \partial_{x^{\prime}} \phi$ and \eqref{eqj23}$_2$ by $\frac{1}{2 \mu+\lambda} \partial_{{1}} \partial_{x^{\prime}} \phi$, then adding the resulted equations together and integrating it with respect to $t, x$ over $(0, t) \times \Omega$, similar to \eqref{eqj21}, it yields
\begin{align}\label{eqj24}
\begin{aligned}
& \left\|\partial_{{1}} \partial_{x^{\prime}} \phi(t)\right\|^{2}+\int_{0}^{t}\left\|\partial_{{1}} \partial_{x^{\prime}} \phi\right\|^{2} d \tau \\
\leq & C\left(\left\|\partial_{{1}} \partial_{x^{\prime}} \phi_0\right\|^{2}+\|\p_1\phi_0\|^2\right)+C \int_{0}^{t}\left(\left\|\partial_{\tau} \partial_{x^{\prime}} \zeta_{1}\right\|^{2}+\left\|\nabla \partial_{x^{\prime}} \zeta_{1}\right\|^{2}+\left\|\nabla \partial_{x^{\prime}}^{2} \zeta\right\|^{2}\right) d \tau\\
&\qquad+C(\chi+\delta)\int_0^t\left(\|\nabla\zeta\|^2_{H^2}+\|\nabla\phi\|^2+\|\p_\tau\zeta\|_{H^1}^2\right)d\tau+C\delta\int_0^t
\|(\phi,\p_1\Phi,\p_{x_2}\Psi)\|^2d\tau. 
\end{aligned}
\end{align}
From \eqref{eqj23}$_{2}$, it is direct to know that
\begin{align}\label{eqj25}
\begin{aligned}
 \int_{0}^{t}\left\|\partial_{{1}}^{2} \partial_{x^{\prime}} \zeta_{1}\right\|^{2} d \tau
\leq  C \int_{0}^{t}\left(\left\|\partial_{{1}} \partial_{x^{\prime}} \phi\right\|^{2}+\left\|\partial_{\tau} \partial_{x^{\prime}} \zeta_{1}\right\|^{2}+\left\|\nabla \partial_{x^{\prime}} \zeta_{1}\right\|^{2}+\left\|\nabla \partial_{x^{\prime}}^{2} \zeta\right\|^{2}\right) d \tau+C\delta.
\end{aligned}
\end{align}

Multiplying \eqref{eqj24} by a large positive constant $C$, combining with \eqref{eqj25}, and 
by {Lemma \ref{Lem2} and Lemma \ref{Lem6}}, we have \eqref{eqj122}. Hence the proof of \cref{Lem8} is completed.
\end{proof}

To close the a priori assumption \eqref{C4}, we need to derive the higher order tangential derivative estimates of $\phi$.

\begin{Lem}\label{Lem9}
 Let $T>0$ be a constant and $(\phi, \zeta) \in X(0, T)$ be the solution of \eqref{eq26} satisfying a priori assumption \eqref{C4} with suitably small $\chi+\delta$, it holds that for $t \in[0, T]$,
\begin{align}\label{eqj40}
\begin{aligned}
& \int_{0}^{t}\left(\left\|\partial_{x^{\prime}}^{2} \phi\right\|^{2}+\left\|\partial_{{1}}^{2} \partial_{x^{\prime}} \zeta^{\prime}\right\|^{2}+\left\|\partial_{{1}}^{3} \zeta^{\prime}\right\|^{2}\right) d \tau \\
\leq & \int_{0}^{t}\left(\left\|\partial_{x^{\prime}} \phi\right\|^{2}+\left\|\nabla \partial_{\tau} \zeta^{\prime}\right\|^{2}+\left\|\nabla^{2} \zeta^{\prime}\right\|^{2}\right) d \tau{+C\left(E^{2}(0)+M^{2}(t)+\delta
\right)}.
\end{aligned}
\end{align}
\end{Lem}
\begin{proof}
Applying $\partial_{x^{\prime}}$ to \eqref{eqj37} yields
\begin{align}\label{eqj41}
\begin{aligned}
& \rho\left(\partial_{t} \partial_{x^{\prime}} \zeta_{2}+{\bf u} \cdot \nabla \partial_{x^{\prime}} \zeta_{2}\right)+p^{\prime}(\rho) \partial_{x_{2}} \partial_{x^{\prime}} \phi+\partial_{x^{\prime}} \rho \partial_{t} \zeta_{2}+\partial_{x^{\prime}}(\rho {\bf u}) \cdot \nabla \zeta_{2} \\
& \quad+p^{\prime \prime}(\rho) \partial_{x^{\prime}} \rho \partial_{x_{2}} \phi-\mu \partial_{{1}}^{2} \partial_{x^{\prime}} \zeta_{2}-\mu \Delta^{\prime} \partial_{x^{\prime}} \zeta_{2}-(\mu+\lambda) \partial_{x_{2}} \partial_{x^{\prime}} \operatorname{div} \zeta=0.
\end{aligned}
\end{align}
Multiplying the above equation by $\partial_{x_{2}} \partial_{x^{\prime}} \phi$ and  integrating the resulted equation with respect to $t, x$ over $(0, t) \times \Omega$, one has
\begin{align}\label{eqj42}
\begin{aligned}
 \int_{0}^{t} \int p^{\prime}(\rho)&\left(\partial_{x_{2}} \partial_{x^{\prime}} \phi\right)^{2} d x d \tau
=  \int_{0}^{t} \int \mu \partial_{{1}}^{2} \partial_{x^{\prime}} \zeta_{2} \partial_{x_{2}} \partial_{x^{\prime}} \phi d x d \tau \\
& +\int_{0}^{t} \int\left(\mu \Delta^{\prime} \partial_{x^{\prime}} \zeta_{2}+(\mu+\lambda) \partial_{x_{2}} \partial_{x^{\prime}} \operatorname{div} \zeta-\rho\left(\partial_{\tau} \partial_{x^{\prime}} \zeta_{2}+{\bf u} \cdot \nabla \partial_{x^{\prime}} \zeta_{2}\right)\right) \partial_{x_{2}} \partial_{x^{\prime}} \phi d x d \tau \\
& -\int_{0}^{t} \int\left(\partial_{x^{\prime}} \rho \partial_{\tau} \zeta_{2}+\partial_{x^{\prime}}(\rho {\bf u}) \cdot \nabla \zeta_{2}+p^{\prime \prime}(\rho) \partial_{x^{\prime}} \rho \partial_{x_{2}} \phi\right) \partial_{x_{2}} \partial_{x^{\prime}} \phi d x d \tau:=\sum_{i=3}^{5}\hat I_i .
\end{aligned}
\end{align}
Integration by parts under the boundary conditions \eqref{eq28}, it yields 
\begin{align}\label{eqj43}
\begin{aligned}
\hat I_3
=&  -\mu \int_{0}^{t} \int \partial_{{1}}^{2} \partial_{x_{2}} \partial_{x^{\prime}} \zeta_{2} \partial_{x^{\prime}} \phi d x d \tau \\
= & \left.\mu \int_{0}^{t} \int_{\mathbb{T}^{2}} \partial_{{1}} \partial_{x_{2}} \partial_{x^{\prime}} \zeta_{2} \partial_{x^{\prime}} \phi\right|_{{x_1}=0} d x_{2} d x_{3} d \tau+\mu \int_{0}^{t} \int \partial_{{1}} \partial_{x_{2}} \partial_{x^{\prime}} \zeta_{2} \p_{1} \partial_{x^{\prime}} \phi d x d \tau \\
= & \left.\mu \int_{0}^{t} \int_{\mathbb{T}^{2}} \partial_{x_{2}} \partial_{x^{\prime}}\left(\frac{\zeta_{2}}{k\left(x^{\prime}\right)}\right) \partial_{x^{\prime}} \phi\right|_{x_{1}=0} d x_{2} d x_{3} d \tau+\mu \int_{0}^{t} \int \partial_{{1}} \partial_{x_{2}} \partial_{x^{\prime}} \zeta_{2} \partial_{{1}} \partial_{x^{\prime}} \phi d x d \tau,
\end{aligned}
\end{align}
where 
\begin{align}\label{eqj44}
\begin{aligned}
& \left.\quad \mu \int_{0}^{t} \int_{\mathbb{T}^{2}} \partial_{x_{2}} \partial_{x^{\prime}}\left(\frac{\zeta_{2}}{k\left(x^{\prime}\right)}\right) \partial_{x^{\prime}} \phi\right|_{x_{1}=0} d x_{2} d x_{3} d \tau \\
& \leq C \int_{0}^{t} \int_{\mathbb{T}^{2}}\left[\left.\left.\left|\partial_{x_{2}} \partial_{x^{\prime}} \zeta_{2}\right|_{x_{1}=0}|+| \partial_{x^{\prime}} \zeta_{2}\right|_{x_{1}=0}|+| \zeta_{2}\right|_{x_{1}=0} \mid\right]\left\|\partial_{x^{\prime}} \phi\right\|_{L^{\infty}(\mathbb R^+)} d x_{2} d x_{3} d \tau \\
& \leq C \int_{0}^{t} \int_{\mathbb{T}^{2}}\left[\left.\left.\left|\partial_{x_{2}} \partial_{x^{\prime}} \zeta_{2}\right|_{x_{1}=0}|+| \partial_{x^{\prime}} \zeta_{2}\right|_{x_{1}=0}|+| \zeta_{2}\right|_{x_{1}=0} \mid\right] \\
&\qquad\cdot\left\|\partial_{x^{\prime}} \phi\right\|_{L^{2}(\mathbb R^+)}^{\frac{1}{2}}|| \partial_{{1}} \partial_{x^{\prime}} \phi \|_{L^{2}(\mathbb R^+)}^{\frac{1}{2}} d x_{2} d x_{3} d \tau \\
& \leq C \int_{0}^{t}\left[\left\|\left.\partial_{x_{2}} \partial_{x^{\prime}} \zeta_{2}\right|_{x_{1}=0}\right\|_{L^{2}\left(\mathbb{T}^{2}\right)}+\left\|\left.\partial_{x^{\prime}} \zeta_{2}\right|_{x_{1}=0}\right\|_{L^{2}\left(\mathbb{T}^{2}\right)}+\left\|\left.\zeta_{2}\right|_{x_{1}=0}\right\|_{L^{2}\left(\mathbb{T}^{2}\right)}\right] \\
&\qquad\cdot\left\|\partial_{x^{\prime}} \phi\right\|_{L^{2}(\Omega)}^{\frac{1}{2}}\left\|\partial_{{1}} \partial_{x^{\prime}} \phi\right\|_{L^{2}(\Omega)}^{\frac{1}{2}} d \tau\\
&\leq C \int_{0}^{t}\left[\left\|\partial_{x^{\prime}} \phi\right\|^{2}+\left\|\partial_{{1}} \partial_{x^{\prime}} \phi\right\|^{2}+\left\|\left.\partial_{x_{2}} \partial_{x^{\prime}} \zeta_{2}\right|_{x_{1}=0}\right\|_{L^{2}\left(\mathbb{T}^{2}\right)}^{2}+\left\|\left.\partial_{x^{\prime}} \zeta_{2}\right|_{x_{1}=0}\right\|_{L^{2}\left(\mathbb{T}^{2}\right)}^{2}\right.\\
&\left.\qquad+\left\|\left.\zeta_{2}\right|_{x_{1}=0}\right\|_{L^{2}\left(\mathbb{T}^{2}\right)}^{2}\right] d \tau.
\end{aligned}
\end{align}
substituting the above inequality into \eqref{eqj43} 
and combining with Cauchy's inequality, one can obtain
\begin{align}\label{eqj45}
\begin{aligned}
\hat I_3\leq & C \int_{0}^{t}\left(\left\|\partial_{x^{\prime}} \phi\right\|^{2}+\left\|\partial_{{1}} \partial_{x^{\prime}} \phi\right\|^{2}+\left\|\partial_{{1}} \partial_{x_{2}} \partial_{x^{\prime}} \zeta_{2}\right\|^{2}+\left\|\left.\partial_{x_{2}} \partial_{x^{\prime}} \zeta_{2}\right|_{x_{1}=0}\right\|_{L^{2}\left(\mathbb{T}^{2}\right)}^{2}\right. \\
& \left.+\left\|\left.\partial_{x^{\prime}} \zeta_{2}\right|_{x_{1}=0}\right\|_{L^{2}\left(\mathbb{T}^{2}\right)}^{2}+\left\|\left.\zeta_{2}\right|_{x_{1}=0}\right\|_{L^{2}\left(\mathbb{T}^{2}\right)}^{2}\right) d \tau .
\end{aligned}
\end{align}
By Cauchy's inequality, it yields
\begin{align}\label{eqj46}
\begin{aligned}
\hat I_4
\leq & \frac{1}{16} \int_{0}^{t}\left\|\sqrt{p^{\prime}(\rho)} \partial_{x_{2}} \partial_{x^{\prime}} \phi\right\|^{2} d \tau+C \int_{0}^{t}\left(\left\|\partial_{\tau} \partial_{x^{\prime}} \zeta_{2}\right\|^{2}+\left\|\nabla \partial_{x^{\prime}} \zeta_{2}\right\|^{2}+\left\|\nabla \partial_{x^{\prime}}^{2} \zeta\right\|^{2}\right) d \tau.
\end{aligned}
\end{align}
Similar to \eqref{eqj14} and \eqref{eqj18}, 
from Sobolev's inequality and the assumption \eqref{C4}, we know 
\begin{align}\label{eqj47}
\begin{aligned}
\hat I_5
\leq & \frac{1}{16} \int_{0}^{t}\left\|\sqrt{p^{\prime}(\rho)} \partial_{x_{2}} \partial_{x^{\prime}} \phi\right\|^{2} d \tau+C \chi \int_{0}^{t}\left(\left\|\partial_{\tau} \zeta_{2}\right\|_{H^{1}}^{2}+\left\|\nabla \zeta_{2}\right\|_{H^{1}}^{2}+\left\|\partial_{x_{2}} \phi\right\|_{H^{1}}^{2}\right) d \tau. 
\end{aligned}
\end{align}
Substituting \eqref{eqj45}-\eqref{eqj47} into \eqref{eqj42}, one has 
\begin{align}\label{eqj48}
\begin{aligned}
 \int_{0}^{t}\left\|\partial_{x_{2}} \partial_{x^{\prime}} \phi\right\|^{2} &d \tau
\leq  C \int_{0}^{t}\left(\left\|\partial_{\tau} \partial_{x^{\prime}} \zeta_{2}\right\|^{2}+\left\|\nabla \partial_{x^{\prime}} \zeta_{2}\right\|^{2}+\left\|\nabla \partial_{x^{\prime}}^{2} \zeta\right\|^{2}\right) d \tau\\
&+C \int_{0}^{t}\left(\left\|\left.\partial_{x_{2}} \partial_{x^{\prime}} \zeta_{2}\right|_{x_{1}=0}\right\|_{L^{2}\left(\mathbb{T}^{2}\right)}^{2}
+\left\|\left.\partial_{x^{\prime}} \zeta_{2}\right|_{x_{1}=0}\right\|_{L^{2}\left(\mathbb{T}^{2}\right)}^{2}+\left\|\left.\zeta_{2}\right|_{x_{1}=0}\right\|_{L^{2}\left(\mathbb{T}^{2}\right)}^{2}\right) d \tau\\
&+C \int_{0}^{t}\left(\left\|\p_{1} \partial_{x^{\prime}} \phi\right\|^{2}+\left\|\partial_{x^{\prime}} \phi\right\|^{2}\right) d \tau  {{+C \chi \int_{0}^{t}\left(\left\|\partial_{\tau} \zeta_{2}\right\|_{H^{1}}^{2}+\left\|\nabla \zeta_{2}\right\|_{H^{1}}^{2}+\left\|\partial_{x_{2}} \phi\right\|_{H^{1}}^{2}\right) d \tau.}}
\end{aligned}
\end{align}
It follows from \eqref{eqj41} that
\begin{align}\label{eqj49}
\begin{aligned}
 \int_{0}^{t}\left\|\partial_{{1}}^{2} \partial_{x^{\prime}} \zeta_{2}\right\|^{2} d \tau
\leq & C \int_{0}^{t}\left(\left\|\partial_{x_{2}} \partial_{x^{\prime}} \phi\right\|^{2}+\left\|\partial_{\tau} \partial_{x^{\prime}} \zeta_{2}\right\|^{2}+\left\|\nabla \partial_{x^{\prime}} \zeta_{2}\right\|^{2}+\left\|\nabla \partial_{x^{\prime}}^{2} \zeta\right\|^{2}\right) d \tau\\
&+C \chi \int_{0}^{t}\left(\left\|\partial_{\tau} \zeta_{2}\right\|^{2}+\left\|\nabla \zeta_{2}\right\|^{2}+\left\|\partial_{x_{2}} \phi\right\|^{2}\right) d \tau.
\end{aligned}
\end{align}
Then multiplying \eqref{eqj48} by a large constant $C$ and combining 
with \eqref{eqj49}, it implies
\begin{align}\label{eqj50}
\begin{aligned}
\int_{0}^{t}&\left(\left\|\partial_{x_{2}} \partial_{x^{\prime}} \phi\right\|^{2}+\left\|\partial_{{1}}^{2} \partial_{x^{\prime}} \zeta_{2}\right\|^{2}\right) d \tau
\leq  C \int_{0}^{t}\left(\left\|\partial_{{1}} \partial_{x^{\prime}} \phi\right\|^{2}+\left\|\partial_{x^{\prime}} \phi\right\|^{2}\right) d \tau\\
&\qquad\qquad+C \chi \int_{0}^{t}\left(\left\|\partial_{\tau} \zeta_{2}\right\|_{H^{1}}^{2}+\left\|\nabla \zeta_{2}\right\|_{H^{1}}^{2}+\left\|\partial_{x_{2}} \phi\right\|_{H^{1}}^{2}\right) d \tau
+C \int_{0}^{t}\left[\left\|\partial_{\tau} \partial_{x^{\prime}} \zeta_{2}\right\|^{2}+\left\|\nabla \partial_{x^{\prime}} \zeta_{2}\right\|^{2}\right.\\
&\qquad\qquad\left.+\left\|\nabla \partial_{x^{\prime}}^{2} \zeta\right\|^{2}+\left\|\left.\partial_{x_{2}} \partial_{x^{\prime}} \zeta_{2}\right|_{x_{1}=0}\right\|_{L^{2}\left(\mathbb{T}^{2}\right)}^{2}+\left\|\left.\partial_{x^{\prime}} \zeta_{2}\right|_{x_{1}=0}\right\|_{L^{2}\left(\mathbb{T}^{2}\right)}^{2}+\left\|\left.\zeta_{2}\right|_{x_{1}=0}\right\|_{L^{2}\left(\mathbb{T}^{2}\right)}^{2}\right] d \tau .
\end{aligned}
\end{align}
Applying $\partial_{{1}}$ to \eqref{eqj37} yields
\begin{align}\label{eqj51}
\begin{aligned}
& \rho\left(\partial_{t} \partial_{{1}} \zeta_{2}+{\bf u} \cdot \nabla \partial_{{1}} \zeta_{2}\right)+p^{\prime}(\rho) \partial_{{1}} \partial_{x_{2}} \phi+\partial_{{1}} \rho \partial_{t} \zeta_{2}+\partial_{{1}}(\rho {\bf u}) \cdot \nabla \zeta_{2}+p^{\prime \prime}(\rho) \partial_{{1}} \rho \partial_{x_{2}} \phi \\
& \quad-\mu \partial_{{1}}^{3} \zeta_{2}-\mu \Delta^{\prime} \partial_{{1}} \zeta_{2}-(\mu+\lambda) \partial_{{1}}^{2} \partial_{x_{2}} \zeta_{1}-(\mu+\lambda) \partial_{{1}} \partial_{x_{2}} \nabla^{\prime} \cdot \zeta^{\prime}=0.
\end{aligned}
\end{align}
By Cauchy's inequality, one has
\begin{align}\label{eqj52}
\begin{aligned}
\int_{0}^{t}\left\|\partial_{{1}}^{3} \zeta_{2}\right\|^{2} d \tau \leq & C \int_{0}^{t}\left(\left\|\partial_{{1}} \partial_{x_{2}} \phi\right\|^{2}+\left\|\partial_{{1}}^{2} \partial_{x_{2}} \zeta_{1}\right\|^{2}+\left\|\partial_{{1}} \partial_{x^{\prime}}^{2} \zeta\right\|^{2}\right) d \tau \\
& +C \int_{0}^{t}\left(\left\|\partial_{\tau} \partial_{{1}} \zeta_{2}\right\|^{2}+\left\|\nabla \partial_{{1}} \zeta_{2}\right\|^{2}\right) d \tau+C \chi \int_{0}^{t}\left(\left\|\partial_{\tau} \zeta_{2}\right\|^{2}+\left\|\nabla \zeta_{2}\right\|^{2}+\left\|\partial_{x_{2}} \phi\right\|^{2}\right) d \tau,
\end{aligned}
\end{align}
which together with \eqref{eqj50} leads to
\begin{align}\label{eqj53}
\begin{aligned}
& \int_{0}^{t}\left(\left\|\partial_{x_{2}} \partial_{x^{\prime}} \phi\right\|^{2}+\left\|\p_{1}^{2} \partial_{x^{\prime}} \zeta_{2}\right\|^{2}+\left\|\p_{1}^{3} \zeta_{2}\right\|^{2}\right) d \tau \\
\leq & C \int_{0}^{t}\left(\left\|\p_{1} \partial_{x^{\prime}} \phi\right\|^{2}+\left\|\p_{1}^{2} \partial_{x_{2}} \zeta_{1}\right\|^{2}+\left\|\partial_{x^{\prime}} \phi\right\|^{2}\right) d \tau +C \chi \int_{0}^{t}\left(\left\|\partial_{\tau} \zeta_{2}\right\|_{H^{1}}^{2}+\left\|\nabla \zeta_{2}\right\|_{H^{1}}^{2}+\left\|\partial_{x_{2}} \phi\right\|_{H^{1}}^{2}\right) d \tau
 \\
& +C \int_{0}^{t}\left(\left\|\nabla \partial_{\tau} \zeta_{2}\right\|^{2}+\left\|\nabla^{2} \zeta_{2}\right\|^{2}+\left\|\nabla \partial_{x^{\prime}}^{2} \zeta\right\|^{2}+\left\|\left.\partial_{x_{2}} \partial_{x^{\prime}} \zeta_{2}\right|_{x_{1}=0}\right\|_{L^{2}\left(\mathbb{T}^{2}\right)}^{2}\right. \\
& \left.+\left\|\left.\partial_{x^{\prime}} \zeta_{2}\right|_{x_{1}=0}\right\|_{L^{2}\left(\mathbb{T}^{2}\right)}^{2}+\left\|\left.\zeta_{2}\right|_{x_{1}=0}\right\|_{L^{2}\left(\mathbb{T}^{2}\right)}^{2}\right) d \tau \\
\leq & C \int_{0}^{t}\left(\left\|\partial_{x^{\prime}} \phi\right\|^{2}+\left\|\nabla \partial_{\tau} \zeta_{2}\right\|^{2}+\left\|\nabla^{2} \zeta_{2}\right\|^{2}\right) d \tau{+C\left(E^{2}(0)+M^{2}(t)+\delta
\right)},
\end{aligned}
\end{align}
where the last inequality be obtained by using {\eqref{C27}, \eqref{lem2-1}, \eqref{eqj74} and \eqref{eqj122}}. Note that the estimates of $\int_{0}^{t}\left(\left\|\partial_{x_{3}} \partial_{x^{\prime}} \phi\right\|^{2}+\left\|\partial_{{1}}^{2} \partial_{x^{\prime}} \zeta_{3}\right\|^{2}+\left\|\partial_{{1}}^{3} \zeta_{3}\right\|^{2}\right) d \tau$ can be got similarly as \eqref{eqj53}. Therefore, we can derive \eqref{eqj40} and complete the proof of \cref{Lem9}.
\end{proof}

Now we prove Proposition \ref{Cproposition} by combining with the above all lemmas.

Proof of Proposition \ref{Cproposition}: Combining \eqref{eqj0} and \eqref{eqj30} together, and choosing $\eta$ suitably small, one can obtain
$$
\begin{aligned}
& \left\|\p_{1} \phi(t)\right\|^{2}+\int_{0}^{t}\left(\|\nabla \phi\|^{2}+\left\|\p_{1}^{2} \zeta\right\|^{2}\right) d \tau \\
& \leq C\left(M^{2}(t)+E^{2}(0)+\delta\right) +C \int_{0}^{t}\left(\left\|\partial_{\tau} \zeta\right\|^{2}+\|\nabla \zeta\|^{2}+\left\|\nabla \partial_{x^{\prime}} \zeta\right\|^{2}+\left\|\p_{1} \partial_{x^{\prime}} \phi\right\|^{2}+\left\|\left.\zeta^{\prime}\right|_{x_{1}=0}\right\|_{L^{2}\left(\mathbb{T}^{2}\right)}^{2}\right) d \tau,
\end{aligned}
$$
which together with \eqref{lem2-1} and \eqref{lem2-2} leads to 
$$
\begin{aligned}
& \left\|\partial_{t}(\phi, \zeta)(t)\right\|^{2}+\|\nabla \phi(t)\|^{2}+\int_{0}^{t}\left(\left\|\nabla \partial_{\tau} \zeta\right\|^{2}+\|\nabla \phi\|^{2}+\left\|\nabla^{2} \zeta\right\|^{2}\right) d \tau \\
& \quad+\int_{0}^{t}\left(\left\|\left.\partial_{\tau} \zeta^{\prime}\right|_{x_{1}=0}\right\|_{L^{2}\left(\mathbb{T}^{2}\right)}^{2}+\left\|\left.\partial_{x^{\prime}} \zeta^{\prime}\right|_{x_{1}=0}\right\|_{L^{2}\left(\mathbb{T}^{2}\right)}^{2}\right) d \tau \\
& \leq C\left(E^{2}(0)+M^{2}(t)+\delta\right)+C \int_{0}^{t}\left(\left\|\partial_{\tau} \zeta\right\|^{2}+\|\nabla \zeta\|^{2}+\left\|\p_{1} \partial_{x^{\prime}} \phi\right\|^{2}+\left\|\left.\zeta^{\prime}\right|_{x_{1}=0}\right\|_{L^{2}\left(\mathbb{T}^{2}\right)}^{2}\right) d \tau .
\end{aligned}
$$
Then multiplying \eqref{eqj70} by a large positive constant $C$,  combining it with the above inequality 
and choosing $\eta$ suitably small, one has
\begin{align}\label{eqj80}
\begin{aligned}
& \quad\left\|\partial_{t}(\phi, \zeta)(t)\right\|^{2}+\|\nabla(\phi, \zeta)(t)\|^{2}+\int_{0}^{t}\left(\left\|\partial_{\tau} \zeta\right\|_{H^{1}}^{2}+\|\nabla \phi\|^{2}+\left\|\nabla^{2} \zeta\right\|^{2}\right) d \tau \\
& \quad+\left\|\left.\zeta^{\prime}(t)\right|_{x_{1}=0}\right\|_{L^{2}\left(\mathbb{T}^{2}\right)}^{2}+\int_{0}^{t}\left(\left\|\left.\partial_{\tau} \zeta^{\prime}\right|_{x_{1}=0}\right\|_{L^{2}\left(\mathbb{T}^{2}\right)}^{2}+\left\|\left.\partial_{x^{\prime}} \zeta^{\prime}\right|_{x_{1}=0}\right\|_{L^{2}\left(\mathbb{T}^{2}\right)}^{2}\right) d \tau \\
& \leq C\|\phi(t)\|^{2} +C \int_{0}^{t}\left(\left\|\p_{1} \partial_{x^{\prime}} \phi\right\|^{2}+\|\nabla \zeta\|^{2}+\left\|\left.\zeta^{\prime}\right|_{x_{1}=0}\right\|_{L^{2}\left(\mathbb{T}^{2}\right)}^{2}\right) d \tau+C\left(E^{2}(0)+M^{2}(t)+\delta\right) .
\end{aligned}
\end{align}
By \eqref{eq26}, one has
\begin{align}\label{eqj81}
\int_{0}^{t}\left\|\partial_{\tau} \phi\right\|^{2} d \tau \leq C \int_{0}^{t}\|\nabla(\phi, \zeta)\|^{2} d \tau+C \delta \int_{0}^{t}\left\| \left(\phi, \zeta_{1}\right)\right\|^{2} d \tau.
\end{align}
Hence, multiplying \eqref{eqj80} by a large constant $C$ and adding \eqref{eqj81} together, one has
\begin{align}\label{eqj84}
\begin{aligned}
& \left\|\partial_{t}(\phi, \zeta)(t)\right\|^{2}+\|\nabla(\phi, \zeta)(t)\|^{2}+\int_{0}^{t}\left(\left\|\partial_{\tau} \zeta\right\|_{H^{1}}^{2}+\left\|\partial_{\tau} \phi\right\|^{2}+\|\nabla \phi\|^{2}+\left\|\nabla^{2} \zeta\right\|^{2}\right) d \tau \\
& \quad+\left\|\left.\zeta^{\prime}(t)\right|_{x_{1}=0}\right\|_{L^{2}\left(\mathbb{T}^{2}\right)}^{2}+\int_{0}^{t}\left(\left\|\left.\partial_{\tau} \zeta^{\prime}\right|_{x_{1}=0}\right\|_{L^{2}\left(\mathbb{T}^{2}\right)}^{2}+\left\|\left.\partial_{x^{\prime}} \zeta^{\prime}\right|_{x_{1}=0}\right\|_{L^{2}\left(\mathbb{T}^{2}\right)}^{2}\right) d \tau \\
& \leq C\|\phi(t)\|^{2}+C \int_{0}^{t}\left(\left\|\p_{1} \partial_{x^{\prime}} \phi\right\|^{2}+\|\nabla \zeta\|^{2}+\left\|\left.\zeta^{\prime}\right|_{x_{1}=0}\right\|_{L^{2}\left(\mathbb{T}^{2}\right)}^{2}\right) d \tau+C\left(E^{2}(0)+M^{2}(t)+\delta\right) \\
& \leq C\|\phi(t)\|^{2}+C \int_{0}^{t}\left(\|\nabla \zeta\|^{2}+\left\|\left.\zeta^{\prime}\right|_{x_{1}=0}\right\|_{L^{2}\left(\mathbb{T}^{2}\right)}^{2}\right) d \tau+C\left(E^{2}(0)+M^{2}(t)+\delta\right), 
\end{aligned}
\end{align}
where \eqref{eqj122} be used in the last inequality. 

Next, we deal with the higher order derivative estimates. The combination of \eqref{eqj74}, \eqref{eqj11} and \eqref{eqj122} yields

$$
\begin{aligned}
& \left\|\nabla^{2} \phi(t)\right\|^{2}+\int_{0}^{t}\left(\left\|\nabla \p_{1} \phi\right\|^{2}+\left\|\p_{1}^{3} \zeta_{1}\right\|^{2}+\left\|\p_{1}^{2} \partial_{x^{\prime}} \zeta_{1}\right\|^{2}
+\left\|\nabla \partial_{x^{\prime}}^{2} \zeta\right\|^{2}+\left\|\left.\partial_{x^{\prime}}^{2} \zeta^{\prime}\right|_{x_{1}=0}\right\|_{L^{2}\left(\mathbb{T}^{2}\right)}^{2}\right) d \tau \\
& \leq C \int_{0}^{t}\left(\left\|\partial_{\tau} \p_{1} \zeta_{1}\right\|^{2}+\left\|\nabla \p_{1} \zeta_{1}\right\|^{2}+\left\|\p_{1}^{2} \partial_{x^{\prime}} \zeta^{\prime}\right\|^{2}\right) d \tau+C\left(E^{2}(0)+M^{2}(t)+\delta\right).
\end{aligned}
$$
Multiplying \eqref{eqj40} by a large constant $C$ and adding the resulted inequality with the above inequality, one has 
\begin{align}\label{eqj82}
\begin{aligned}
& \left\|\nabla^{2} \phi(t)\right\|^{2}+\int_{0}^{t}\left(\left\|\nabla^{2} \phi\right\|^{2}+\left\|\nabla^{3} \zeta\right\|^{2}+\left\|\left.\partial_{x^{\prime}}^{2} \zeta^{\prime}\right|_{x_{1}=0}\right\|_{L^{2}\left(\mathbb{T}^{2}\right)}^{2}\right) d \tau \\
\leq & \int_{0}^{t}\left(\left\|\nabla \partial_{\tau} \zeta\right\|^{2}+\left\|\nabla^{2} \zeta\right\|^{2}+\left\|\partial_{x^{\prime}} \phi\right\|^{2}\right) d \tau+C\left(E^{2}(0)+M^{2}(t)+\delta \right) .
\end{aligned}
\end{align}
The combination \eqref{eqj77} and \eqref{eqj82} yields
\begin{align}\label{eqj83}
\begin{aligned}
& \left\|\nabla^{2}(\phi, \zeta)(t)\right\|^{2}+\left\|\left.\partial_{x^{\prime}} \zeta^{\prime}\right|_{x_{1}=0}\right\|_{L^{2}\left(\mathbb{T}^{2}\right)}^{2}+\int_{0}^{t}\left(\left\|\nabla^{2} \phi\right\|^{2}+\left\|\nabla^{3} \zeta\right\|^{2}+\left\|\left.\partial_{x^{\prime}}^{2} \zeta^{\prime}\right|_{x_{1}=0}\right\|_{L^{2}\left(\mathbb{T}^{2}\right)}^{2}\right) d \tau \\
\leq & C\left(\left\|\partial_{t} \zeta(t)\right\|^{2}+\|\nabla(\phi, \zeta)(t)\|^{2}+\left\|\left.\zeta^{\prime}\right|_{x_{1}=0}(t)\right\|_{L^{2}\left(\mathbb{T}^{2}\right)}^{2}\right)\\
&\quad+C \int_{0}^{t}\left(\left\|\nabla \partial_{\tau} \zeta\right\|^{2}+\left\|\nabla^{2} \zeta\right\|^{2}+\left\|\partial_{x^{\prime}} \phi\right\|^{2}\right) d \tau+C\left(E^{2}(0)+M^{2}(t)+\delta \right) .
\end{aligned}
\end{align}
By \eqref{eq42}$_{1}$ and \eqref{eqj1}$_{1}$, one has
$$
\int_{0}^{t}\left\|\nabla \partial_{\tau} \phi\right\|^{2} d \tau \leq C \int_{0}^{t}\left\|\nabla^{2}(\phi, \zeta)\right\|^{2} d \tau+C(\chi+\delta ) \int_{0}^{t}\|\nabla(\phi, \zeta)\|_{H^{1}}^{2} d \tau+C \delta \int_{0}^{t}\left\| \left(\phi, \zeta_{1}\right)\right\|^{2} d \tau.
$$
Multiplying \eqref{eqj83} by a large constant $C$, then combining with the above inequality, one has 
\begin{align}\label{eqj85}
\begin{aligned}
& \quad\left\|\nabla^{2}(\phi, \zeta)(t)\right\|^{2}+\left\|\left.\partial_{x^{\prime}} \zeta^{\prime}\right|_{x_{1}=0}(t)\right\|_{L^{2}\left(\mathbb{T}^{2}\right)}^{2}+\int_{0}^{t}\left(\left\|\nabla \partial_{\tau} \phi\right\|^{2}+\left\|\nabla^{2} \phi\right\|^{2}+\left\|\nabla^{3} \zeta\right\|^{2}\right) d \tau \\
& \qquad\qquad+\int_{0}^{t}\left\|\left.\partial_{x^{\prime}}^{2} \zeta^{\prime}\right|_{x_{1}=0}\right\|_{L^{2}\left(\mathbb{T}^{2}\right)}^{2} d \tau \\
& \leq C\left(\left\|\partial_{t} \zeta(t)\right\|^{2}+\|\nabla(\phi, \zeta)(t)\|^{2}+\left\|\left.\zeta^{\prime}\right|_{x_{1}=0}(t)\right\|_{L^{2}\left(\mathbb{T}^{2}\right)}^{2}\right) \\
& \quad+C \int_{0}^{t}\left(\left\|\nabla \partial_{\tau} \zeta\right\|^{2}+\left\|\nabla^{2} \zeta\right\|^{2}+\left\|\partial_{x^{\prime}} \phi\right\|^{2}\right) d \tau+C\left(E^{2}(0)+M^{2}(t)+\delta \right) .
\end{aligned}
\end{align}
Multiplying \eqref{eqj84} by a large constant $C$, and using \eqref{eqj85}, one can obtain
\begin{align}\label{eqj86}
\begin{aligned}
& \left\|\partial_{t}(\phi, \zeta)(t)\right\|^{2}+\|\nabla(\phi, \zeta)(t)\|_{H^{1}}^{2}+\left\|\left.\zeta^{\prime}\right|_{x_{1}=0}(t)\right\|_{L^{2}\left(\mathbb{T}^{2}\right)}^{2}+\left\|\left.\partial_{x^{\prime}} \zeta^{\prime}\right|_{x_{1}=0}(t)\right\|_{L^{2}\left(\mathbb{T}^{2}\right)}^{2} \\
& \quad+\int_{0}^{t}\left(\left\|\left.\partial_{\tau} \zeta^{\prime}\right|_{x_{1}=0}\right\|_{L^{2}\left(\mathbb{T}^{2}\right)}^{2}+\left\|\left.\partial_{x^{\prime}} \zeta^{\prime}\right|_{x_{1}=0}\right\|_{L^{2}\left(\mathbb{T}^{2}\right)}^{2}+\left\|\left.\partial_{x^{\prime}}^{2} \zeta^{\prime}\right|_{x_{1}=0}\right\|_{L^{2}\left(\mathbb{T}^{2}\right)}^{2}\right) d \tau \\
& \quad+\int_{0}^{t}\left(\left\|\partial_{\tau}(\phi, \zeta)\right\|_{H^{1}}^{2}+\|\nabla \phi\|_{H^{1}}^{2}+\left\|\nabla^{2} \zeta\right\|_{H^{1}}^{2}\right) d \tau \\
& \leq C\|\phi(t)\|^{2}+C \int_{0}^{t}\left(\|\nabla \zeta\|^{2}+\left\|\left.\zeta^{\prime}\right|_{x_{1}=0}\right\|_{L^{2}\left(\mathbb{T}^{2}\right)}^{2}\right) d \tau+C\left(E^{2}(0)+M^{2}(t)+\delta\right) .
\end{aligned}
\end{align}
Then multiplying \eqref{C27} by a large constant $C$, then adding it with \eqref{eqj86}, one has
\begin{align}\label{eqj92}
\begin{aligned}
& \left\|\partial_{t}(\phi, \zeta)(t)\right\|^{2}+\|(\phi, \zeta)(t)\|_{H^{2}}^{2}+\left\|\left.\zeta^{\prime}\right|_{x_{1}=0}(t)\right\|_{L^{2}\left(\mathbb{T}^{2}\right)}^{2}+\left\|\left.\partial_{x^{\prime}} \zeta^{\prime}\right|_{x_{1}=0}(t)\right\|_{L^{2}\left(\mathbb{T}^{2}\right)}^{2} \\
& \quad+\int_{0}^{t}\left(
\left\|\left.\zeta^{\prime}\right|_{x_{1}=0}\right\|_{L^{2}\left(\mathbb{T}^{2}\right)}^{2}+\left\|\left.\partial_{\tau} \zeta^{\prime}\right|_{x_{1}=0}\right\|_{L^{2}\left(\mathbb{T}^{2}\right)}^{2}+\left\|\left.\partial_{x^{\prime}} \zeta^{\prime}\right|_{x_{1}=0}\right\|_{L^{2}\left(\mathbb{T}^{2}\right)}^{2}+\left\|\left.\partial_{x^{\prime}}^{2} \zeta^{\prime}\right|_{x_{1}=0}\right\|_{L^{2}\left(\mathbb{T}^{2}\right)}^{2}\right) d \tau\\
&\quad+\int_{0}^{t}\left(\left\|\partial_{\tau}(\phi, \zeta)\right\|_{H^{1}}^{2}+\|( \phi, \nabla\zeta)\|_{H^{2}}^{2}\right) d \tau 
\le C\left(E^{2}(0)+M^{2}(t)+\delta\right).
\end{aligned}
\end{align}
 Finally, combining the above inequality with \eqref{C30} and choosing $\chi+\delta$ suitably small, we can obtain \eqref{eqj90} and complete the proof of 
  Proposition \ref{Cproposition}.

\section{Proof of Theorem \ref{thm1}}\label{sc4}
Now we finish the proof of the main result in Theorem \ref{thm1}.
The global existence result follows directly from Proposition \ref{Cproposition} (A priori estimates) and local existence which can be obtained similarly as in \cite{solo}. To finish the proof of Theorem \ref{thm1}, we only need to justify the time-asymptotic behavior \eqref{eq9}. In fact, according to the estimates \eqref{eqj90}, it yields
$$
\int_{0}^{\infty}\left(\|\nabla(\phi, \zeta)\|^{2}+\left|\frac{d}{d \tau}\|\nabla(\phi, \zeta)\|^{2}\right|\right) d \tau<\infty,
$$
which implies
\begin{align}\label{eqj93}
\lim _{t \rightarrow +\infty}\|\nabla(\phi, \zeta)(t)\|^{2}=0.
\end{align}

By Lemma \ref{Lem-GN}, it holds 
$$
\|(\phi, \zeta)(t)\|_{L^{\infty}}^{2} \leq C\|(\phi, \zeta)(t)\|\left\|\nabla(\phi, \zeta)(t)\right\|+C\|\nabla(\phi, \zeta)(t)\|\left\|\nabla^{2}(\phi, \zeta)(t)\right\|,
$$
which together with \eqref{eqj90} and \eqref{eqj93} yields
\begin{align}\label{eqj94}
\lim _{t \rightarrow +\infty}\|(\phi, \zeta)(t)\|_{L^{\infty}}=0.
\end{align}
Hence, by \eqref{eqj94} and $(iii)$ in Lemma \ref{lem1}, one can obtain 
 \eqref{eq9} and proof of Theorem \ref{thm1} is completed.


\

\textbf{Competing interests} \
The authors do not have any other competing interests to declare.

\

\textbf{Data availability} \
This article does not have any external supporting data.

\

\textbf{Publisher's Note} \ 
Springer Nature remains neutral with regard to jurisdictional claims in published maps and institutional affiliations.

		\vspace{1.5cm}
		
\end{document}